\documentclass[11pt]{article}
\usepackage{xcolor,amsmath,mathrsfs,amssymb,accents,graphicx,epstopdf, amsthm}

\numberwithin{equation}{section}

\newtheorem{theorem}{Theorem}[section]
\newtheorem{lemma}[theorem]{Lemma}
\newtheorem{definition}[theorem]{Definition}
\newtheorem{remark}[theorem]{Remark}
\newtheorem{proposition}[theorem]{Proposition}
\newtheorem{corollary}[theorem]{Corollary}

\newenvironment{Proof}{\removelastskip\par\medskip % inizio e fine dimostrazione
\noindent{\em Proof.}
\rm}{\penalty-20\null\hfill$\square$\par\medbreak}

\newcommand{\sfera}[1]{{\mathbb S}^{#1}}
\newcommand{\Q}{\mathbb{Q}}
\newcommand{\N}{\mathbb{N}}
\newcommand{\R}{\mathbb{R}}

\newcommand{\Leb}[1]{{\mathscr L}^{#1}} % Misura di Lebesgue

\newcommand{\F}{\mathcal{F}}
\renewcommand{\H}{\mathcal{H}}

\newcommand{\var}{\varphi}

\newcommand{\p}{\partial}

\newcommand{\BorelSets}[1]{\mathcal B(#1)}
\newcommand{\Probabilities}[1]{\mathscr P\bigl(#1\bigr)} % misure di probabilita'
\newcommand{\Measures}{\mathscr M} % misure con segno finite
\newcommand{\Measuresp}[1]{\mathscr M_+\bigl(#1\bigr)} % misure positive finite
\renewcommand\div{\operatorname{div}}

\newcommand{\supp}{\operatorname{supp}}

\newcommand{\res}{\mathop{\hbox{\vrule height 7pt width .5pt depth 0pt
\vrule height .5pt width 6pt depth 0pt}}\nolimits}

 \newcommand{\bb}{{\mbox{\boldmath$b$}}}
 \newcommand{\cc}{{\mbox{\boldmath$c$}}}

 \newcommand{\tauV}{{\kern-3pt\tau}}

 \newcommand{\eff}{{\mbox{\scriptsize\rm eff}}}

 \newcommand{\sXX}{{\mbox{\scriptsize\boldmath$X$}}}
  \newcommand{\ssXX}{{\scriptscriptstyle \mbox{\scriptsize\boldmath$X$}}}

 \newcommand{\sYY}{{\mbox{\scriptsize\boldmath$Y$}}}

 \newcommand{\XX}{{\mbox{\boldmath$X$}}}
  \newcommand{\YY}{{\mbox{\boldmath$Y$}}}
 
 \newcommand{\oVVVk}{\overline{\mbox{\boldmath$V$}}\kern-3pt}
 \newcommand{\tVVVk}{\tilde{\mbox{\boldmath$V$}}\kern-3pt}

 \newcommand{\seeta}{{\mbox{\scriptsize\boldmath$\eta$}}}
 \newcommand{\ssigma}{{\mbox{\boldmath$\sigma$}}}

 \newcommand{\eeta}{{\mbox{\boldmath$\eta$}}}
 \newcommand{\eps}{\varepsilon}
 
\newcommand\rncp[1]{{\ring{\R}^{#1}}}

\begin{document}

\title{
On the Lagrangian structure of transport equations:\\ the Vlasov-Poisson system
}
\date{}

\author{Luigi Ambrosio\
   \thanks{Scuola Normale Superiore, Pisa. email: \textsf{luigi.ambrosio@sns.it}}
   \and
   Maria Colombo\
   \thanks{Scuola Normale Superiore, Pisa. email: \textsf{maria.colombo@sns.it}}
 \and
   Alessio Figalli\
   \thanks{University of Texas at Austin. email:
   \textsf{figalli@math.utexas.edu}}
   }

\maketitle

\begin{abstract}
The Vlasov-Poisson system is  a classical model in physics used to describe the evolution of particles under their self-consistent electric or gravitational field.
The existence of classical solutions is limited to dimensions $d\leq 3$ under strong assumptions on the initial data,
while weak solutions are known to exist under milder conditions. However, in the setting of weak solutions
it is unclear whether the Eulerian description provided by the equation physically corresponds to a Lagrangian evolution of the particles.
In this paper we develop several general tools concerning the Lagrangian structure of transport equations with non-smooth vector fields
and we apply these results: (1)  to show that weak solutions of Vlasov-Poisson are Lagrangian; (2) to obtain global existence of weak solutions under minimal assumptions on the initial data.
\end{abstract}

\let\thefootnote\relax\footnote{\newline
\textit{MSC-2010:} 35F25, 35Q83,34A12, 37C10.}

\tableofcontents

\section{Introduction}

The $d$-dimensional Vlasov-Poisson system describes the evolution of a nonnegative distribution function
$f:(0,\infty)\times \R^d \times \R^d \to [0,\infty)$ according to Vlasov's equation, under the action of a
self-consistent force determined by the Poisson's equation:
\begin{equation}
\label{eqn:VP}
\begin{cases}
\partial_t f_t + v\cdot \nabla_x f_t + E_t \cdot \nabla_v f_t = 0 & \qquad \mbox{in }(0,\infty) \times \R^d \times \R^d
\\
\displaystyle{\rho_t(x) = \int_{\R^d} f_t(x,v) \, dv} & \qquad \mbox{in }(0,\infty) \times \R^d
\\
\displaystyle{E_t(x) = \sigma \,c_d  \int_{\R^d} \rho_t(y) \frac{x-y}{|x-y|^d}}\, dy & \qquad \mbox{in }(0,\infty) \times \R^d.
\end{cases}
\end{equation}
Here $f_t(x,v)$ stands for the density of particles having position $x$ and velocity $v$ at time $t$, $\rho_t(x)$ is the distribution of particles in the physical space, $E_t=\sigma\, \nabla (\Delta^{-1}\rho_t)$ is the force field, $c_d>0$ is a dimensional constant
chosen in such a way that $c_d \,{\rm div}\left( \frac{x}{|x|^d}\right)=\delta_0$, and $\sigma \in \{\pm 1\}$.
The case $\sigma=1$ corresponds to the case of electrostatic forces between charged particles with the same sign
(repulsion) while $\sigma=-1$ corresponds to the gravitational case (attraction).

%We use the notation $f_t,\  \rho_t,\ E_t$ to denote the functions
%$f(t,\cdot),\ \rho(t,\cdot),\ E(t,\cdot)$.

This system  appears in several physical models. For instance, when $\sigma=1$  it describes  in plasma physics the evolution of charged particles under their self-consistent electric field, while when $\sigma=-1$ the same system is used in astrophysics to describe the motion of galaxy clusters under the gravitational field. Many different models have been developed in connection with the Vlasov-Poisson equation: amongst others, we mention the relativistic version of \eqref{eqn:VP} (where the velocity of particles is given by $v/\sqrt{1+|v|^2}$) and the Vlasov-Maxwell system (which takes into account both the electric and magnetic fields of the Maxwell equations).

%
%In dimension $d=3$ this system is used to modelize the evolution of a large ensemble of particles subject to their own gravity, under the assumption that both the relativistic effects and the collisions between particles can be neglected. 
%
%
%

\smallskip

Regarding the existence of classical solutions, namely, solutions where all the relevant derivatives exist, the first contributions were given by Iordanskii \cite{iord} in dimension $1$, by Ukai and Okabe \cite{uo} in dimension $2$, and by Bardos and Degond \cite{bard} in dimension $3$ for small data. For symmetric initial data, more existence results have been proven in \cite{batt,woll,hors,scha} (see also the presentation in \cite{rein} for an overview of the topic and the references quoted therein).
Finally, in 1989 Pfaffelm\"oser \cite{pfaf} and Lions and Perthame \cite{lp} were able to prove global existence of classical solutions starting from general data. 
%In \cite{lp} the authors consider an initial datum $f_0 \in L^1 \cap L^\infty( \R^6)$ with finite moments
%$ |v|^m f_0(x,v) \in L^1(\R^6)$ for some $m>3$,
%and, thanks to an a priori estimate on the propagation of moments, they show the existence of a distributional $f \in C((0,\infty); L^p(\R^{6})) \cap L^\infty((0,\infty); L^\infty(\R^{6}))$ for every $1\leq p<\infty$.
Moreover, in \cite{lp} the problem of uniqueness is also addressed:
%under more restrictive assumptions on the initial datum,
there the authors show uniqueness in the class of solutions with bounded space densities in $[0,\infty) \times \R^3$
by considering the Lagrangian flow associated to the vector field $\bb_t(x,v) := (v, E_t(x))$  (see also \cite{loe} for a different proof based on stability in the Wasserstein metric).

\smallskip

The above mentioned results require strong assumptions on the initial data.
However, it would be very desirable to get global existence of solutions under much weaker conditions.
In the classical paper \cite{ars}, Arsen'ev proved global existence of weak solutions 
under the assumption that the initial datum is bounded and has finite kinetic energy (see also \cite{in}).
This result has then been improved in \cite{hh}, where the authors relaxed the boundedness assumption
on an $L^p$ bound for some suitable $p>1$.

Notice that these higher integrability assumptions are needed even to give a meaning to the equation
in the distributional sense:
indeed, when $f_t$ is merely $L^1$ the product $E_tf_t$ does not belong to $L^1_{\rm loc}$
(when $d=3$, for the term $E_tf_t$ to belong to $L^1_{\rm loc}$ one needs to have $f_t \in L^p$ with $p\geq  (12+2\sqrt{5})/11$, see for instance \cite{dpl}).
To overcome this difficulty, in \cite{dpl} the authors considered the concept of renormalized solutions
and obtained global existence in the case $\sigma=1$ under the assumption that the total energy is finite and $f_0\log(1+f_0) \in L^1$
(in the case $\sigma=-1$ they still need some $L^p$ assumption on $f$).
Also, under some suitable integrability assumptions on $f_t$, they can show that the concepts of weak and renormalized solutions are equivalent.

\smallskip

It is important to observe that the Vlasov-Poisson system has a transport structure which allows one to prove that, when the solutions is sufficiently smooth,
$f_t$ is transported along the characteristics of the vector field $\bb_t(x,v) := (v, E_t(x))$.
However, when dealing with weak or renormalized solutions, it is not clear that such a vector field defines a flow on the phase-space
and one loses the relation between the Eulerian and Lagrangian picture.

\smallskip
 
The goal of this paper is twofold: on the one hand we show that the Lagrangian picture is still valid
even for weak/renormalized solutions, and secondly we obtain global existence of weak solutions under minimal assumptions on the initial data.
Both results rely on a combination of the following tools, which we believe have their own interest:\\
(i) the local version of the DiPerna-Lions theory developed in \cite{amcofi};\\
(ii) the uniqueness of bounded compactly supported solutions to the continuity equation
for a special class of vector fields obtained by convolving a singular kernel with a measure (this is based on the techniques developed in \cite{boucrising,bbc-abstract}, see Section \ref{sect:singular integrals});\\
(iii) the fact that the concept of Lagrangian solution is {\em equivalent} to the one of renormalized solution  (see Sections \ref{sect:flow renormalized}
and \ref{sec:repres-flow});\\
(iv) a general superposition principle stating that every nonnegative solution of the continuity equation has a Lagrangian
structure without any regularity or growth assumption on the vector field (see Section \ref{sec:repres-flow}).

The above machinery is needed to prove a general result on the renormalization property for solutions of transport equations which is crucial in our proof.
However, from a PDE viewpoint this renormalization property is all we shall need, so in order to keep the presentation as much as possible independent of  this heavy machinery we shall organize the paper as
follows: in the next section we state our results keeping the presentation on the Lagrangian structure of solutions at an informal level.
Then in Sections~\ref{sect:flowVP} and~\ref{sec:VP-ex} we prove our PDE results without introducing the tools mentioned above but simply using the consequences of them, and we postpone points (i)-(iv) above to Sections~\ref{sect:flow} and~\ref{sec:repres-flow}.

\smallskip
\noindent {\bf Acknowledgement.} The authors are grateful to Anna Bohun, Fran\c cois Bouchut, and Gianluca Crippa for useful discussions on the topic of this paper.
The first and third author acknowledge the support
of the ERC ADG GeMeThNES, the second author has been partially supported by PRIN10 
grant from MIUR for the project Calculus of Variations and by the {Gruppo Nazionale per l'Analisi Matematica, la Probabilit\`a e le loro Applicazioni (GNAMPA)} of the {Istituto Nazionale di Alta Matematica (INdAM)}, the third author has been partially supported 
by NSF Grant DMS-1262411 and NSF Grant DMS-1361122.
This material is also based upon work supported by the National Science
Foundation under Grant No. 0932078 000, while the second and third authors were in residence 
at the Mathematical Sciences Research Institute in Berkeley, California, during the fall semester of 2013.

\section{Statement of the results}
\label{sect:statement}

As already observed in the introduction, the Vlasov-Poisson system has a transport structure: indeed we can rewrite it as 
\begin{equation}
\label{eq:VPtransport}
\partial_t f_t+\bb_t \cdot \nabla_{x,v}f_t=0,
\end{equation}
where the vector field $\bb_t(x,v) = (v, E_t(x)):\R^{2d}\to \R^{2d}$ is divergence-free,
and is coupled to $f_t$ via the relation $E_t = \sigma\, c_d\,  \rho_t \ast ({x}/{|x|^d})$.
Recalling that  $c_d \,{\rm div}\left( \frac{x}{|x|^d}\right)=\delta_0$, the vector field $E_t$ can also be found as 
$E_t = -\nabla_x V_t$ where the potential $V_t:(0,\infty) \times \R^d \to \R$ solves\footnote{This description
is correct in dimension $d \geq 3$ since the fundamental solution of the Laplacian decays at infinity, while in dimension $2$ the function $V_t$ is given by the convolution of $\rho_t$ with $-\frac{1}{2\pi}\log|x|$.}
\begin{equation}
\label{eqn:coupling}
-\Delta V_t= \sigma\, \rho_t \quad \mbox{in }\R^d, 
\qquad
\lim_{|x| \to \infty} V_t(x) = 0.
\end{equation}
Notice that, because the kernel ${x}/{|x|^d}$ is locally integrable, the electric field $E_t$ belongs to $L^1_{\rm loc}(\R^d;\R^d)$, therefore
$\bb_t \in L^1_{\rm loc}(\R^{2d};\R^{2d})$.

Now, since $\bb_t$ is divergence-free, the above equation can be rewritten as 
$$
\partial_t f_t+\div_{x,v} (\bb_t  f_t)=0,
$$
and the equation can be reinterpreted in the distributional sense provided the product $\bb_t f_t$ belongs to $L^1_{\rm loc}$. 
However, as mentioned in the introduction, this is not true if $f_t$ is merely $L^1$. To overcome this difficulty one notices that
if $f_t$ is a smooth solution of \eqref{eq:VPtransport} then also $\beta(f_t)$ is a solution for all $C^1$ functions $\beta:\R\to \R$; indeed
$$
\partial_t \beta(f_t)+\bb_t \cdot \nabla_{x,v}\beta(f_t)=\bigl[\partial_t f_t+\bb_t \cdot \nabla_{x,v}f_t\bigr]\,\beta'(f_t)=0,
$$
or equivalently (since $\div_{x,v} (\bb_t)=0$)
 \begin{equation}
 \label{defn:renorm-sol-VP}
\partial_t \beta(f_t)+{\rm div}_{x,v} (\bb_t \beta(f_t))=0.
 \end{equation} 
Notice that, since  $\beta$ is bounded by assumption, $\beta(f_t) \in L^\infty$ 
so $\bb_t\beta(f_t) \in L^1_{\rm loc}$ whenever $\bb_t \in L^1_{\rm loc}$, and \eqref{defn:renorm-sol-VP}
is well defined in the sense of distribution.
This motivates the introduction of the concept of renormalized solution \cite{dpl}:
\begin{definition}\label{defn:renorm-sol} {Let $\bb \in L^1_{\rm loc}([0,T]\times \R^{2d}; \R^{2d})$ be a Borel vector field. A Borel function~$f \in L^1_{\rm loc}([0,T]\times \R^{2d})$ is a {\em renormalized solution} of \eqref{eq:VPtransport} 
(starting from $f_0$) if \eqref{defn:renorm-sol-VP} holds
in the sense of distributions for every $\beta \in C^1\cap L^\infty(\R)$, namely, for every $\phi \in C^{1}_c ( [0,T) \times \R^{2d})$,
 \begin{multline}
 \label{defn:renorm-sol-VP2}
\int_{\R^{2d}} \phi_0(x,v) \beta (f_0(x,v))\, dx\,dv\\ +
 \int_0^T \! \int_{\R^{2d}} \bigl[\partial_t \phi_t(x,v) + \nabla_{x,v}\phi_t(x,v) \bb_t(x,v) \bigr]\beta(f_t(x,v)) \, dx \,dv \, dt  =0.
\end{multline}
In the case of the Vlasov-Poisson system, a function $f \in L^\infty((0,T);L^1(\R^{2d}))$ is a renormalized solution of \eqref{eqn:VP} 
(starting from $f_0$) if, setting
\begin{equation}
\label{eqn:diciamochie}
\rho_t(x) := \int_{\R^d} f_t(x,v) \, dv,\quad E_t := \sigma \,c_d  \int_{\R^d} \rho_t(y) \,\frac{x-y}{|x-y|^d}\, dy, \quad \bb_t(x,v) := (v, E_t(x)),
\end{equation}
the function $f_t$ solves \eqref{defn:renorm-sol-VP2} with the vector field $\bb_t$ given by \eqref{eqn:diciamochie}.
}
 \end{definition}
 Notice that, in the case of  the Vlasov-Poisson system, the global integrability of $f_t$ is needed to make sense of $\rho_t$ and $E_t$.

\smallskip

This definition takes care of the integrability of the term $E_tf_t$ appearing in the equation. However a second problem
comes when dealing with weak solutions: the vector field $\bb_t$ is not in general Lipschitz, so one cannot use the standard 
Cauchy-Lipschitz theory to construct a flow for such a vector field.
In the seminal paper \cite{lions}, DiPerna and Lions showed that, even for Sobolev vector fields, one can introduce a suitable 
notion of flow (this result has then been extended in several directions, see for instance \cite{ambrosio,crippade,boucrising}).
However this theory requires the a priori assumption that the trajectories of the flow do not blow up in finite time, which
 is expressed in terms of the vector field by the following global hypothesis:
\begin{equation}\label{eqn:globalDP}
\frac{|\bb_t|(x,v)}{1+|x|+|v|}\in L^1\bigl((0,T);L^1(\R^{2d})\bigr)+L^1\bigl((0,T);L^\infty(\R^{2d})\bigr).
\end{equation}
We notice that for Vlasov-Poisson (or more in general for any Hamiltonian system where $\bb_t(x,v)$ is of the form $(v,-\nabla V_t(x))$) the above assumption is satisfied if and only if
$$
\frac{E_t(x)}{1+|x|}=\frac{-\nabla V_t(x)}{1+|x|}\in L^1\bigl((0,T); L^\infty(\R^d;\R^d)\bigr).
$$
Unfortunately this is a very restrictive assumption, as it requires both 
some integrability and moment (in $v$) conditions on $f_t$, so we cannot apply the classical DiPerna-Lions' theory in this context.

In our recent paper  \cite{amcofi} we developed a local version of the DiPerna-Lions' theory under no global assumptions on the vector field,
and this will be a crucial tool for us to give a Lagrangian description of solutions.
More precisely, in Theorem~\ref{thm:repr-foliat} we shall first prove that every bounded  nonnegative solution of a continuity equation
can be always represented as a superposition of mass transported along integral curves of the vector field (notice that a priori these curves may split/intersect). Then, by a modification of the argument in \cite{bbc-abstract} we shall prove that for any vector field of the form $(v,\mu_t\ast x/|x|^d)$, with $\mu_t$ a time-dependent measure,
there is uniqueness of bounded compactly supported solutions of the continuity equation (see Theorem \ref{thm:bc}).
Finally, combining these facts with the theory from \cite{amcofi}, we can show that all bounded/renormalized solutions of Vlasov-Poisson are Lagrangian.

As mentioned before, to express the fact that solutions are Lagrangian we shall need to introduce the concept of Maximal Regular Flow.
Roughly speaking this is a (uniquely defined) incompressible flow on the phase-space composed of integral curves of $\bb_t$
that ``transport'' the density $f_t$ (notice that, since trajectories may blow-up in finite time, mass of $f_t$ can disappear at infinity and/or come from infinity, but it has to follow the integral curves of $\bb_t$).
However, since the definition is rather technical, in order to keep the presentation simpler we shall not introduce now the concept but postpone it to 
Section~\ref{sect:flow}.
This will leave the general reader with the intuitive concept of what is going on, and only the interested readers may decide to enter into the details of the definition and the proofs.

\smallskip

Our first main result shows that bounded or renormalized solutions of Vlasov-Poisson are Lagrangian.
As shown in Theorem \ref{thm:distr-on-integral-curves}, the concept of Lagrangian solutions
is a priori stronger than the one of renormalized solutions as {all Lagrangian solutions
of Vlasov-Poisson are renormalized}, but thanks to our general superposition principle (Theorem~\ref{thm:repr-foliat})
we can prove that the two concepts are actually equivalent.

Here and in the sequel we shall use the notation $L^1_+$ to denote the space of nonnegative integrable functions.
Also, by weakly continuous solutions we shall always mean that the map $t\mapsto \int_{\R^{2d}}f_t\,\var\,dx\,dv$ is continuous for
any $\var \in C_c(\R^{2d})$.

%\begin{theorem}\label{thm:vp-bounded1}
%{\color{blue} Let $T>0$ and $f_t \in L^\infty((0,T); L^1_+(\R^{2d}))$ be a weakly continuous function.
%Assume that:\\
%(i) either $f_t \in L^\infty((0,T); L^\infty(\R^{2d}))$ and $f_t$ is a distributional solution of the Vlasov-Poisson system \eqref{eqn:VP};\\
%(ii) or $f_t$ is a renormalized solution of the Vlasov-Poisson system \eqref{eqn:VP}.\\
%}Then $f_t$ is a Lagrangian solution transported by the Maximal Regular Flow associated to $\bb_t(x,v) = (v, E_t(x))$.
%In particular $f_t$ is renormalized.
%
%Moreover, if $d \leq 4$ and $f_t \in L^\infty((0,T); L^{q}(\R^{2d}))$ with
%\begin{equation}
%\label{eqn:q-choice}
%q=\left\{
%\begin{array}{ll}
%\displaystyle{\frac{23+\sqrt{145}}{24} \approx 1.46 } & \text{if $d=2$},\\
%\displaystyle{\frac{10+\sqrt{37}}{7} \approx 2.30 }& \text{if $d=3$},\\
%\displaystyle{13+3\sqrt{17} \approx 25.37} & \text{if $d=4$},\\
%\end{array}
%\right.
%\end{equation}
%and the kinetic energy is integrable in time, that is
%\begin{equation}
%\label{eqn:finite-en}
%\int_0^T\int_{\R^{2d}} |v|^2 \, f_t(x,v)\,dx\,dv \, dt <\infty,
%\end{equation}
%then the flow is globally defined on $[0,T]$ (i.e., trajectories do not blow-up).
%In particular $f_t$ is the image of $f_0$ through an incompressible flow, hence, for all $\psi:[0,\infty)\to [0,\infty)$ Borel,
%$$
%[0,T]\ni t\mapsto \int_{\R^{2d}}\psi\bigl(f_t(x,v)\bigr)\,dx\,dv
%$$
%is constant in time.
%\end{theorem}
%

{

\begin{theorem}\label{thm:vp-bounded1}
Let $T>0$ and $f_t \in L^\infty((0,T); L^1_+(\R^{2d}))$ be a weakly continuous function.
Assume that:\\
(i) either $f_t \in L^\infty((0,T); L^\infty(\R^{2d}))$ and $f_t$ is a distributional solution of the Vlasov-Poisson equation  \eqref{eqn:VP};\\
(ii) or $f_t$ is a renormalized solution of the Vlasov-Poisson equation  \eqref{eqn:VP} (according to Definition~\ref{defn:renorm-sol}).\\
Then $f_t$ is a Lagrangian solution transported by the Maximal Regular Flow associated to $\bb_t(x,v) = (v, E_t(x))$.
In particular $f_t$ is renormalized.
\end{theorem}

The next corollary provides conditions in dimension $d=2,3,4$ in order to avoid the finite-time blow up of the flow that transports $f_t$. 
%As shown in Corollary \ref{cor:dato-L1} and Remarks \ref{rmk:3d} and \ref{rmk:2d} below, in the repulsive case the bound \eqref{eqn:finite-en-new} is satisfied provided the energy is finite at time $0$,
%while stronger integrability assumptions are needed in the attractive case.
\begin{corollary}\label{cor:vp-bounded1}
Let $d=2,3,4$, fix $T>0$, and let $f_t \in L^\infty((0,T); L^1_+(\R^{2d}))$ be a renormalized solution of the Vlasov-Poisson equation \eqref{eqn:VP} (according to Definition~\ref{defn:renorm-sol}). Assume that both the kinetic energy and the potential energy are integrable in time, that is
\begin{equation}
\label{eqn:finite-en-new}
\int_0^T\int_{\R^{2d}} |v|^2 \, f_t(x,v)\,dx\,dv \, dt + \int_0^T\int_{\R^{d}} |E_t(x)|^2 \,dx \, dt <\infty,
\end{equation}
Then the flow associated to $\bb_t= (v, E_t)$ is globally defined on $[0,T]$
for $f_0$-a.e. $(x,v)$,
$f_t$ is the image of $f_0$ through this flow, and the map
$$
[0,T]\ni t \mapsto \int_{\R^{2d}}\psi\bigl(f_t(x,v)\bigr)\,dx\,dv
$$
is constant in time for all $\psi:[0,\infty)\to[0,\infty)$ Borel.
\end{corollary}
\begin{remark}\label{rmk:finit-ener} {\rm
As can be formally seen performing an integration by parts, the quantity
$$
\int_{\R^{2d}} |v|^2 \, f_t(x,v)\,dx\,dv +\sigma \,\int_{\R^{d}} |E_t(x)|^2 \,dx
$$
coincides with the total energy of the system (i.e., the sum of the kinetic and potential energy), namely
$$
\int_{\R^{2d}} |v|^2 \, f_t\,dx\,dv+\sigma \int_{\R^d}H \ast \rho_t \, \rho_t \, dx,\qquad H(x) := \frac{c_d}{d-2} |x|^{2-d},
$$
see \eqref{eqn:integriamo-per-parti} and Lemma \ref{lemma:force-bound}.
This quantity is formally conserved in time along solutions of the Vlasov-Poisson system; whether this property holds also for distributional/renormalized solutions is an important open problem in the theory. However, since weak solutions are usually built by approximation, a lower semicontinuity argument shows that the energy at time $t$ is controlled from above by the initial energy. 
Hence, when  $\sigma=1$ the validity of \eqref{eqn:finite-en-new} is often guaranteed by the assumption on the initial datum
$$
\int_{\R^{2d}} |v|^2 \, f_0\,dx\,dv+\int_{\R^d}H \ast \rho_0 \, \rho_0 \, dx<\infty,
$$
see Corollary \ref{cor:dato-L1} and Remark \ref{rmk:2d} below. 

Notice that, in the case $\sigma = -1$, a bound on the total energy does not provide in general a control on both the kinetic energy and the potential energy.
Still, one can prove the validity of \eqref{eqn:finite-en-new} under the additional integrability assumptions on $f_0$ (see Remark~\ref{rmk:3d}).
}
\end{remark}

}

Our second result deals with existence of global Lagrangian solutions under minimal assumptions on the initial data.
In this case the sign of $\sigma$ (i.e., whether the potential is attractive or repulsive) plays a crucial role, since in the repulsive case
the total energy controls the kinetic part, while in the attractive case the loss of an a priori bound of the kinetic energy prevents us  
for showing such a result. However we can state a general existence theorem that holds both in the attractive and repulsive case, and 
then show that in the repulsive case it gives us what we want.

The basic idea is the following: when proving existence of solutions by approximation it may happen that, in the approximating sequence, there are some particles that move at higher and higher speed while still remaining localized in a compact set in space (think of a family of particle rotating faster and faster along circles around the origin).
Then, while in the limit these particles will disappear from the phase-space (having infinite velocity), the electric field generated by them
will survive, since they are still in the physical space.
Hence the electric field is not anymore generated by the marginal of $f_t$ in the $v$-variable, instead it is generated by an
``effective density'' $\rho_t^{\eff}(x)$ that may be larger than $\rho_t(x)$. 

So, our strategy will be first to prove global existence of Lagrangian (hence renormalized) solutions for a generalized Vlasov-Poisson system where the electric field 
is generated by $\rho_t^{\eff}$. Then, in the particular case $\sigma=1$, we show that if the initial datum has finite total energy then $\rho_t^{\eff}=\rho_t$ and our solution solves the classical Vlasov-Poisson system.

\smallskip

We begin by introducing the concept of generalized solutions to Vlasov-Poisson.
We use the notation $\Measures_+$ to denote the space of nonnegative measures with finite total mass.

\begin{definition}[Generalized solution of the Vlasov-Poisson equation]\label{defn:gensol}

Given $\overline f \in L^1(\R^{2d})$, let $f_t \in L^\infty((0,\infty); L^1_+(\R^{2d}))$ and $\rho^{\eff}_t \in L^\infty((0,\infty); \Measures_+(\R^d))$.
We say that the couple $(f_t, \rho^{\eff}_t)$
is a (global in time) generalized solution of the Vlasov-Poisson system starting from $\overline f$ if, setting
\begin{equation}
\label{defn:rho}
\rho_t(x) := \int_{\R^d} f_t(x,v) \, dv,\quad E^{\eff}_t := \sigma \,c_d  \int_{\R^d} \rho^\eff_t(y) \frac{x-y}{|x-y|^d}\, dy, \quad \bb_t(x,v) := (v, E^{\eff}_t(x)),
\end{equation}
$f_t$ is a renormalized solution of the continuity equation with vector field $\bb_t$ starting from $\overline f$, 
\begin{equation}
\label{eqn:ineq-rho-rho-eff}
\rho_t \leq \rho^{\eff}_t \qquad \mbox{ as measures for a.e. $t\in (0,\infty)$,}
\end{equation}
and
\begin{equation}
\label{eqn:ineq-rho-rho-eff2}
 |\rho^{\eff}_t|(\R^{d})\leq \|f_0\|_{L^1(\R^{2d})} \qquad \mbox{for a.e. }t\in (0,\infty).
\end{equation}
\end{definition}
Notice that, since $\|\rho_t\|_{L^1(\R^{d})}=\|f_t\|_{L^1(\R^{2d})}$, it follows by \eqref{eqn:ineq-rho-rho-eff} and \eqref{eqn:ineq-rho-rho-eff2} that
whenever the mass of $f_t$ is conserved in time, that is $\|f_t\|_{L^1(\R^{2d})}=\|f_0\|_{L^1(\R^{2d})}$ for a.e. $t \in (0,\infty)$,
then $\rho^{\eff}_t=\rho_t$ and generalized solutions of the Vlasov-Poisson system are just standard renormalized solutions. 

We prove here that generalized solutions of the Vlasov-Poisson equation exist globally for any $L^1$ initial datum, both in the attractive and 
in the repulsive case.

\begin{theorem}\label{thm:dato-mis}
Let us consider $f_0 \in L_+^1(\R^{2d} )$. Then there exists a generalized solution  $(f_t, \rho^{\eff}_t)$ of the Vlasov-Poisson system starting from $f_0$. 
In addition, the map $$[0,\infty) \ni t \mapsto  f_t \in  L^1_{\rm loc}(\R^{2d})$$ is continuous, and the solution $f_t$ is transported by the Maximal Regular Flow associated to $\bb_t(x,v) = (v,E_t^{\eff}(x))$.
\end{theorem}

As observed before, if $\rho^{\eff}_t =  \rho_t $ then $f_t$ is a renormalized solution of the Vlasov-Poisson system.
When $\sigma=1$ (i.e., in the repulsive case) the equality $\rho^{\eff} =  \rho_t$ is satisfied in many cases of interest,
for instance whenever the total initial energy is finite (see Corollary~\ref{cor:dato-L1} below),
or in the case of infinite energy if other weaker conditions are satisfied as it happens in the context of \cite{za} and \cite{lp} (see Remark~\ref{rmk:infiniteen}).

The following result improves the result announced in \cite{dpl}, generalizing their statement to any dimension and with weaker conditions on the initial datum.
As we shall explain in Remark \ref{rmk:2d}, the case $d=2$ is slightly different from $d \geq 3$ because of the slower decay at infinity of the kernel $x/|x|^d$.
For this reason we restrict the next two statements to the case $d \geq 3$, while in Remark \ref{rmk:2d}
we explain how to deal with the case $d=2$.

%
%{\color{red}\begin{corollary}\label{cor:dato-L1}
%Let $d\geq 3$, and let $f_0 \in L^1_+(\R^{2d})$ satisfy
%$$\int_{\R^{2d}} |v|^2 f_0(x,v) \, dx\, dv + \int_{\R^d}H \ast \rho_0 \, \rho_0 \, dx <\infty,\qquad H(x) := \frac{c_d}{d-2} |x|^{2-d}.$$
%Assume that $\sigma=1$.
%Then there exists a global Lagrangian (hence renormalized) solution of the Vlasov-Poisson system \eqref{eqn:VP} with initial datum $f_0$.
%\end{corollary}
%}
{
\begin{corollary}\label{cor:dato-L1}
Let $d\geq 3$, and let $f_0 \in L^1_+(\R^{2d})$ satisfy
$$\int_{\R^{2d}} |v|^2 f_0 \, dx\, dv + \int_{\R^d}H \ast \rho_0 \, \rho_0 \, dx <\infty,\qquad H(x) := \frac{c_d}{d-2} |x|^{2-d}.$$
Assume that $\sigma=1$.
Then there exists a global Lagrangian (hence renormalized) solution $f_t \in C([0,\infty); L^1_{\rm loc}(\R^{2d}))$ of the Vlasov-Poisson system \eqref{eqn:VP} with initial datum $f_0$.\\
Moreover, the following properties hold:
\begin{enumerate}
\item the density $\rho_t$ and the electric field $E_t$ are strongly continuous in $L^1_{\rm loc} (\R^d)$;
\item for every $t \geq 0$, we have the energy bound
\begin{equation}
\label{eqn:bdd-en}
\int_{\R^{2d}} |v|^2 f_t \, dx\, dv + \int_{\R^d} H \ast \rho_t \, \rho_t \, dx \leq \int_{\R^{2d}} |v|^2 f_0 \, dx\, dv + \int_{\R^d} H \ast \rho_0 \, \rho_0 \, dx;
\end{equation}
\item if $d=3,4$ then the flow is globally defined on $[0,\infty)$ (i.e., trajectories do not blow-up) and $f_t$ is the image of $f_0$ through an incompressible flow.
\end{enumerate}
\end{corollary}
}

\begin{remark}\label{rmk:3d}{\rm
When $d=3$ (resp. $d=4$), the above result can be generalized to the attractive case $\sigma=-1$ under the additional assumption $f_0 \in L^{9/7}(\R^6)$ (resp. $f_0 \in L^2(\R^8))$.
Indeed this allows one to prove the the kinetic energy is uniformly bounded in time,
and then by standard interpolation inequalities one obtains that also the potential energy is bounded (see for instance \cite{dpl} or \cite[Remark 8.5]{maria}).
}
\end{remark}

\begin{remark}\label{rmk:2d}{\rm 
In dimension $d=2$, even with an initial datum $f_0\in C^\infty_c(\R^d)$, the electric field $E_0$ cannot belong to $L^2$
(this is due to the fact that the kernel $x/|x|^{d}$ does not belong to $L^2$ at infinity) and therefore the initial potential energy cannot be finite.  
For this reason one needs to slightly modify the equation adding a fixed background density $\rho_b$
satisfying
$$\int_{\R^d} \rho_b(x) \, dx = \int_{\R^d} \rho_0(x) \, dx,$$
giving rise to the following system:
\begin{equation*}
\begin{cases}
\partial_t f_t + v\cdot \nabla_x f_t + E_t \cdot \nabla_v f_t = 0 & \qquad \mbox{in }(0,\infty) \times \R^d \times \R^d
\\
\displaystyle{\rho_t(x) = \int_{\R^d} f_t(x,v) \, dv} & \qquad \mbox{in }(0,\infty) \times \R^d
\\
\displaystyle{E_t(x) = \sigma \,c_d  \int_{\R^d} \bigl(\rho_t(y)-\rho_b(y)\bigr)\, \frac{x-y}{|x-y|^d}}\, dy & \qquad \mbox{in }(0,\infty) \times \R^d,
\end{cases}
\end{equation*}
The presence of $\rho_b$ allows for cancellations in the expression for the $L^2$ norm of $E_0$, which turns out to be finite if $\rho_b$ and $\rho_0$ are sufficiently nice.
In this setting, when $\sigma=1$ one can show that an analogous statement to Corollary~\ref{cor:dato-L1} holds also for $d=2$.
On the other hand, when $\sigma=-1$ one needs to assume that $f_0 \in L\log L(\R^4)$ (compare with Remark \ref{rmk:3d} above).}
\end{remark}

\begin{remark}{\rm
In this paper we restricted ourselves to the Vlasov-Poisson equation but the argument and techniques introduced here generalize
to other equations. For instance, a minor modification of our proofs allows one to obtain the same results in the context of the relativistic Vlasov-Poisson system.}
\end{remark}

The proofs of Theorems \ref{thm:vp-bounded1} and \ref{thm:dato-mis} and Corollary \ref{cor:dato-L1} are given in the next section.

\section{Vlasov-Poisson: Lagrangian solutions and global existence}

{
\subsection{The flow associated to Vlasov-Poisson: proof of Theorem \ref{thm:vp-bounded1} and Corollary~\ref{cor:vp-bounded1}}
\label{sect:flowVP}

\begin{proof}[Proof of Theorem~\ref{thm:vp-bounded1}.]
Notice that the vector field $\bb$ satisfies assumption {\bf (A1)} of Section~\ref{sec:maxflow} and is divergence-free.
Also, by Theorem~\ref{thm:bc} it satisfies assumption {\bf (A2)}. Therefore by Theorem~\ref{thm:repr-foliat} we deduce that $f_t$ (resp. $\beta(f_t)$ with $\beta(s)=\arctan(s)$ if $f_t$ is not bounded but is renormalized) is a Lagrangian solution, and
Theorem \ref{thm:distr-on-integral-curves} ensures in particular that $f_t$ is a renormalized solution.
\end{proof}

\begin{proof}[Proof of Corollary~\ref{cor:vp-bounded1}.]
%We assume that \eqref{eqn:finite-en-new} holds and that 
%$f_t \in L^\infty((0,T); L^{1}(\R^{2d}))$.
Thanks to Theorem~\ref{thm:vp-bounded1} we know that the solution is transported by the maximal regular flow associated to $\bb_t=(v,E_t)$.
Also, since $f_t$ is renormalized, also $g_t:= \frac{2}{\pi}\arctan f_{t}: (0,T) \times \R^d \to [0,1]$ is a solution of  the continuity equation with vector field $\bb$.
Hence, in order to prove that trajectories do not blow up, it is enough to apply the criterion stated in Proposition~\ref{prop:no-blow-up} with $\mu_t=g_t\,dx$, that is
\begin{equation}
\label{eqn:no-blow-up-ass}
\int_0^T\!\!\int_{\R^{2d}}\frac{| \bb_t (x,v)|\,g_t (x,v)}{\bigl(1+(|x|^2+|v|^2)^{1/2}\bigr) \log\bigl(2+(|x|^2+|v|^2)^{1/2}\bigr)} \,dx\,dv \, dt< \infty.
\end{equation}
To this end, we observe that $g_t^2 \leq g_t \leq f_t$, hence
%Integrating \eqref{eqn:finite-en} and \eqref{eqn:finite-potential} with respect to time, % and using that $|v|\leq 1+|v|^2$, 
%we get
\begin{equation*}
\begin{split}
&\int_0^T\!\!\int_{\R^{2d}}\frac{| \bb_t |\,g_t}{\bigl(1+(|x|^2+|v|^2)^{1/2}\bigr) \log\bigl(2+(|x|^2+|v|^2)^{1/2}\bigr)} \,dx\,dv \, dt 
\\&\leq \int_0^T\!\!\int_{\R^{2d}} \, f_t \, dx\, dv\, dt + \int_0^T\!\!\int_{\R^{2d}} |E_t| \, \frac{g_t}{(1+|v|) \log(2+|v|)} \, dx \,dv\, dt 
\\& \leq
 \int_0^T\!\!\int_{\R^{2d}} \, f_t \, dx\, dv\, dt + \int_0^T\!\!\int_{\R^{2d}} \bigg( \frac{|E_t|^2}{(1+|v|)^4\log^2(2+|v|)}+ (1+|v|)^2g_t^2 \bigg) \, dx \,dv\, dt 
\\& \leq
 \Big( \int_{\R^{d}} \frac1 {(1+|v|)^4\log^2(2+|v|)} \, dv\Big)
 \Big(\int_0^T\!\!\int_{\R^{d}} |E_t|^2\, dx \, dt\Big)
 +2 \int_0^T\!\!\int_{\R^{2d}}  (1+|v|)^2f_t \, dx \,dv\, dt .
 \end{split}
\end{equation*}
Also, since $d\leq 4$, 
$$ \int_{\R^{d}} \frac1 {(1+|v|)^4\log^2(2+|v|)} \, dv < \infty,$$
thus \eqref{eqn:no-blow-up-ass} follows from \eqref{eqn:finite-en-new}.

Now, by the no blow-up criterion in Proposition~\ref{prop:no-blow-up} we obtain that the Maximal Regular Flow
$\XX$ of $\bb$ is globally defined on $[0,T]$, namely its trajectories
$\XX(\cdot,x,v)$ belong to $AC([0,T];\R^{2d})$ for $g_0$-a.e. $(x,v)\in\R^{2d}$, and $g_t = \XX(t,\cdot)_\#g_0=g_0\circ \XX(t,\cdot)^{-1}$.
Since $f_t=\tan\left(\frac{\pi}2 g_t\right)$ and the map $[0,1)\ni s \to \tan\left(\frac{\pi}2 s\right) \in [0,\infty)$ is a diffeomorphism,
we obtain that $f_t = \XX(t,\cdot)_\#f_0=f_0\circ \XX(t,\cdot)^{-1}$ as well.
In particular, for all Borel functions $\psi:[0,\infty)\to [0,\infty)$ we have
$$
\int_{\R^{2d}}\psi(f_t)\,dx\,dv=\int_{\R^{2d}}\psi(f_0)\circ \XX(t,\cdot)^{-1}\,dx\,dv=\int_{\R^{2d}}\psi(f_0)\,dx\,dv,
$$
where the second equality follows by the incompressibility of the flow.
\end{proof}

%
%\begin{remark}{\rm
%When  $\sigma=1$ the validity of \eqref{eqn:finite-en} is guaranteed by the assumption
%$$
%\int_{\R^{2d}} |v|^2 \, f_0(x,v)\,dx\,dv+\int_{\R^d}|E_0(x)|^2dx<\infty
%$$
%(see also Corollary \ref{cor:dato-L1} below),
%while when $d=3$ and $\sigma=-1$ one needs
%the additional hypothesis that $f_0 \in L^{9/7}(\R^{6})$ (see \cite[Equation (38)]{dpl3}).
%A similar result could also be given when $d=2$ and $\sigma=-1$, but one would need 
%to slightly change the form of the electric field (see Remark \ref{rmk:2d} below).
%}
%\end{remark}
}

\subsection{Global existence results: proof of Theorem~\ref{thm:dato-mis} and Corollary \ref{cor:dato-L1}}\label{sec:VP-ex}

%In this section we shall prove Theorem~\ref{thm:dato-mis} and Corollary \ref{cor:dato-L1}.

\begin{proof}[Proof of Theorem \ref{thm:dato-mis}]
To prove global existence of generalized Lagrangian solutions of Vlasov-Poisson we shall use an approximation procedure.
Since the argument is rather long and involved, we divide the proof in five steps that we now describe briefly:
In Step 1 we start from approximate solutions $f^n$, obtained by smoothing the initial datum and the kernel, and we decompose them along their level sets. Exploiting the incompressibility of the flow, these functions are still solutions of the continuity equation with the same vector field and, when $n$ varies, they are uniformly bounded. This allows us to take their limit as $n \to \infty$ in Step 2, and show that the limit belongs to $L^1$.
In Step 3 we introduce $\rho^\eff$ as the limit as $n\to \infty$ of the approximate densities $\rho^n$, and we motivate its properties.
In Step 4 we show that the vector fields $E^n$ converge to the vector field obtained by convolving $\rho^\eff$ with the Poisson kernel.
Finally, in Step 5 we combine stability results for continuity equations with the results of Section~\ref{sec:repres-flow} to take the limit in the approximate Vlasov-Poisson equation and show that the limiting solution is transported by the limiting incompressible flow.
We now enter into the details of the proof.

\smallskip

\noindent {\bf Step 1: approximating solutions.} 
Let $K(x) := \sigma\, c_d\,x/|x|^d$ and let us consider approximating kernels $K_n := K \ast \psi_n$, where $\psi_n(x) := n^d\psi(nx)$ and $\psi \in C^\infty_c(\R^d)$ is a standard convolution kernel in $\R^d$.
Let $f^{n}_0 \in C^\infty_c(\R^{2d})$ be a sequence of functions such that
\begin{equation}
\label{initialapprox}
f^{n}_0 \to f_0 \qquad \mbox{in } L^1(\R^{2d}),
\end{equation}
and denote by
$f^{n}_t$ distributional solutions of the Vlasov system with initial datum $f^{n}_0$ and kernel $K_n$ (see for instance \cite{dobby} or \cite{rein} for this classical construction based on a fixed point argument in the Wasserstein metric).
Also, define  $\rho^n_t:=\int f_t^n\,dv$ and $E_t^n:=K_n\ast \rho_t^n$.
Notice that since $K_n$ is smooth and decays at infinity, both $E_t^n$ and $\nabla E_t^n$ are bounded on $[0,\infty) \times \R^d$ (with a bound that depends on $n$).
Hence, since $\bb^n_t:=(v,E_t^n)$ is a Lipschitz divergence-free vector field
its flow  $\XX^n(t):\R^{2d} \to \R^{2d}$ is well defined and incompressible, and by standard theory for the transport equation we obtain that\begin{equation}
\label{eqn:fn-transported}
f^n_t = f^n_0\circ \XX^n(t)^{-1}\qquad \forall\,t \in [0,\infty),
\end{equation}
and
\begin{equation}\label{eqn:l1normsfn}
\|\rho^{n}_t\|_{L^1(\R^{d})}= \|f^{n}_t\|_{L^1(\R^{2d})} =  \|f^{n}_0\|_{L^1(\R^{2d})}\qquad \forall\,t \in [0,\infty).
\end{equation}
Assuming without loss of generality that $\Leb{2d}(\{ f_0 = k\} ) = 0$ for every $k\in \N$ (otherwise we consider as level sets the values $R+k$ in place of $k$ for some $R \in [0,1]$), from \eqref{initialapprox} we deduce that
\begin{equation}
\label{eqn:initial}
f^{n,k}_0:=1_{\{k \leq f_0^n < k+1\}}f^n_0 \to f^k_0 := 1_{\{k \leq f_0 < k+1\}}f_0 \quad \mbox{in } L^1(\R^{2d})\qquad \forall\,k \in \N.
\end{equation}
Now, for any $k,n\in \N$ we consider $f^{n,k}_t := 1_{\{k \leq f^n_t < k+1\}}f^n_t$.
Then it follows by \eqref{eqn:fn-transported} that
\begin{equation}
\label{eqn:repr:fnk}
f^{n,k}_t =  1_{\{k \leq f^n_0\circ \sXX^n(t)^{-1}  < k+1 \}}f^n_0\circ \XX^n(t)^{-1}\qquad \forall\,t \in [0,\infty),
\end{equation}
%is the image of $f^{n,k}_0 := 1_{\{k \leq f^n_0 < k+1\}}f^n_0$ through the flow $\XX^n(t)$,
$f^{n,k}_t$ is a distributional solution of the continuity equation with vector field $\bb^n_t$, and 
\begin{equation}\label{eqn:l1norms}
\|f^{n,k}_t\|_{L^1(\R^{2d})} =  \|f^{n,k}_0\|_{L^1(\R^{2d})}\qquad \forall\,t\in [0,\infty).
\end{equation}

\smallskip

\noindent {\bf Step 2: limit in the phase-space.}
By construction the functions $\{ f^{n,k}\}_{n\in \N}$ are nonnegative and bounded by $k+1$, hence
there exists $f^k \in L^\infty((0,\infty) \times \R^{2d})$ nonnegative such that, up to subsequences,
\begin{equation}
\label{eqn:conv-fnk}
f^{n,k} \rightharpoonup f^k \qquad \mbox{weakly* in } L^\infty((0,\infty) \times \R^{2d}) \;  \mbox{as } n\to\infty\qquad \forall\,k\in \N.
\end{equation}
Moreover, for any $K$ compact subset of $\R^{2d}$ and any bounded function $\phi:(0,\infty)\to [0,\infty)$ with compact support, we can the test function $\phi(t) \,1_{K}(x,v)\, {\rm sign}(f^k_t)(x,v)$ in the previous weak convergence, and thanks to Fatou's Lemma, \eqref{eqn:l1norms}, and \eqref{eqn:initial}, we get
\begin{equation}
\begin{split}
\int_0^\infty \phi(t) \|f^k_t\|_{L^1(K)} \, dt 
&\leq \liminf_{n\to \infty} \int_0^\infty \phi(t) \|f^{n,k}_t\|_{L^1(K)} \, dt 
\\
&\leq \liminf_{n\to \infty} \int_0^\infty \phi(t) \|f^{n,k}_t\|_{L^1(\R^{2d})} \, dt 
\\
&= \liminf_{n\to \infty} \int_0^\infty \phi(t) \|f^{n,k}_0\|_{L^1(\R^{2d})} \, dt 
\\
&=  \Bigl(\int_0^\infty \phi(t)\,dt \Bigr)\, \|f^{k}_0\|_{L^1(\R^{2d})}.
\end{split}
\end{equation}
Since $\phi$ was arbitrary, taking the supremum among all compact subsets $K\subset \R^{2d}$ we obtain
\begin{equation}
\label{eqn:L1boundfk}
\|f^k_t\|_{L^1(\R^{2d})} \leq  \|f^{k}_0\|_{L^1(\R^{2d})}\qquad \mbox{for a.e. }t\in (0,\infty),
\end{equation}
so, in particular, $f^k \in L^\infty((0,\infty); L^1(\R^{2d}))$.

%Moreover, $f^k_t$ is weakly continuous with respect to time since the sequence $f^{n,k}_t$ is weakly continuous uniformly with respect to $n$, being a solution of the continuity equation.

Thanks to \eqref{eqn:L1boundfk} we see that, if we define
\begin{equation}
\label{eqn:f}
f := \sum_{k=0}^ \infty f^{k} \qquad \mbox{in }(0,\infty) \times \R^{2d},
\end{equation}
then
\begin{equation}
\label{eqn:f-L1bound}
\|f_t\|_{L^1(\R^{2d})} \leq \sum_{k=0}^ \infty \|f^{k}_t\|_{L^1(\R^{2d})} \leq \sum_{k=0}^ \infty \|f^{k}_0\|_{L^1(\R^{2d})} =\|f_0\|_{L^1(\R^{2d})} \qquad \mbox{for a.e. $t\in[0,\infty)$,}
\end{equation}
which implies that  $f \in L^\infty ((0,\infty) ; L^1(\R^{2d}))$.

We now claim that
\begin{equation}
\label{eqn:conv-fn} f^{n} \rightharpoonup f \qquad \mbox{weakly in } L^1((0,T) \times \R^{2d})
\end{equation}
for every $T>0$.
%\begin{equation}
%\label{eqn:conv-fn-tested}
%\lim_{n\to \infty}
%\int_0^T \!\!\!\int_{\R^{2d}} \varphi f^n \, dx \,dv \, dt 
%=
%\int_0^T\!\!\!\int_{\R^{2d}} \varphi f \, dx \,dv \, dt .
%\end{equation}
Indeed, 
fix $\varphi \in L^\infty ((0,T) \times \R^{2d})$. Noticing that $f^n=\sum_{k=0}^\infty f^{n,k}$ and $f=\sum_{k=0}^\infty f^{k}$,
by the triangle inequality we have that,  for every $k_0\geq 1$, 
\begin{align*}
\Big|  \int_0^T\!\!\! \int_{\R^{2d}} \varphi\, (f^{n}- f) \, dx \, dv \, dt \Big|
&= 
\Big| \sum_{k=0} ^\infty  \int_0^T\!\!\!\int_{\R^{2d}} \varphi \,(f^{n,k}- f^k) \, dx \, dv \, dt \Big|
\\
&\leq 
\Big| \sum_{k=0} ^{k_0-1} \int_0^T\!\!\!\int_{\R^{2d}} \varphi \,(f^{n,k}- f^k) \, dx\, dv \, dt \Big|
\\&
+ \sum_{k=k_0}^\infty  \int_0^T\!\!\! \int_{\R^{2d}} |\varphi|\, |f^{n,k}| \, dx\, dv \, dt
+ \sum_{k=k_0}^\infty  \int_0^T\!\!\!\int_{\R^{2d}} |\varphi| \,|f^{k}| \, dx\, dv \, dt.
\end{align*}
Using \eqref{eqn:l1norms} and \eqref{eqn:L1boundfk}, the last two terms can be estimated by
\begin{align*}
&\hspace{-3em}\sum_{k=k_0}^\infty  \int_0^T\!\!\! \int_{\R^{2d}} |\varphi| \,|f^{n,k}| \, dx\, dv \, dt
+ \sum_{k=k_0}^\infty  \int_0^T\!\!\!\int_{\R^{2d}} |\varphi| \,|f^{k}| \, dx\, dv \, dt
\\
&\leq 
T\|\varphi\|_{\infty}\sum_{k=k_0}^\infty  \int_{\R^{2d}}|f^{n,k}_0|  \, dx\, dv
+ T\|\varphi\|_{\infty}
\sum_{k=k_0}^\infty \int_{\R^{2d}} |f^k_0| \, dx\, dv
\\
&\leq  T\|\varphi\|_{\infty} \int_{\{ f^n_0\geq k_0\}}|f^{n}_0|  \, dx\, dv
+ T\|\varphi\|_{\infty}
 \int_{\{ f_0\geq k_0\}} |f_0| \, dx\, dv\\
 &=T\|\varphi\|_\infty\Bigl(\|f^{n}_0 1_{\{ f^n_0\geq k_0\}}\|_{L^1( \R^{2d})}+\|f_0 1_{\{ f_0\geq k_0\}}\|_{L^1( \R^{2d})}\Bigr).
\end{align*}
Notice that, thanks to \eqref{eqn:initial} and \eqref{initialapprox}, it follows that
$$
f^{n}_0 1_{\{ f^n_0\geq k_0\}} \to f_0 1_{\{ f_0\geq k_0\}}\qquad \mbox{in } L^1(\R^{2d}),
$$
so by letting $n \to \infty$ and using \eqref{eqn:conv-fnk} we deduce that
\begin{align*}
\limsup_{n\to \infty} \Big|  \int_0^T\!\!\! \int_{\R^{2d}} \varphi\, (f^{n}- f) \, dx \, dt \Big|&
\leq \limsup_{n\to \infty} \Big| \sum_{k=0} ^{k_0-1} \int_0^T\!\!\!\int_{\R^{2d}} \varphi (f^{n,k}- f^k) \, dx\, dv \, dt \Big|\\
&\qquad +2T\|\varphi\|_\infty \|f_0 1_{\{ f_0\geq k_0\}}\|_{L^1( \R^{2d})}\\
&=2T\|\varphi\|_\infty \|f_0 1_{\{ f_0\geq k_0\}}\|_{L^1( \R^{2d})}.
\end{align*}
Hence, letting $k_0 \to \infty$, since $\varphi \in L^\infty$ was arbitrary  we obtain \eqref{eqn:conv-fn}.

\smallskip

\noindent {\bf Step 3: limit of the physical densities.}
Since the sequence $\{\rho^n\}_{n\in \N}$ is bounded in $L^\infty((0,\infty); \Measures_+(\R^d))\subset \bigl[L^1((0,\infty),C_0(\R^d))\bigr]^*$
(see \eqref{eqn:l1normsfn}), there exists $\rho^\eff \in L^\infty((0,\infty); \Measures_+(\R^d))$ such that
\begin{equation}
\label{eqn:conv-densities}
\rho^n \rightharpoonup \rho^{\eff}  \qquad \mbox{weakly* in }L^\infty((0,\infty); \Measures_+(\R^d)).
\end{equation}
Moreover, by  the lower semicontinuity of the norm  under weak* convergence, using \eqref{eqn:l1normsfn} again we deduce that
\begin{equation}
\label{eqn:dens-bound}
{\rm ess} \negthickspace\negthickspace\sup_{t\in (0,\infty)} |\rho^{\eff}_t| (\R^{d}) \leq \lim_{n\to \infty} \Big(\sup_{t\in (0,\infty)} \|\rho^n_t \|_{L^1(\R^d)} \Big) =  \lim_{n\to \infty} \|f^n_0\|_{L^1(\R^{2d})} = \|f_0\|_{L^1(\R^{2d})}.
\end{equation}
Now, let us consider any nonnegative function $\varphi \in C_c((0,\infty) \times \R^d)$. 
By \eqref{eqn:conv-densities} and \eqref{eqn:conv-fn} we obtain that, for any $R>0$,
\begin{equation*}
\begin{split}
\int_{0}^\infty \int_{\R^d} \varphi_t(x) \, d \rho_t^{\eff}(x) \, dt&= \lim_{n\to \infty}\int_0^\infty \int_{\R^d} \rho^{n}_t(x)\, \varphi_t(x) \, dx\, dt
\\
& = \lim_{n\to \infty}\int_0^\infty \int_{\R^{2d}} f^{n}_t(x,v)\, \varphi_t(x) \, dv \, dx\, dt
\\
& \geq \liminf_{n\to \infty}\int_0^\infty \int_{\R^{d}\times B_R} f^{n}_t(x,v)\, \varphi_t(x) \, dv \, dx\, dt
\\
&= \int_0^\infty \int_{\R^{d}\times B_R} f_t(x,v)\, \varphi_t(x) \, dv \, dx\, dt,
\end{split}
\end{equation*}
so by letting $R\to \infty$ we get
$$
\int_{0}^\infty \int_{\R^d} \varphi_t(x) \, d \rho_t^{\eff}(x) \, dt
\geq  \int_0^\infty \int_{\R^{2d}} f_t(x,v)\, \varphi_t(x) \, dv \, dx\, dt=\int_{0}^\infty \int_{\R^d} \varphi_t(x) \, d \rho_t(x) \, dt.
$$
By the arbitrariness of $\varphi$ we deduce that \begin{equation}
\label{eqn:rhoeffgrossa}
\rho_t \leq \rho^{\eff}_t \qquad \mbox{ as measures for a.e. $t\in (0,\infty)$,}
\end{equation}
as desired.

\smallskip

\noindent {\bf Step 4: limit of the vector fields.}
Set $E_t ^{\eff}:= K\ast \rho^{\eff}_t$ and $\bb_t (x,v) := (v,E_t^{\eff}(x))$. We claim that
\begin{equation}
\label{eqn:conv-b}
\bb^n \rightharpoonup \bb \qquad \mbox{weakly in } L^1_{\rm loc}((0,\infty) \times \R^{2d}; \R^{2d})
\end{equation}
and that, for every ball $B_R \subset \R^d$,
 \begin{equation}
 \label{eqn:conv-b-forte-deb-VP}
 [\rho^n_t \ast K_n](x+h)  \to [\rho^n_t \ast K_n](x) \quad \text{as $|h| \to 0$ in $L^1_{\rm loc}((0,\infty); L^1 (B_R))$, uniformly in $n$.}
 \end{equation}

To show this we first prove that the sequence $\{\bb^n\}_{n\in \N}$ is bounded in $L^p_{\rm loc}((0,\infty) \times \R^{2d}; \R^{2d})$ for every $p \in [1, d/(d-1))$. Indeed, using Young's inequality, for every $t \geq 0$, $n\in \N$, and $r>0$,
\begin{equation*}
\begin{split}
& \| \rho^n_t \ast K_n \|_{L^p(B_r)} = \| (\rho^n_t \ast \psi_n)\ast K \|_{L^p(B_r)} 
\\& \leq  \| (\rho^n_t \ast \psi_n) \ast( K1_{B_1}) \|_{L^p(B_r)} +  \| (\rho^n_t \ast \psi_n) \ast( K1_{\R^{d} \setminus B_1}) \|_{L^p(B_r)}
\\& \leq  \| (\rho^n_t \ast \psi_n) \ast( K1_{B_1}) \|_{L^p(\R^d)} + \Leb{d}(B_r)^{1/p}  \| (\rho^n_t \ast \psi_n) \ast( K1_{\R^{d} \setminus B_1}) \|_{L^\infty(\R^d)}
 \\
 & \leq 
 \| \rho^n_t\|_{L^1(\R^d)} \| \psi_n\|_{L^1(\R^d)} \| K \|_{L^p(B_1)} + \Leb{d}(B_r)^{1/p} \| \rho^n_t\|_{L^1(\R^d)} \| \psi_n\|_{L^1(\R^d)} \| K\|_{L^\infty(\R^{d} \setminus B_1)}
\end{split}
\end{equation*}
 hence,  up to subsequences, the sequence $\{\bb^n\}_{n\in \N}$ converges weakly in $L^p_{\rm loc}$. In order to identify the limit we now show that for every $\varphi \in C_c((0,\infty) \times \R^d)$
$$\lim_{n\to \infty}\int_0^\infty \int_{\R^d} \rho^n_t \ast K_n \,\varphi_t \, dx\, dt 
= \int_0^\infty \int_{\R^d} \rho^\eff_t \ast K \,\varphi_t  \, dx\, dt .
$$

Indeed, by standard properties of convolution,
\begin{equation*}
\begin{split}
&\Big| \int_0^\infty \int_{\R^d} \rho^n_t \ast K_n\, \varphi_t \, dx\, dt 
-  \int_0^\infty \int_{\R^d} \rho^\eff_t\ast K\, \varphi_t \, dx\, dt \Big|
\\
&=\Big| \int_0^\infty \int_{\R^d} \rho^n_t \,\varphi_t\ast K_n  \, dx\, dt 
-  \int_0^\infty \int_{\R^d} \rho^\eff_t \,\varphi_t \ast K \, dx\, dt \Big|
\\
&\leq \Big| \int_0^\infty \int_{\R^d} (\rho^n_t- \rho^\eff_t)\, \varphi_t \ast K \, dx\, dt \Big|
+
\Big| \int_0^\infty \int_{\R^d} \rho^n_t\, (\varphi_t \ast K -\varphi_t \ast K \ast \psi_n) \, dx\, dt \Big|
\\
&\leq \Big| \int_0^\infty \int_{\R^d} (\rho^n_t- \rho^\eff_t)\, \varphi_t \ast K \, dx\, dt \Big|
+\Bigl(\sup_{t\in (0,\infty)} \| \rho^n_t\|_{L^1(\R^d)}\Bigr)\,
\|\varphi_t \ast K -\varphi_t \ast K \ast \psi_n \|_{L^\infty((0,\infty) \times \R^d)}.
\end{split}
\end{equation*}
Letting $n \to \infty$, the first term converges to $0$ thanks to the weak convergence \eqref{eqn:conv-densities} of $\rho^n_t$ to $\rho^\eff_t$ and the fact that $\varphi \ast K= \varphi \ast (1_{B_1}K) + \varphi \ast (1_{\R^d \setminus B_1} K)$ is a bounded continuous function, compactly supported in time and decaying at infinity in space. The second term, in turn, converges to $0$ since the first factor is bounded (see \eqref{eqn:dens-bound}) and $\varphi_t \ast K \ast \psi_n$ converges to $\varphi_t \ast K$ uniformly in $(0,\infty)\times \R^d$.

 This computation identifies the weak limit of $ \rho^n_t \ast K_n$ in $L^1_{\rm loc}([0,T] \times \R^{2d})$, showing that it coincides with $ \rho^{\eff}_t \ast K$ and proving \eqref{eqn:conv-b}.
 
 \smallskip
 
 We now prove \eqref{eqn:conv-b-forte-deb-VP}.
First of all, since $K \in W^{\alpha,p}_{\rm loc}(\R^d;\R^d)$ for every $\alpha<1$ and $p< d/(d-1+\alpha)$,\footnote{
This can be seen by a direct computation, using the definition of fractional Sobolev spaces.} using Young's inequality we deduce that, for any $t\in (0,\infty)$,
$$
\|\rho^n_t \ast K_n\|_{W^{\alpha,p}(B_R; \R^d)} =\|(\rho^n_t \ast \psi_n) \ast K\|_{W^{\alpha,p}(B_R; \R^d)} 
\leq C(R)\|\rho^n_t \ast \psi_n\|_{L^{1}(\R^d)}.
$$
Since $\|\psi_n\|_{L^{1}(\R^d)}=1$, thanks to \eqref{eqn:l1normsfn} we deduce that the last term is bounded
independently of $t$ and $n$, that is, for every $R>0$,
\begin{equation}
\label{eqn:fracsob}
\sup_{t\in (0,\infty)}\sup_{n\in \N} \|\rho^n_t \ast K_n\|_{W^{\alpha,p}(B_R; \R^d)} <\infty.
\end{equation}
Hence, by a classical embedding between fractional Sobolev spaces and Nikolsky spaces (see for instance \cite[Lemma 2.3]{KM})
we find that, for $|h| \leq R$,
$$\int_{B_R} |\rho^n_t \ast K_n(x+h) - \rho^n_t \ast K_n(x)|^p \, dx \leq C\bigl(p,\alpha,R,  \|\rho^n_t \ast K_n\|_{W^{\alpha,p}(B_{2R}; \R^d)}\bigr)\,
|h|^{\alpha p},
$$
from which \eqref{eqn:conv-b-forte-deb-VP} follows.

\smallskip

\noindent {\bf Step 5: conclusion.}
Thanks to \eqref{eqn:conv-b} and \eqref{eqn:conv-b-forte-deb-VP}, we can apply the stability result from \cite[Theorem II.7]{lions} (which does not require any growth condition on the vector fields, see also \cite[Proposition 6.5]{amcofi} for the stability of the associated flows)
to deduce that, for every $k \in \N$, $f^k$ is a weakly continuous distributional solution of the continuity equation starting from $f^k_0$,
so by linearity also $F^m:= \sum_{k=1}^m f^k$ is a distributional solution for every $m \in \N$.

Since $F^m$ is bounded, Theorem~\ref{thm:repr-foliat} gives that $F^m$ is a renormalized solution for every $m\in \N$. Letting $m\to \infty$, because $F^m \to f$ strongly in $L^1_{\rm loc}((0,\infty) \times \R^{2d})$ we obtain that $f$ is a renormalized solution of the continuity equation starting from $f_0$ with vector field $\bb$. Together with \eqref{eqn:rhoeffgrossa}, \eqref{eqn:f-L1bound}, and \eqref{eqn:dens-bound}, this proves that $(f_t,\rho^\eff_t)$ is a generalized solution of the Vlasov-Poisson equation starting from $f_0$ according to Definition~\ref{defn:gensol}.

\smallskip

Finally, the fact that $f$ is transported by the Maximal Regular Flow associated to $\bb_t$
simply follows by the fact that each density $f^k$ is transported by Maximal Regular Flow associated to $\bb_t$
(thanks to Theorem~\ref{thm:vp-bounded1}) and that $f=\sum_{k=0}^\infty f^k$ is an absolutely convergent series (see \eqref{eqn:f-L1bound}).
{Also, thanks to Theorem~\ref{thm:distr-on-integral-curves} we deduce that $f_t$ belongs to $C([0,\infty); L^1_{\rm loc}(\R^{2d}))$.}
%
%
% $f^k$ is transported by the Maximal Regular Flow associated to $\bb_t$ by Theorem~\ref{thm:vp-bounded1},
%namely there exists a measure $\eeta^k \in \Measures_+\bigl(C([0,T]; \rncp{2d})\bigr)$ such that $\eeta^k$ is concentrated on the incompressible foliation of $(0,T) \times \R^d$
%and
%$$(e_t)_\# \eeta^k \res \R^{2d}= f^k_t \qquad \mbox{for every }t\in [0,T].$$
%Moreover, $\eeta^k(C([0,T]; \rncp{2d}))= \|f^{k}_0\|_{L^1(\R^{2d})}$.
%Adding over $k$ we deduce that also $f$ satisfies the same property, namely the measure $\eeta \in \Measures\bigl(C([0,T]; \rncp{2d})\bigr)$ defined by
%$$\eeta = \sum_{k=0}^\infty \eeta^k$$
%is concentrated on the incompressible foliation of $(0,T) \times \R^d$
%and
%$$(e_t)_\# \eeta \res \R^{2d}= f_t \qquad \mbox{for every }t\in [0,T].$$
\end{proof}

The proof of Corollary~\ref{cor:dato-L1} follows the lines of the proof of Theorem~\ref{thm:dato-mis}, obtained by approximating both the initial datum and the kernel with a sequence of smooth data with uniformly bounded energy. In turn, this bound ensures that the approximating sequence of phase-space distributions is tight in the $v$ variable uniformly in time, allowing us to show that $\rho^\eff_t =\rho_t$ for a.e.~$t\in(0,\infty)$. The approximation of the initial datum with a smooth sequence having uniformly bounded energy is a technical task that we describe in the next lemma.

\begin{lemma}\label{lemma:approx-initial}
Let $d \geq 3$, let $\psi$ be a standard convolution kernel, and set $\psi_k(x) := k^d\psi(kx)$ for every $k \geq 1$.
Let $f_0 \in L^1(\R^{2d})$ be an initial datum of finite energy, namely
$$
\int_{\R^{2d}} |v|^2 f_0(x,v)\, dx\, dv + \int_{\R^d} [H \ast \rho_0](x)  \, \rho_0(x) \, dx < \infty,
$$
where $\rho_0 (x) := \int_{\R^d} f_0(x,v) \, dv$ and $H(x) := \frac{c_d}{d-2} |x|^{2-d}$ for every $x\in \R^d$.
Then there exist a sequence of functions $\{ f^n_0\}_{n \in \N} \subset C^\infty_c(\R^{2d})$ and a sequence $\{k_n\}_{n\in \N}$ such that $k_n \to \infty$ and, setting $\rho^n_0 (x) = \int_{\R^d} f^n_0(x,v) \, dv$,
\begin{equation}
\label{eqn:init-en-bound}
\lim_{n\to \infty} \Big( \int_{\R^{2d}} |v|^2 f^n_0\, dx\, dv + \int_{\R^d} H \ast \psi_{k_n} \ast \rho^n_0 \, \rho^n_0\, dx \Big)
=\int_{\R^{2d}} |v|^2 f_0\, dx\, dv + \int_{\R^d} H \ast \rho_0 \, \rho_0\, dx.
\end{equation}

\end{lemma}
\begin{Proof} We split the approximation procedure in three steps. Here and in the sequel we use the notation $L^\infty_c$ to denote the space of bounded functions with compact support.

\smallskip

\noindent {\bf Step 1: approximation of the initial datum when $f_0 \in L^\infty_c(\R^{2d})$.}
Assuming that $f_0 \in L^\infty_c(\R^{2d})$,
we claim that there exists $\{ f^n_0\}_{n \in \N} \subset C^\infty_c(\R^{2d})$ such that
\begin{equation}
\label{eqn:en-conv}
\lim_{n\to \infty} \Big( \int_{\R^{2d}} |v|^2 f^n_0\, dx\, dv + \int_{\R^d} H \ast \rho^n_0 \, \rho^n_0\, dx \Big)
=\int_{\R^{2d}} |v|^2 f_0\, dx\, dv + \int_{\R^d} H \ast \rho_0 \, \rho_0\, dx.
\end{equation}

To this end, consider smooth functions $f^n_0$ which converge to $f_0$ pointwise, whose $L^\infty$ norms are bounded by $\|f_0\|_{L^\infty(\R^{2d})}$, and whose supports are all contained in the same ball. By construction the densities $\rho^n_0$ are bounded as well and their supports are also contained in a fixed ball; moreover, the functions $H \ast \rho^n_0$ are bounded and converge to $H \ast \rho_0$ locally in every $L^p_{\rm loc}$. By dominated convergence, these observations show the validity of \eqref{eqn:en-conv}.

\smallskip

\noindent {\bf Step 2: approximation of the initial datum when $f_0 \in L^1(\R^{2d})$.} 
Assuming that $f_0 \in L^1(\R^{2d})$,
we claim that there exists a sequence of functions $\{ f^n_0\}_{n \in \N} \subset C^\infty_c(\R^{2d})$ such that \eqref{eqn:en-conv} holds.

Indeed, by Step 1 it is enough to approximate $f_0$ with a sequence in $L^\infty_c(\R^{2d})$ with converging energies. To this aim, for every $n\in \N$ we define the truncations of $f_0$ given by
$$f^n_0(x,v):= \min\{ n, 1_{B_n}(x,v) f_0(x,v) \} \qquad (x,v) \in \R^{2d}.$$
Since $H \geq 0$ the integrands in the left-hand side of \eqref{eqn:en-conv} converge monotonically, hence the integrals converge by monotone convergence.

\smallskip

\noindent {\bf Step 3: approximation of the kernel.} 
We conclude the proof of the lemma.
In order to approximate the kernel, we notice that, given the sequence of functions $f^n_0 \in C^\infty_c(\R^d)$ provided by Steps 1-2, for $n\in \N$ fixed we have
$$\lim_{k\to \infty} \int_{\R^d} H \ast \psi_{k} \ast \rho^n_0 \, \rho^n_0\, dx
=\int_{\R^d} H \ast \rho^n_0 \, \rho^n_0\, dx.
$$
Hence, choosing $k_n$ sufficiently large so that
$$ \Big| \int_{\R^d} H \ast \psi_{k_n} \ast \rho^n_0 \, \rho^n_0\, dx
-\int_{\R^d} H \ast \rho^n_0 \, \rho^n_0\, dx\Big| \leq \frac{1}{n},$$
we conclude the proof of the approximation lemma.
\end{Proof}

We now begin the proof of Corollary~\ref{cor:dato-L1}.
We first prove the existence of renormalized solutions, while we postpone the proof of properties (i)-(ii)-(iii) to the end of the section,
as we shall first need some few other preliminary estimates
\begin{proof}[Proof of Corollary~\ref{cor:dato-L1}: {existence of renormalized solutions}]
Given $f_0$ with finite energy, let $\{ f^n_0\}_{n \in \N} \subset C^\infty_c(\R^{2d})$ and $\{k_n\}_{n\in \N}$ be as in Lemma~\ref{lemma:approx-initial}.
Also let $K:=c_d \,x/|x|^d$ and $K_n := K \ast \psi_{k_n}$.
Applying verbatim the arguments in Steps 1-3 in the proof of Theorem~\ref{thm:dato-mis} we get a sequence $f_n$ of smooth solutions with kernels $K_n$ such that
\begin{equation}
\label{eqn:weak-conv-fn-coroll}
f^{n} \rightharpoonup f \qquad \mbox{weakly in } L^1([0,T] \times \R^{2d})\quad \text{for any $T>0$},
\end{equation}
and
$$
\rho^n \rightharpoonup \rho^{\eff}  \qquad \mbox{weakly* in }L^\infty((0,T); \Measures_+(\R^d)),
$$
where $\rho^n_t(x):=\int_{\R^d} f^n_t(x,v) \, dv$.
In addition, the conservation of the energy along classical solutions gives that, for every $n\in \N$ and $t\in [0,\infty)$
\begin{equation}
\label{eqn:en-approx} \int_{\R^{2d}} |v|^2 f^n_t\, dx\, dv + \int_{\R^d} H \ast \psi_{k_n} \ast \rho^n_t \, \rho^n_t\, dx =  \int_{\R^{2d}} |v|^2 f^n_0\, dx\, dv + \int_{\R^d} H \ast \psi_{k_n} \ast \rho^n_0 \, \rho^n_0\, dx \leq C,
\end{equation}
Hence, since $H \geq 0$ we deduce that
\begin{equation}
\label{eqn:en-down}\sup_{n\in \N} \sup_{t\in [0,\infty)} \int_{\R^{2d}} |v|^2 f^n_t\, dx\, dv \leq C,
\end{equation}
and by lower semicontinuity of the kinetic energy we deduce that, for every $T>0$,
\begin{equation}
\label{eqn:lsc-kinetic}
\int_0^T \int_{\R^{2d}} |v|^2 f_t\, dx\, dv \, dt \leq \liminf_{n\to \infty} \int_0^T
  \int_{\R^{2d}} |v|^2 f^n_t\, dx\, dv \, dt \leq C\,T.
\end{equation}
We now want to exploit \eqref{eqn:en-down} and \eqref{eqn:lsc-kinetic} to show that $\rho^\eff = \rho$, where $\rho_t (x) := \int_{\R^d} f_t(x,v) \, dv\in L^\infty((0,T); L^1(\R^d))$. For this, we want to show that for any $\varphi \in C_c((0,\infty) \times \R^{d})$
\begin{equation}
\label{eqn:conv-rhon}
\lim_{n\to \infty}
\int_0^\infty\int_{\R^{d}} \varphi \,\rho^n_t \, dx \, dt
=
\int_0^\infty\int_{\R^{d}} \varphi \,\rho_t \, dx \, dt.
\end{equation}
To this aim, for every $k\in \N$ we consider a continuous nonnegative function $\zeta_k:\R^d \to [0,1]$ which equals $1$ inside $B_k$ and $0$ outside $B_{k+1}$, and observe that
\begin{align*}
\int_0^\infty\!\!\!\int_{\R^{d}} \varphi\, (\rho^n_t-\rho_t) \, dx \, dt 
=&
 \int_0^\infty\!\!\!\int_{\R^{2d}} \varphi_t(x) \,f^n_t(x,v) \,(1-\zeta_k(v)) \, dx \, dv \, dt
 \\
 &+  \int_0^\infty\!\!\!\int_{\R^{2d}} \varphi_t(x)\, (f^n_t(x,v)- f(x,v))\, \zeta_k(v) \, dx \, dv \, dt
\\&+ \int_0^\infty\!\!\!\int_{\R^{2d}} \varphi_t(x) \,f_t(x,v) \,(\zeta_k(v)-1) \, dx \, dv \, dt.
\end{align*}
The second term in the right-hand side converges to $0$ by the weak convergence of $f^n$ to $f$ in $L^1$,
while, thanks to \eqref{eqn:en-down} and \eqref{eqn:lsc-kinetic}, the other two terms are estimated as
$$
\Big| \int_0^\infty\!\!\!\int_{\R^{2d}} \varphi \,f^n_t(x,v) \,(1-\zeta_k(v)) \, dx \, dv \, dt \Big| 
\leq
\frac{\| \varphi\|_{\infty}}{k^2} \int_0^T\!\!\!\int_{\R^{2d}}  f^n_t(x,v) |v|^2  \, dx \, dv \, dt  
\leq
\frac{C\,T\| \varphi\|_{\infty}}{k^2},
$$
and
$$
\Big| \int_0^\infty\!\!\!\int_{\R^{2d}} \varphi\, f_t(x,v) \,(1-\zeta_k(v)) \, dx \, dv \, dt \Big| 
\leq
\frac{C\,T\| \varphi\|_{\infty}}{k^2}.
$$
Letting $k \to \infty$, this proves \eqref{eqn:conv-rhon}.
Thanks to this fact, the conclusion of the proof proceeds exactly as in Steps 4 and 5 in the proof of Theorem~\ref{thm:dato-mis} with $\rho^\eff_t= \rho_t$.
\end{proof}

\begin{remark}\label{rmk:infiniteen}{\rm 
As shown in Corollary~\ref{cor:dato-L1},
the construction from Theorem~\ref{thm:dato-mis} provides distributional solutions of the Vlasov-Poisson system if further assumptions are made on the initial datum, such as finiteness of the total energy.
Still, there are examples of infinite energy data such that the generalized solution built in Theorem~\ref{thm:dato-mis} is in fact distributional.
For instance,
in~\cite{pe} Perthame considers an initial datum $f_0 \in L^1 \cap L^\infty(\R^6)$ with $(1+|x|^2) f_0 \in L^1(\R^6)$ and infinite energy, and he shows the existence of a solution $f \in L^\infty( [0,\infty); L^1 \cap L^\infty(\R^6))$
 of the Vlasov-Poisson system such that the quantities
\begin{equation}
\label{eqn:pert}
 t^{1/2} \|E_t\|_{L^2}, \qquad t^{3/5} \| \rho_t\|_{L^{5/3}}, \qquad \int_{\R^6} \frac{ |x-vt|^2}{t} f_t(x,v) \, dx \, dv
\end{equation}
are bounded for all $t \in (0,\infty)$.
It can be easily seen that, under Perthame's assumptions, the construction in the proof of Theorem~\ref{thm:dato-mis} provides a solution of the Vlasov-Poisson equation as the one built in~\cite{pe}.
In particular, thanks to the a priori estimate \eqref{eqn:pert} on the approximating sequence, it is easy to see that $\rho^\eff = \rho$, therefore providing a Lagrangian (and therefore renormalized and distributional) solution of Vlasov-Poisson.

Analogously, under the assumptions of \cite{za}, a similar argument shows that the generalized solutions built in Theorem~\ref{thm:dato-mis} solve the classical
Vlasov-Poisson system.}
\end{remark}

{
Our goal now is to prove the validity of properties (i)-(ii)-(iii) in Corollary \ref{cor:dato-L1}.
As we shall see, the proof of the energy inequality \eqref{eqn:bdd-en} is based on the conservation of energy along approximate solutions and on a lower semicontinuity argument.
Notice that, since $-\Delta H=\delta_0$, a formal integration by parts (rigorously justified in the case that $\mu$ has smooth, compactly supported density with respect to the Lebesgue measure) shows that, for every $ \mu \in \Measures_+(\R^d)$,
\begin{equation}
\label{eqn:integriamo-per-parti}
\int_{\R^d} H \ast \mu(x) \, d \mu (x) = \int_{\R^d} |\nabla H \ast \mu(x)|^2 \, d x,
\end{equation}
meaning that, if one of the two sides is finite, than so is the other and they coincide.
The above identity would immediately imply the convexity of the potential energy and its lower semicontinuity with respect to the weak* convergence of measures. However, since the justification of \eqref{eqn:integriamo-per-parti} requires some work, we shall prove directly the lower semicontinuity.

\begin{lemma}Let  $d \geq 3$ and  $H(x) := \frac{c_d}{d-2} |x|^{2-d}$, with the convention $H(0)=+\infty$. Then the functional
$$\F(\mu):= \int_{\R^d} H \ast \mu(x) \, d \mu (x), \qquad \mu \in \Measures_+(\R^d),$$
is %convex in $L^1(\R^d)$ and
lower semicontinuous with respect to the weak* topology of $\Measures(\R^d)$.
\end{lemma}
\begin{proof}
%
%
%We observe that the functional is lower semicontinuous with respect to the strong $L^1$ topology. \todo{Indeed,}
%
%The lower semicontinuity with respect to the weak topology follows then from the fact that each sublevel set of $\F$ is strongly closed and 
%
%A different proof of the lower semicontinuity with respect to the weak topology is obtained by noticing that, g
%\noindent {\bf Step 1: Lower semicontinuity}. 
Given a sequence of nonnegative measures $\mu^n$ weakly* converging to $\mu$ in $\Measures(\R^d)$, the measures $d\mu^n(x)\,d \mu ^n(y)\in \Measures(\R^{2d})$ weakly* converge to $d\mu(x)\,d\mu(y)$. Hence, since the function $\hat H(x,y) :=H(x-y)$ is continuous as a map from $\R^{2d}$ to $[0,+\infty]$, we deduce that
$$ 
 \int_{\R^{d}}\int_{\R^{d}} H(x-y)\, d\mu(x)\,  d\mu(y) \leq \liminf_{n\to \infty}\int_{\R^{d}}\int_{\R^{d}} H(x-y) \, d\mu^n(x)\,  d\mu^n(y).$$
\end{proof}

The following lemma adapts the previous one to the time-dependent framework. In particular it takes care of a further approximation of the kernel in the right-hand side of \eqref{eqn:en-integrated} below and involves the time dependence of the functional.
 We need this kind of lemma since, at the level of generality of Theorem~\ref{thm:dato-mis}, the weak convergence of the approximating solutions is not pointwise in time, but it happens only as functions in space-time. 
\begin{lemma}\label{lemma:lsc-en-int-in-t} Let  $d \geq 3$, $T>0$, $\phi \in C_c((0,T))$ nonnegative, let $\psi\geq 0$ be a convolution kernel, and let $\psi_n(x) := n^d\psi(nx)$ for every $n \geq 1$. 
Then, for every sequence $\{ \rho^n\}_{n\in \N} \subseteq L^\infty((0,T); \Measures_+(\R^d))$ converging
weakly* in $L^\infty((0,T); \Measures_+(\R^d))$ to $\rho \in L^\infty((0,T); \Measures_+(\R^d))$, 
we have
\begin{equation}
\label{eqn:en-integrated}
\int_0^T \phi(t) 
\int_{\R^d} H \ast \rho_t(x) \, d \rho_t (x) \,dt \leq \liminf_{n\to \infty}\int_0^T \phi(t) \int_{\R^d} H \ast \psi_n\ast \rho^n_t(x) \, d \rho^n_t (x)\,dt.
\end{equation}

\end{lemma}
\begin{proof} Since $\psi_n\ast \rho^n_t \, dt $ weakly* converges to $ \rho_t \, dt $ in $\Measures((0,T) \times \R^d)$, the sequence of nonnegative measures $\psi_n\ast\rho^n_t(x)\, \rho ^n_t(y) \,dt \in \Measures((0,T) \times \R^{2d})$ weakly* converges to $\rho_t(x) \,\rho _t(y)\, dt$. Hence, since the function $\phi(t) H(x-y)$ is continuous as a map from $(0,T)\times\R^{2d}$ to $[0,+\infty]$, we get
%\begin{equation*}
%\begin{split}
%&\hspace{-2em}\int_0^T  \int_{\R^{d}}\int_{\R^{d}} \phi(t) \,H(x-y)\, d\rho(x)\,  d\rho(y) \,dt
%\\& \leq \liminf_{n\to \infty}\int_0^T \phi(t) \int_{\R^{d}}\int_{\R^{d}} H(x-y) \, d(\psi_n\ast\rho^n)(x)\,  d\rho^n(y)\,dt,
%\end{split}
%\end{equation*}
%which proves
that \eqref{eqn:en-integrated} holds. 
\end{proof}

In the following lemma we establish a general inequality between the potential energy and the $L^2$-norm of the force field, that will be used to show the property (iii) in Corollary~\ref{cor:dato-L1}.%, namely that, in dimension $d=3$, for every renormalized solution $f_t$ of finite energy the trajectories of the flow of $\bb_t(x,v) =(v, E_t(x))$ do not blow up in finite time.
\begin{lemma}\label{lemma:force-bound}
Let  $d \geq 3$ and $H(x) := \frac{c_d}{d-2} |x|^{2-d}$. Then, for every $\rho \in L^1(\R^d)$ nonnegative,
\begin{equation}
\label{eqn:energyineq}
\int_{\R^d} H \ast \rho \, \rho \,dx \geq \int_{\R^d} |\nabla H \ast \rho|^2 \, d x.
\end{equation}% \todo{serve anche con convoluzione con $\psi_n$.}
\end{lemma}
\begin{proof}We split the approximation procedure in three steps. 
\smallskip

\noindent {\bf Step 1: Proof of equality in \eqref{eqn:energyineq} for $\rho \in L^\infty_c(\R^d)$.} 
Consider first $\rho$ a smooth, compactly supported function. For every $R>0$, the integration by parts formula gives
$$\int_{B_R} H \ast \rho \, \rho\, dx  = \int_{B_R} |\nabla H \ast \rho|^2 \, d x - \int_{\partial B_R}  H \ast \rho \,  \nabla (H \ast \rho) \cdot \nu_{B_R} \, d\H^{d-1}.$$
By approximation, the same identity holds when $\rho$ is bounded and compactly supported.
Now, since $H \ast \rho$ and $\nabla H \ast \rho$ respectively decay as $R^{2-d}$ and $R^{1-d}$ when evaluated on $\partial B_R$,
we see that the boundary term in the previous equality disappears as $R\to \infty$ (recall that $d \geq 3$).
This proves that equality holds in \eqref{eqn:energyineq} for $\rho \in L^\infty_c(\R^d)$

\smallskip

\noindent {\bf Step 2: Proof of \eqref{eqn:energyineq} for $\rho \in L^1(\R^d)$.} Given $\rho \in L^1(\R^d)$, 
for every $n\in \N$ consider the truncations of $\rho$ given by $\rho^n:= \min\{ n, 1_{B_n} \rho\}$.
Since $H \geq 0$, it follows by monotone convergence and Step 1 that
$$\int_{\R^d} H \ast \rho  \,\rho\, d x =\lim_{n\to \infty} \int_{\R^d} H \ast \rho^n \, \rho^n\, d x 
\geq \lim_{n\to \infty}  \int_{\R^d} |\nabla H \ast \rho^n|^2 \, d x.$$
Assuming without loss of generality that the left hand side is finite, we
see that the sequence $\{ \nabla H \ast \rho^n \}_{n\in \N}$ is bounded in $L^2$.
Hence, since its limit in the sense of distribution is $ \nabla H \ast \rho$, the lower semicontinuity of the $L^2$-norm with respect to weak convergence implies 
that $ \nabla H \ast \rho \in L^2(\R^d)$ and \eqref{eqn:energyineq} holds.
\end{proof}
%
%We are left to show that $f_t$ has bounded energy for every $t\geq 0$ and that \eqref{eqn:bdd-en} holds. Since the convergence of $f^n$ to $f$ and of $\rho^n$ to $\rho$ are not pointwise in time, we proceed in two steps. First, by a lower semicontinuity argument in space-time, we prove that \eqref{eqn:bdd-en} holds for $\Leb{1}$-a.e. $t\in [0,\infty)$; then, by employing the continuity of $f_t$ in $L^1_{\rm loc}(\R^{2d})$ with respect to $t$ we extend the energy bound to all times.

\begin{proof}[Proof of Corollary~\ref{cor:dato-L1}: proof of properties (i)-(ii)-(iii)] In order to prove the desired properties (in particular \eqref{eqn:bdd-en}) we perform a lower semicontinuity argument on the energy of the approximate solutions $f^n$ constructed in the first part of the proof of Corollary~\ref{cor:dato-L1}.
%
%Since the convergence of $f^n$ to $f$ is only in space-time, in Step 1 below we shall first obtain the energy inequality integrated in time to deduce that \eqref{eqn:bdd-en} holds for $\Leb{1}$-a.e. $t\in[0,\infty)$. Then, in Steps 2, 3, and 4, we employ the bound on the kinetic energy to prove the strong $L^1_{\rm loc}$ continuity of $\rho_t$ and $E_t$ in time. We remark that this does not imply, by itself, the conservation of mass of $\rho_t$ in time, since we don't have any information on the compactness of $f_t$ in the $x$ variable, but only in $v$ (see Remark~\ref{remark:cons-mass} below). In Step 5, we use again the lower semicontinuity of the energy to deduce that the energy inequality \eqref{eqn:bdd-en} holds for every $t\in [0,\infty)$.
%Finally, in Step 6 we show the existence of a global measure-preserving flow associated to our solution $f_t$; this implies in particular the conservation of mass, namely that $\|f_t\|_{L^1(\R^{2d})}=\rho_t(\R^d) = \rho_0(\R^d)=\|f_0\|_{L^1(\R^{2d})}$ for every $t\in [0,\infty)$.

\smallskip

\noindent {\bf Step 1: bound on the total energy for $\Leb{1}$-almost every time.} 
Consider a nonnegative function $\phi \in C_{c}((0,\infty))$. Testing the weak convergence \eqref{eqn:weak-conv-fn-coroll} with $\phi(t)\,|v|^2 \chi_r(x,v)$ where $\chi_r\in C^\infty_c(\R^{2d})$ is a nonnegative cutoff function between $B_r$ and $B_{r+1}$, we find that, for every $r>0$,
\begin{equation*}
\begin{split}
\int_0^\infty \int_{\R^{2d}}\phi(t)\, |v|^2 \chi_r(x,v) \,f_t\, dx\, dv \, dt &= \lim_{n\to \infty} \int_0^\infty
  \int_{\R^{2d}}  \phi(t) \,|v|^2 \chi_r(x,v)\, f^n_t\, dx\, dv \, dt 
\\&\leq \liminf_{n\to \infty} \int_0^\infty \phi(t)
  \int_{\R^{2d}} |v|^2 f^n_t\, dx\, dv \, dt .
\end{split}
\end{equation*}
Taking the supremum with respect to $r$, we deduce that
\begin{equation}
\label{eqn:lsc-kinetic-integr}
\int_0^\infty \phi(t) \int_{\R^{2d}} |v|^2 f_t\, dx\, dv \, dt \leq \liminf_{n\to \infty} \int_0^\infty \phi(t)
  \int_{\R^{2d}} |v|^2 f^n_t\, dx\, dv \, dt .
\end{equation}
As regards the potential energy, %let us consider the functional
%$$\F_{\phi} (\sigma) = \int_0^\infty \phi(t) \int_{\R^d} H \ast \sigma_t \, \sigma_t\, dx \, dt \qquad \mbox{for every }\sigma \in L^1((0,T)\times \R^d).
%$$
it follows from Lemma~\ref{lemma:lsc-en-int-in-t} %(applied to each time slice) and the fact that $\phi$ is nonnegative we deduce that $\F_{\phi}$ is convex and lower semicontinuous with respect to strong convergence. \todo{Indeed,...}Hence, it is lower semicontinuous with respect to the weak $L^1$-convergence; by ... 
that %\todo{ATTENZIONE QUI NEL RHS C'E' PURE $\psi$...}
\begin{equation}
\label{eqn:lsc-potential-integr}
\int_0^\infty \phi(t) \int_{\R^d} H \ast \rho_t \, \rho_t\, dx \, dt
 \leq \liminf_{n\to \infty}
 \int_0^\infty \phi(t) \int_{\R^d} H \ast \psi_{k_n} \ast \rho^n_t \, \rho^n_t\, dx \, dt
\end{equation}
Adding \eqref{eqn:lsc-kinetic-integr} and \eqref{eqn:lsc-potential-integr}, by the subadditivity of the $\liminf$ and by the energy bound  \eqref{eqn:en-approx} 
on the approximating solutions, we find that
\begin{equation*}
\begin{split}
&\int_0^\infty \phi(t) \Big( \int_{\R^{2d}} |v|^2 f_t\, dx\, dv +
\int_{\R^d} H \ast \rho_t \, \rho_t\, dx \Big) \, dt
\\&
\leq
%\liminf_{n\to \infty} \int_0^\infty \phi(t)
%\Big( \int_{\R^{2d}} |v|^2 f^n_t\, dx\, dv
%+ \int_{\R^d} H \ast \psi_{k_n} \ast \rho^n_t \, \rho^n_t\, dx \Big)\, dt
%\\&
 \lim_{n\to \infty} \int_0^\infty \phi(t)
\Big( \int_{\R^{2d}} |v|^2 f^n_0\, dx\, dv
+ \int_{\R^d} H \ast \psi_{k_n} \ast \rho^n_0 \, \rho^n_0\, dx \Big)\, dt
\\&
=\Big(\int_0^\infty \phi(t) \, dt \Big) \Big( \int_{\R^{2d}} |v|^2 f_0 \, dx\, dv + \int_{\R^d} H \ast \rho_0 \, \rho_0 \, dx\Big).
\end{split}
\end{equation*}
By the arbitrariness of $\phi$ it follows that \eqref{eqn:bdd-en} holds for $\Leb{1}$-a.e. $t\in (0,\infty)$
and that $|v|^2f_t \in L^1_{\rm loc}((0,\infty)\times \R^{2d})$. In particular this allows us to integrate the transport equation $\p_t f_t+\div_{x,v}(\bb_t f_t)=0$
with respect to $v$ on the whole $\R^d$ and obtain
$$
\partial_t \rho_t + \div_x(J_t)=0,\qquad J_t(x):=\int_{\R^d}v\,f_t(x,v)\,dv \in L^1_{\rm loc}((0,\infty)\times \R^{d}).
$$
By classical results on continuity equations, this implies that $\rho_t$ is weakly* continuous in time (see for instance \cite[Lemma 8.1.2]{amgisa}).

\smallskip
\noindent {\bf Step 2: boundedness of the total energy for every time.} 
Observe that the kinetic energy (resp. the potential energy) is lower semicontinuous with respect to strong $L^1_{\rm loc}(\R^{2d})$-convergence of $f$ (resp. weak* convergence in $\Measures(\R^d)$ of $\rho$).
Since \eqref{eqn:bdd-en} holds true for a.e. $t\in (0,\infty)$ by Step 1, and the maps $t \mapsto f_t \in L^1(\R^{2d})$
and $t \mapsto \rho_t \in \Measures(\R^d)$ are continuous for the $L^1_{\rm loc}$ and the weak* convergence respectively,
given any time $\bar t \in [0,\infty)$
it suffices to approximate it with a sequence  $t_n \to \bar t$ such that the energy bound \eqref{eqn:bdd-en} holds for every $t_n$
and let $n \to \infty$ to obtain that  \eqref{eqn:bdd-en} holds for $t=\bar t$.
%We show that
%\begin{equation}
%\label{eqn:kinetic-bound}
%\sup_{t\in[0,\infty)} \int_{\R^{2d}} |v|^2 f_t \,dx \,dv \leq  \int_{\R^{2d}} |v|^2 f_0 \, dx\, dv + \int_{\R^d} H \ast \rho_0 \, \rho_0 \, dx.
%\end{equation}
%To this end, let $t\geq 0$ and $t_n \to t$ be a sequence of times such that the energy bound \eqref{eqn:bdd-en} holds for every $t_n$. The strong convergence of $f_{t_n}$ to $f_t$ in $L^1_{\rm loc}$ implies that, for every $r>0$,
%\begin{equation*}
%\begin{split}
%\int_{B_r}|v|^2 f_t\, dx\, dv &= \lim_{n\to \infty}
%  \int_{B_r} |v|^2 f_{t_n}\, dx\, dv\leq \liminf_{n\to \infty} 
%  \int_{\R^{2d}} |v|^2 f_{t_n}\, dx\, dv.
%\end{split}
%\end{equation*}
%Taking the supremum in $r$, we deduce that
%\begin{equation}
%\label{eqn:lsc-kinetic-pointw}
%\int_{\R^{2d}} |v|^2 f_t\, dx\, dv \leq \liminf_{n\to \infty}
%  \int_{\R^{2d}} |v|^2 f_{t_n}\, dx\, dv.
%\end{equation}
%This proves \eqref{eqn:kinetic-bound}.

\smallskip
\noindent {\bf Step 3: strong $L^1_{\rm loc}$-continuity of the physical density and the electric fields.} Given $t \in [0,\infty)$, consider a sequence of times $t_n \to t$.
Fix $r>0$, and notice that for any $R>0$
$$ \int_{B_r} \int_{\R^d} |f_{t_n}-f_t| \, dv \, dx \leq  \int_{B_r} \int_{B_R} |f_{t_n}-f_t| \, dv \, dx + \int_{B_r} \int_{\R^d \setminus B_R} \frac{|v|^2} {R^{2}} (f_{t_n} +f_t) \, dv \, dx.$$
Thanks to  \eqref{eqn:bdd-en}  and the strong $L^1_{\rm loc}$ continuity of $f_t$, we can first let 
$n \to \infty$ and then $R \to \infty$ to deduce that
$$\lim_{n\to \infty} \int_{B_r} |\rho_{t_n}-\rho_t|  \, dx=\lim_{n\to \infty} \int_{B_r} \int_{\R^d} |f_{t_n}-f_t| \, dv \, dx = 0.$$
This proves the strong $L^1_{\rm loc}$-continuity of $\rho_t.$
Since $E_t = K \ast \rho_t$ and $\|\rho_t\|_{L^1(\R^d)} \leq C$, it is simple to see that also $E_t$ is strongly continuous in $L^1_{\rm loc}(\R^d)$.

%\smallskip
%
%\noindent {\bf Step 4: strong $L^1_{\rm loc}$-continuity of the force field.}  The force field $E_t = K \ast \rho_t$ is weakly $L^1_{\rm loc} (\R^d)$ continuous since, by Step 3, $\rho_t$ is weakly* continuous in time (as measures in $\Measures(\R^d)$).
%
%Also, since $K = \nabla H \in W^{\alpha,p}_{\rm loc}(\R^d;\R^d)$ for every $\alpha<1$ and $p< d/(d-1+\alpha)$
%(as already observed in Step 4 of the proof of Theorem \ref{thm:dato-mis}), a simple computation shows that, for every $R>0$ and $t\in (0,\infty)$,
%$$
%\|\rho_t \ast K\|_{W^{\alpha,p}(B_R; \R^d)} \leq \|\rho_t\|_{L^1(\R^d)}  \sup_{y\in \R^d} \| K\|_{W^{\alpha,p}(B_R(y); \R^d)} \leq C(R).
%$$
%Hence the strong continuity of the force field $E_t = K \ast \rho_t$ in $L^1_{\rm loc} (\R^d)$ follows then from the fractional Rellich theorem, which provides the compact embedding of the fractional space $ W^{\alpha,p}(B_R; \R^d)$ in $L^1(B_R; \R^d)$.

%\smallskip
%
%\noindent {\bf Step 5: bound on the total energy for every time.} 
%In order to conclude the proof of \eqref{eqn:bdd-en}, we observe that both the kinetic and the potential energy are lower semicontinuous with respect to strong $L^1_{\rm loc}(\R^{2d})$-convergence of $f$ and weak convergence (as measures) of $\rho$, respectively. Indeed, the first has been observed in Step 2 and the second is proved in Lemma~\ref{lemma:lsc-en-int-in-t}. Then by Step 1 the total energy is bounded by the initial energy for $\Leb{1}$-a.e. time and the lower semicontinuity of the total energy implies that the same property holds for {\em every} time.

\smallskip
{
\noindent {\bf Step 4: global characteristics in dimension $3$ and $4$.} 
Since $E_t= \nabla H \ast \rho_t$, 
the bound \eqref{eqn:bdd-en} and Lemma~\ref{lemma:force-bound} allow us to apply Corollary~\ref{cor:vp-bounded1} to deduce that trajectories do not blow up. %\todo{Domanda: sarebbe bello capire se il teorema di BBC ricade nel nostro caso, cioe' se si puo' dire che il flusso non esplode. Non e' che per caso $ \int_{\R^d} H \ast \rho\, \rho \, dx < \infty$ implica che $\rho \in L^{...}$ cosicche' si applica il nostro risultato sul flusso globalmente definito?}
}
\end{proof}

\begin{remark}\label{remark:cons-mass}{\rm 
As a consequence of Theorem~\ref{thm:vp-bounded1},
Corollary~\ref{cor:dato-L1}, and Remarks \ref{rmk:3d} and \ref{rmk:2d}, we deduce that, for $d=2,3,4$, finite energy solutions conserve the mass, namely $\|f_t\|_{L^1(\R^{2d})}=\rho_t(\R^d) = \rho_0(\R^d)=\|f_0\|_{L^1(\R^{2d})}$ for every $t\in [0,\infty)$.
In particular, in this case solutions are strongly continuous in $L^1(\R^{2d})$ and not only in $L^1_{\rm loc}(\R^{2d})$ (see for instance the argument in Step 2
of the proof of Theorem \ref{thm:distr-on-integral-curves}).}
\end{remark}
}

\section{Maximal Regular Flows of the state space and renormalized solutions}
\label{sect:flow}

The aim of this and next section is to develop the abstract theory of Maximal Regular Flows
and Lagrangian/renormalized solutions that are behind the results presented in the previous sections.
We warn the reader that from now on, since the theory is completely general,
we shall often consider flows of vector fields in $\R^d$ and denote by $x$ a point in $\R^d$. Then, for the applications to kinetic equations
in the phase-space $\R^{2d}$, one should apply these results replacing $d$ with $2d$ and $x$ with $(x,v)$.

\subsection{Preliminaries on Maximal Regular Flows}\label{sec:maxflow}

In this section we recall the basic results in \cite{amcofi}, where a local version of the theory of DiPerna-Lions \cite{lions} and Ambrosio \cite{ambrosio} was developed.  First we recall the definition of a local (in space and time) 
version of the Regular Lagrangian Flow introduced by Ambrosio \cite{ambrosio}.
Here and in the sequel, $\BorelSets{\R^d}$ denotes the collection of Borel sets in $\R^d$, 
and $AC([\tau_1,\tau_2];\R^d)$ is the space of absolutely continuous curves on $[\tau_1,\tau_2]$ with values in $\R^d$.

\begin{definition} [Regular Flow] \label{def:regflow}
Let $B\in\BorelSets{\R^d}$, $\tau_1<\tau_2$, and  $\bb: (\tau_1,\tau_2) \times \R^d \to \R^d$ be a Borel vector field. We say that  a Borel map
$\XX: [\tau_1,\tau_2]\times B\to \R^d$ is a 
{\em Regular Flow} (relative to $\bb$) in $[\tau_1,\tau_2]\times B$ if the following two properties hold: 
\begin{itemize}
\item[(i)] for a.e. $x\in B$, $\XX(\cdot,x)\in AC([\tau_1,\tau_2];\R^d)$ and solves the ODE $\dot x(t)=\bb_t(x(t))$ a.e.
in $(\tau_1,\tau_2)$, with the initial condition $\XX(\tau_1,x)=x$;
\item[(ii)] there exists a constant $C=C(\XX)$ satisfying $\XX(t,\cdot)_\#(\Leb{d}\res B)\leq C\Leb{d}$ for all $t\in [\tau_1,\tau_2]$.
\end{itemize}
\end{definition}

Let $T \in (0,\infty)$ and let $\bb:(0,T) \times \R^d \to \R^d$ be a Borel vector field.
The main object of our analysis is the Maximal Regular Flow, which takes into account the possibility of blow-up before time $T$ (or
after time $0$, when an initial condition $s\in (0,T)$ is under consideration).

\begin{definition}[Maximal Regular Flow]\label{def:maxflow}
For every $s\in (0,T)$ we say that  a Borel map $\XX(\cdot, s, \cdot)$ is a {\em Maximal Regular Flow} starting at time $s$ if there exist two Borel maps $T^+_{s,\sXX}:\R^d\to(s,T]$, $T^-_{s,\sXX}:\R^d\to [0,s)$ such that $\XX(\cdot,x)$ is defined in $(T^-_{s,\sXX}(x),T^+_{s,\sXX}(x))$ and the following two properties hold:
\begin{itemize}
\item[(i)] for a.e. $x\in\R^d$, $\XX(\cdot,s,x)\in AC_{\rm loc}((T^-_{s,\sXX}(x),T^+_{s,\sXX}(x));\R^d)$ and solves the ODE $\dot x(t)=\bb_t(x(t))$ a.e.
in $(T^-_{s,\sXX}(x),T^+_{s,\sXX}(x))$, with the initial condition $\XX(s,s,x)=x$;
\item[(ii)] there exists a constant $C=C(s,\XX)$ such that  
\begin{equation}
\label{eqn:incompr-o}
\XX(t,s,\cdot)_\# \big(\Leb{d} \res \{T^-_{s,\sXX}<t< T^+_{s,\sXX} \} \big)\leq C\Leb{d}\qquad\forall \,t\in [0,T];
\end{equation}
\item[(iii)] for a.e. $x\in\R^d$, either $T^+_{s,\sXX}(x)=T$ (resp. $T^-_{s,\sXX}(x)=0$) and $\XX(\cdot,s,x)$ can be 
continuously extended up to $t=T$ (resp. $t=0$) so that $\XX(\cdot,s,x)\in C([s,T];\R^d)$ (resp. $\XX(\cdot,s,x)\in C([0,s];\R^d)$), or 
\begin{equation}\label{eq:blowup}
\lim_{t\uparrow T^+_{s,\sXX}(x)}|\XX(t,s,x)|=\infty \qquad \textit{(resp. }\lim_{t\downarrow T^-_{s,\sXX}(x)}|\XX(t,s,x)|=\infty).
\end{equation}
In particular, $T^+_{s,\sXX}(x)<T$ (resp. $T^-_{s,\sXX}(x) > 0$) implies \eqref{eq:blowup}.
\end{itemize}
\end{definition}

The definition of Maximal Regular Flow can be extended up to the times $s=0$ and $s=T$, setting $T^-_{0,\sXX}\equiv 0$ and
$T^+_{T,\sXX}\equiv T$.

A Maximal Regular Flow has been built in \cite{amcofi} under general local assumptions on $\bb$. Before stating the result, 
we recall these assumptions. For $T\in (0,\infty)$ we are given a Borel vector field $\bb:(0,T)\times\R^d\to\R^d$ satisfying:
\begin{itemize}
\item[{\bf (A1)}] $\int_0^T\int_{B_R}|\bb_t(x)|\,dx\,dt<\infty$ for any $R>0$;
\item[{\bf (A2)}] for any nonnegative $\bar\rho\in L^\infty_+(\R^d)$ with compact support and 
any closed interval $[a,b]\subset [0,T]$, the continuity equation
\begin{equation}
\label{eqn:ce}
\frac{d}{dt}\rho_t+{\rm div\,}(\bb_t\rho_t)=0\qquad\text{in $(a,b)\times\R^d$}
\end{equation}
has at most one solution in the class of all weakly$^*$ nonnegative continuous functions $[a,b]\ni t\mapsto\rho_t$ with $\rho_a=\bar\rho$ and
$\cup_{t\in [a,b]}\supp\rho_t\Subset\R^d$.
\end{itemize}
Since the vector fields that arise in the applications we have in mind are divergence-free, we assume throughout the paper that our velocity field $\bb$ satisfies 
\begin{equation}
\label{eqn:incompb-o-L}
\div \bb_t = 0 \qquad \mbox{in } \R^d\text{ in the sense of distributions,  for a.e. $t\in (0,T)$.}
\end{equation}
The existence and uniqueness of the Maximal Regular Flow after time $s$, as well as the semigroup property, were proved in 
\cite[Theorems 5.7, 6.1, 7.1]{amcofi} assuming a one sided bound (specifically a lower bound) on the
divergence. In this context, uniqueness should be understood as follows: if $\XX$ and $\YY$ are Maximal Regular Flows, for all $s\in [0,T]$ one has
\begin{equation}\label{eq:understanding_uniqueness}
\begin{cases}
T^\pm_{s,\sXX}(x)=T^\pm_{s,\sYY}(x)\,\,\text{for a.e. $x\in\R^d$}&\\
\XX(\cdot,s,x)=\YY(\cdot,s,x)\,\,\text{in $(T_{s,\sXX}^-(x),T^+_{s,\sXX}(x))$ for a.e. $x\in\R^d$.}&
\end{cases}
\end{equation}
Under our assumptions on the divergence, by simply reversing the time variable, the Maximal Regular Flow can be built both forward and backward in time, so we state the result directly in the time-reversible case.

\begin{theorem}[Existence, uniqueness, and semigroup property]\label{thm:maximalflow}
 Let $\bb: (0,T) \times \R^d \to \R^d$ be a Borel vector field satisfying {\bf (A1)} and {\bf (A2)}.
Then the Maximal Regular Flow starting from any $s\in [0,T]$ is unique according to
\eqref{eq:understanding_uniqueness}, and existence is ensured under the additional assumption \eqref{eqn:incompb-o-L}.
In addition, still assuming \eqref{eqn:incompb-o-L}, for all $s\in [0,T]$ the following properties hold:
\begin{itemize}
\item[(i)]  the compressibility constant $C(s,\XX)$ in Definition~\ref{def:maxflow}  equals $1$ and for every $t\in [0,T]$
\begin{equation}
\label{eqn:incompr-R}
\XX(t,s,\cdot)_\# \big(\Leb{d} \res \{T^-_{s,\sXX}<t< T^+_{s,\sXX} \} \big) = \Leb{d} \res\bigl( \XX(t,s,\cdot)(\{T^-_{s,\sXX}<t< T^+_{s,\sXX} \})\bigr);
\end{equation}
\item[(ii)] if $\tau_1 \in [0,s]$, $\tau_2 \in [s,T]$, and $\YY$ is a Regular Flow in $[\tau_1,\tau_2] \times B$, then $T^+_{s,\sXX}>\tau_2$, $T^-_{s,\sXX}<\tau_1$ a.e. in $B$; moreover
\begin{equation}\label{coincidence1-o}
\XX(\cdot,s,x)=\YY(\cdot,\XX(\tau_1,s,x))\quad\text{in $[\tau_1,\tau_2]$, for a.e. $x\in B$};
\end{equation}
\item[(iii)] the Maximal Regular Flow satisfies the semigroup property, namely for all $s,\,s'\in [0,T]$
\begin{equation}\label{eqn:consecutive}
T^{\pm}_{s',\sXX}(\XX(s',s,x))=T^\pm_{s,\sXX}(x),
\qquad\text{for $\Leb{d}$-a.e. $x\in\{T^+_{s,\sXX}> s'>T^-_{s,\sXX}\}$,}
\end{equation}
and, for a.e. $x\in\{T^+_{s,\sXX}>s'>T^-_{s,\sXX}\}$,
\begin{equation}\label{eqn:semig3}
\XX\bigl(t,s',\XX(s',s,x)\bigr)=\XX(t, s, x)\qquad \text{$\forall \,t \in(T^-_{s,\sXX}(x),T^+_{s,\sXX}(x))$.  }
\end{equation}
\end{itemize}
\end{theorem}

%The following criterion, taken from \cite[Theorem 7.6]{amcofi}, provides a simple condition for global existence of the maximal flow.
%We state the result in the ``forward'' case, i.e. for $\XX(\cdot,0,\cdot)$, because this is the case explicitly 
%considered in \cite{amcofi}, the statement can be immediately adapted to cover the other cases when $s\in (0,T]$.
%
%\begin{theorem}[No blow-up criterion]\label{thm:noblowup}
%Let  $\bb: (0,T) \times \R^d \to \R^d$ be a divergence-free Borel vector field which satisfies {\bf (A1)} and {\bf (A2)}. Assume that $\rho_t\in L^\infty\bigl((0,T);L^\infty_+(\R^d)\bigr)$ is a weakly$^*$ continuous solution of the continuity equation 
%satisfying the integrability condition
%\begin{equation}
%\label{no-blowup}
%\int_0^T\int_{\R^d}\frac{|\bb_t(x)|}{1+|x|}\,\rho_t(x)\, dx\,dt<\infty.
%\end{equation}
%Then $T_{0,\sXX}^+(x)=T$ and $\XX(\cdot,0,x)\in AC([0,T];\R^d)$ for $\rho_0\Leb{d}$-a.e. $x\in\R^d$.
%\end{theorem}

\subsection{Uniqueness for the continuity equation and singular integrals}
\label{sect:singular integrals}

In this section we deal with uniqueness of solutions to the continuity equation when the gradient of the vector field is given by the singular integral of a time dependent family of measures. The theorem is a minor variant of a result by Bohun, Bouchut, and Crippa \cite{bbc-abstract} (see also \cite{boucrising}, where the uniqueness is proved for vector fields whose gradient is the singular integral of an $L^1$ function). We give the proof of the theorem under the precise assumptions that we need later on, since \cite{bbc-abstract} deals with globally defined regular flows (hence the authors need to assume global growth conditions on the vector field), whereas here we present a local version of such result.

\begin{theorem}\label{thm:bc}
Let $\bb: (0,T) \times \R^{2d} \to \R^{2d}$ be given by $\bb_t(x,v)=(\bb_{1t}(v) , \bb_{2t}(x))$, where
$$\bb_1 \in L^\infty((0,T); W^{1,\infty}_{\rm loc}(\R^d;\R^d)),\qquad\bb_{2t}= K \ast \rho_t,$$ 
with $\rho \in L^\infty((0,T); \Measures_+(\R^d))$ and $K(x) = x/|x|^{d}$.\\
Then $\bb$ satisfies ${\bf (A2)}$ of Section~\ref{sec:maxflow}, namely the uniqueness of bounded compactly supported
nonnegative distributional solutions of the continuity equation.
\end{theorem}

\begin{Proof} To simplify the notation we give the proof in the case of autonomous vector fields, 
but the same computations work for the general statement.
From now on, we denote by $\Probabilities{X}$ the set of probability measures on a space $X$,
and we use $e_t:C([0,T];\R^k) \to \R^k$ to denote the evaluation map at time $t$, that is $e_t(\eta):=\eta(t)$
(depending on the context, $k$ may be equal to $d$ or $2d$).

It is enough to show that, given $B_R \subset\R^{d}$ and $\eeta\in\Probabilities{C([0,T];B_R \times B_R)}$ 
concentrated on integral curves of $\bb$ and such that $(e_t)_\#\eeta \leq C_0\Leb{2d}$ for all $t\in [0,T]$,  
the disintegration $\eeta_x$ of $\eeta$ with respect to the map $e_0$ is a Dirac delta for $e_{0\#} \eeta$-a.e. $x$.
Indeed, thanks to Theorem \ref{thm:repr-foliat} below, any two bounded compactly supported
nonnegative distributional solutions with the same initial datum $\bar \rho$ can be represented by $\eeta_1,\,\eeta_2 \in \Probabilities{C([0,T];B_R \times B_R)}$. Hence, setting $\eeta = ( \eeta_1+\eeta_2 )/2$, if we can prove that $\eeta_x$ is a Dirac delta for $\bar \rho$-a.e. $x$ we deduce that $(\eeta_1)_x=(\eeta_2)_x=\eeta_x$ for $\bar \rho$-a.e. $x$, thus
$\eeta_1 = \eeta_2$.

\smallskip

To show that $\eeta_x$ is a Dirac delta for $e_{0\#} \eeta$-a.e. $x$, let us consider the function
$$\Phi_{\delta, \zeta}(t) := \iiint \log \Big( 1+ \frac{| \gamma^1(t) - \eta^1(t)|} {\zeta \,\delta} + \frac{| \gamma^2(t) - \eta^2(t)|} {\delta} \Big) d \eeta_{x}( \gamma) d\eeta_{x}(\eta) \, d \bar\rho(x),
$$
where $\delta,\,\zeta \in (0,1)$ are small parameters to be chosen later, 
$t\in [0,T]$, $\bar\rho:=  (e_0)_\# \eeta$, and we use the notation $\gamma(t)=(\gamma^1(t),\gamma^2(t)) \in \R^d \times \R^d$.
It is clear that $\Phi_{\delta,\zeta}(0) =0$.

Let us define the probability measure $\mu \in \Probabilities{\R^d \times C([0,T]; \R^d)^2}$ by $d\mu(x,\eta,\gamma) := d\eeta_x(\eta) \,d\eeta_x(\gamma) \,d \bar \rho (x)$, and
assume by contradiction that $\eeta_x$ is not a Dirac delta for $\bar\rho$-a.e. $x$. This means that there exists a constant $a>0$
such that
$$
\iiint \biggl( \int_0^T \min\bigl\{|\gamma(t)-\eta(t)|,1 \bigr\} \,dt \biggr)\,d\mu(x,\eta,\gamma)\geq a.
$$
By Fubini's Theorem this implies that there exists a time $t_0 \in (0,T]$ such that
$$
\iiint \min\bigl\{|\gamma(t_0)-\eta(t_0)|,1 \bigr\} \,d\mu(x,\eta,\gamma)\geq \frac{a}{T}.
$$
Since the integrand is bounded by $1$ and the measure $\mu$ has mass $1$, this means that the set
$$
A:=\biggl\{(x,\eta,\gamma)\,:\,\min\bigl\{|\gamma(t_0)-\eta(t_0)|,1 \bigr\} \geq \frac{a}{2T}\biggr\}
$$
has $\mu$-measure at least $a/(2T)$. Then, 
assuming without loss of generality that $a\leq 2T$, this implies that $|\gamma(t_0)-\eta(t_0)|\geq a/(2T)$ for all $(x,\eta,\gamma)\in A$,
hence
\begin{equation}
\label{eqn:BCcontr}
\begin{split}
\Phi_{\delta,\zeta}(t_0)&\geq \iiint_A \log \Big( 1+ \frac{| \gamma^1(t_0) - \eta^1(t_0)|} {\zeta \,\delta} + \frac{| \gamma^2(t_0) - \eta^2(t_0)|} {\delta} \Big)\,
d\mu(x,\eta,\gamma)\\
&\geq \frac{a}{2T}\log\Bigl(1+\frac{a}{2\delta T}\Bigr).
\end{split}
\end{equation}
We now want to show that this is impossible.

\smallskip

Computing the time derivative of $\Phi_{\delta,\zeta}$ we see that
\begin{equation}
\label{eqn:phi-deriv}
\frac{d\Phi_{\delta,\zeta}}{dt} (t) \leq \int_{\R^d} \int \int  \bigg( \frac{|\bb_1(\gamma^2(t)) - \bb_1(\eta^2(t))|} {\zeta\bigl(\delta+| \gamma^2(t) - \eta^2(t)|\bigr)} + \frac{\zeta |\bb_2(\gamma^1(t)) - \bb_2(\eta^1(t))|} {\zeta \,\delta+| \gamma^1(t) - \eta^1(t)|}\bigg) \,d\mu(x,\eta,\gamma).
\end{equation}
By our assumption on $\bb_1$, the first summand is easily estimated using the Lipschitz regularity of $\bb_1$ in $B_R$:
\begin{equation}
\label{eqn:b2}
\int_{\R^d} \int \int  \frac{|\bb_1(\gamma^2(t)) - \bb_1(\eta^2(t))|} {\zeta(\delta+| \gamma^2(s) - \eta^2(s)|)} 
\,d\mu(x,\eta,\gamma) \leq \frac{\|\nabla \bb_1\|_{L^\infty(B_R)}}{\zeta}.
\end{equation}
To estimate the second integral we show that for some constant $C$, which depends only on $d$, $\rho(\R^d)$, and $R$,
one has 
\begin{equation}
\label{eqn:termine-rognoso}
\iiint  \frac{\zeta\, |K \ast \rho(\gamma^1(t)) - K \ast \rho(\eta^1(t))|} {\zeta \,\delta+| \gamma^1(t) - \eta^1(t)|} \,
d\mu(x,\eta,\gamma)\leq  C \,\zeta\,   \bigg(1+\log \Bigl(\frac{C}{\zeta \,\delta}\Bigr) \bigg).
\end{equation}
To this end, we first recall the definition of weak $L^p$ norm of a $\mu$-measurable function $f:X\to\R$ in a measure space $(X,\mu)$:
$$ |||f |||_{M^p(X, \mu)} := \sup \bigl\{ \lambda\, \mu(\{ |f|>\lambda\})^{1/p}: \lambda>0\bigr\}. $$
By \cite[Proposition 4.2 and Theorem 3.3(ii)]{boucrising}
there exists a modified maximal operator $\tilde M$, which associates to every function of the form $DK \ast \sigma$, $\sigma \in \Measures_+(\R^d)$, the function $\tilde M (DK \ast \sigma) \in L^1(\R^d)$ with the following properties:
there exists a set $L$ with $\Leb{d}(L)=0$ such that 
\begin{equation}
\label{eqn:maximal-est-analog}
|K \ast \sigma (x)- K \ast \sigma (y)| \leq C\bigl[\tilde M(DK \ast \sigma)(x)+ \tilde M(DK \ast \sigma )(y)\bigr]\, | x-y | \qquad \forall\,
x,\,y\in\R^d\setminus L,
\end{equation}
and the weak-$L^1$ estimate
\begin{equation}
\label{eqn:weakL1}
 ||| \tilde M(DK \ast \rho) |||_{M^1(B_R)} \leq C\,\rho(\R^d)
\end{equation}
holds with a constant $C$ which depends only on $d$ and $R$.
Applying \eqref{eqn:maximal-est-analog}, we see that
\begin{equation}
\label{eqn:terminerognoso1}\iiint  \frac{|K \ast \rho(\gamma^1(t)) - K \ast \rho(\eta^1(t))|} {\zeta \,\delta+| \gamma^1(t) - \eta^1(t)|}\,
d\mu\leq \int g_t(x,\eta,\gamma) \, d\mu,
\end{equation}¥
where
$$
g_t(x,\eta,\gamma) := \min \bigg\{ C\,\tilde M(DK \ast \rho)(\gamma^1(t)) +  C\,\tilde M(DK \ast \rho)(\eta^1(t)) , 
\frac{|K \ast \rho|(\gamma^1(t)) +|K \ast \rho|(\eta^1(t))}{\zeta \,\delta} \bigg\}.$$
Let us fix $p:=\frac{d}{d-1/2}\in \left(1,\frac{d}{d-1}\right)$, so that $|K| \in L^p_{\rm loc}(\R^d)$.
The last term in \eqref{eqn:terminerognoso1} can be estimated thanks to the following interpolation inequality (see \cite[Lemma 2.2]{boucrising})
$$ \| g_t\|_{L^1(\mu)} \leq \frac{p}{p-1} ||| g_t|||_{M^1(\mu)} \bigg( 1+ \log\Big( \frac{||| g_t|||_{M^p(\mu)}}{||| g_t|||_{M^1(\mu)}} \Big) \bigg).
$$
Then, the first term in the right-hand side above can be estimated 
using our assumption  $(e_t)_\#\eeta \leq C_0\Leb{d}$ and \eqref{eqn:weakL1}:
\begin{align*}
||| g_t |||_{M^1(\mu)} &\leq 2\,  ||| \tilde M(DK \ast \rho)(\eta^1(t)) |||_{M^1(\mu)}
\\&= 2 \,  ||| \tilde M(DK \ast \rho)(\eta^1(t)) |||_{M^1(\seeta)}
\\
&= 2\,   ||| \tilde M(DK \ast \rho)(x) |||_{M^1(B_R \times B_R, e_{t\#} \seeta)}
\\
&\leq 2 \,  C_0\,   ||| \tilde M(DK \ast \rho)(x) |||_{M^1(B_R\times B_R, \Leb{2d})}
\\
&\leq 2 \,  C_0\,  \Leb{d}(B_R)\,  ||| \tilde M(DK \ast \rho)(x) |||_{M^1(B_R, \Leb{d})}  
\\&\leq 2\,  C_0\,  C\,   \Leb{d}(B_R)\,  \rho(\R^d).
\end{align*}
Similarly, the second term in the right hand side can be estimated using $(e_t)_\#\eeta \leq C_0\Leb{d}$ and Young's
inequality:
\begin{align*}
||| g_t |||_{M^p(\mu)} &\leq  2\, (\zeta \,\delta)^{-1} \|(K \ast \rho)(\eta^1(t))\|_{L^p(\mu)}\\
&=
2 \,(\zeta \,\delta)^{-1} \|(K \ast \rho)(\eta^1(t))\|_{L^p(\seeta)}
\\
&\leq 2\, C_0\, (\zeta \,\delta)^{-1} \|(K \ast \rho)(x)\|_{L^p(B_R \times B_R)} \\
&\leq 2\, C_0 \,(\zeta \,\delta)^{-1} \Leb{d}(B_R)\,  \|(K \ast \rho)\|_{L^p(B_R)}
\\&\leq 2\,C_0\, (\zeta \,\delta)^{-1} \Leb{d}(B_R)\,\|K\|_{L^p(B_R)} \,  \rho(\R^d)
\\&\leq C\,(\zeta \,\delta)^{-1},
\end{align*}
where $C$ depends on $d$, $R$, and $\rho(\R^d)$. Combining these last estimates with \eqref{eqn:terminerognoso1}, 
we obtain~\eqref{eqn:termine-rognoso}.

Then, using \eqref{eqn:phi-deriv}, \eqref{eqn:b2}, and \eqref{eqn:termine-rognoso}, we deduce that
$$\frac{d\Phi_{\delta,\zeta}}{dt} (t) \leq \frac{C}{\zeta} + C\,  \zeta+ C\,   \zeta\,   \log \Big(\frac{C }{\zeta \,\delta} \Big)$$
for some constant $C$ depending only on $d$, $R$, $\rho(\R^d)$, and $\|\nabla b_1\|_{L^\infty(\R^d)}$. Integrating with respect to time in $[0,t_0]$, we find that
$$\Phi_{\delta,\zeta}(t_0) \leq C\,  t_0\,  \bigg( \frac{1}{\zeta} + \zeta +\zeta\,   \log\Big( \frac{C}{\zeta} \Big) +\zeta\,   \log \Big( \frac{1}{\delta} \Big) \bigg).$$
Choosing first $\zeta>0$ small enough in order to have $C\,   t_0\,  \zeta <a/(2T)$ and then letting 
$\delta \to 0$, we find a contradiction with \eqref{eqn:BCcontr}, which concludes the proof.
\end{Proof}

\subsection{Generalized flows and Maximal Regular Flows}

We denote by $\rncp{d}=\R^d\cup\{\infty\}$ the one-point compactification of $\R^d$ and we recall the definition of generalized flow and
of regular generalized flow in our context, as introduced in \cite[Definition 5.3]{amcofi}.

\begin{definition}[Generalized flow]\label{defn:genflow}
Let $\bb:(0,T)\times \R^d \to \R^d$ be a Borel vector field.
The measure $\eeta \in \Measuresp{C([0,T]; \rncp{d})}$ is said to be a {\em generalized flow} of $\bb$ if $\eeta$ 
is concentrated on the set\footnote{In connection with the definition of generalized flow, let us provide a sketch of proof of the fact that the
set $\Gamma$ in \eqref{defn:gen-flow} is Borel in $C([0,T];\rncp{d})$.

First of all one notices that for
all intervals $[a,b]\subset [0,T]$ the set $\{\eta:\ \eta([a,b])\subset\R^d\}$ is Borel. Then, considering the
absolute continuity of a curve $\eta$ in the integral form
$$
|\eta(t)-\eta(s)|\leq\int_s^t|\bb_r(\eta(r))|\,dr\qquad\forall \,s,t\in [a,b],\,\,s\leq t,
$$
it is sufficient to verify (arguing componentwise and splitting in positive and negative part)
that for any nonnegative Borel function $\cc$ and for any $s,t\in [0,T]$ with $s\leq t$ fixed, the function
$$
\eta\mapsto\int_s^t \cc_r(\eta(r))\,dr
$$
is Borel in $\{\eta:\ \eta([a,b])\subset\R^d\}$. This follows by a monotone class argument, since the property is obviously true
for continuous functions and it is stable under equibounded and monotone convergence.

As soon as the absolute continuity property is secured, also the verification of the
Borel regularity of 
$$
\Gamma\cap\{\eta:\ \eta([a,b])\subset\R^d\}=\bigl\{\eta\in C([0,T]);\rncp{d}): \eta\in AC([a,b];\R^d),\,\,\,
\text{$\dot\eta(t)=\bb_t(\eta(t))$ a.e. in $(a,b)$}\bigr\}
$$
can be achieved following similar lines. Finally, by letting the endpoints $a$, $b$ vary in a countable dense set we obtain that
$\Gamma$ is Borel.}
\begin{multline}
\label{defn:gen-flow}
\Gamma:=\big\{
\eta \in C([0,T]; \rncp{d}) :\ \textit{$\eta\in AC_{\rm loc}(\{\eta\neq\infty\};\R^d)$ and}
\\
\textit{$\dot \eta(t) = \bb_t(\eta(t))$ for a.e. $t \in \{\eta\neq\infty\}$}
\big\}.
\end{multline}
We say that a generalized flow $\eeta$ is {\em regular} if there exists $L_0\geq 0$ satisfying
\begin{equation}\label{eqn:noconcentration-eta-lemma}
(e_t)_\#\eeta \res \R^d \leq L_0\Leb{d} \qquad\forall \,t\in [0,T].
\end{equation}
\end{definition}

In the case of a smooth bounded vector field, a particular class of generalized flows is the one 
generated by transporting the initial measure along the integral lines of the flow:
$$\eeta = \int_{\R^d}\delta_{\sXX(\cdot,x)} \,d[(e_0)_\#\eeta](x).$$

In the next definition we propose a generalization of this construction involving Maximal Regular Flows.

\begin{definition}[Measures transported by the Maximal Regular Flow]\label{defn:transported}
Let $\bb:(0,T)\times \R^d \to \R^d$ be a Borel vector field having a Maximal Regular Flow $\XX$ and let
$\eeta \in \Measuresp{C([0,T]; \rncp{d})}$ with $(e_t)_\#\eeta\ll\Leb{d}$ for all $t\in [0,T]$. We say that
$\eeta$ is transported by $\XX$ if, for all $s\in [0,T]$, $\eeta$ is concentrated on
\begin{equation}
\label{defn:loc-int-fol}
\big\{
\eta \in C([0,T]; \rncp{d}) : \eta(s) = \infty \mbox{ or }\eta(\cdot) = \XX(\cdot, s,\eta(s)) \mbox{ in } (T^-_{s,\sXX}(\eta(s)) , T^+_{s,\sXX}(\eta(s)) )
\big\}.
\end{equation}
\end{definition}

The absolute continuity assumption $(e_t)_\#\eeta\ll\Leb{d}$ on the marginals of $\eeta$ is needed to ensure that this notion is invariant
with respect to the uniqueness property in \eqref{eq:understanding_uniqueness}. In other words, if $\XX$ and $\YY$ are related as in \eqref{eq:understanding_uniqueness} then
$\eeta$ is transported by $\XX$ if and only if $\eeta$ is transported by $\YY$.

It is easily seen that if $\eeta$ is transported by a Maximal Regular Flow, then $\eeta$ is a generalized flow according
to Definition~\ref{defn:genflow}, but in connection with the proof of the renormalization property we are more interested to the converse statement. As shown in the next theorem, this
holds for regular generalized flows and for divergence-free vector fields satisfying {\bf (A1)}-{\bf (A2)} of Section~\ref{sec:maxflow}.

\begin{theorem}[Regular generalized flows are transported by $\XX$]\label{thm:no-split-until-explosion}
 Let $\bb: (0,T) \times \R^d \to \R^d$ be a divergence-free vector field  satisfying {\bf (A1)}-{\bf (A2)} of Section~\ref{sec:maxflow} and let
 $\XX$ be its Maximal Regular Flow.
Let $\eeta \in \Measuresp{C([0,T];\rncp{d})}$ be a regular generalized flow according to Definition~\ref{defn:genflow}.\\
Given $s\in [0,T]$, consider a Borel family $\{\eeta^s_x\}\subset\Probabilities{C([0,T];\rncp{d})}$, $x\in \rncp{d}$,
of conditional probability measures representing $\eeta$ with respect to the marginal $(e_s)_\# \eeta$, that is, $\int \eeta^s_x\,d[(e_s)_\# \eeta](x)=\eeta$. 
Then for $(e_s)_\# \eeta$-almost every $x \in \R^d$
we have that $\eeta^s_x$ is concentrated on the set
\begin{equation}
\label{defn:loc-int-fol-s}
\hat \Gamma_s:=\big\{
\eta \in C([0,T]; \rncp{d}) \,:\, \eta(s)=x,\,\,\eta(\cdot) = \XX(\cdot, s,\eta(s)) \mbox{ in } (T^-_{s,\sXX}(\eta(s)) , T^+_{s,\sXX}(\eta(s)) )
\big\}.
\end{equation}
In particular $\eeta$ is transported by $\XX$.
\end{theorem}
\begin{Proof} First of all we notice that the set $\hat \Gamma_s$ in \eqref{defn:loc-int-fol-s} is
Borel. Indeed, the maps $\eta\mapsto T_{s,\sXX}^\pm(\eta(s))$ are Borel because $T_{\sXX}^\pm$
are Borel in $\R^d$, and the map $\eta\mapsto\XX(t,s,\eta(s))$ is Borel as well for any $t\in [0,T]$. Therefore, choosing
a countable dense set of times $t \in [0,T]$ the Borel regularity of $\hat \Gamma_s$ is achieved.

\smallskip

The fact that $\eeta^s_x$ is concentrated on the set  $\{\eta\,:\, \eta(s)= x\}$ is immediate from the definition of $\eeta^s_x$.
We now show that  for $(e_s)_\# \eeta$-almost every $x\in\R^d$ the measure
$\eeta^s_x$ is concentrated on the  set
\begin{equation}
\label{defn:loc-int-fol-s-unilat}
\big\{
\eta \in C([0,T]; \rncp{d}) : \eta(\cdot) = \XX(\cdot, s,x) \mbox{ in } [s , T^+_{s,\sXX}(x) )
\big\}.
\end{equation}
Applying the same result backward in time, this will prove that $\eeta_x^s$ is concentrated on the set $\hat \Gamma_s$ in \eqref{defn:loc-int-fol-s}.

\smallskip

For $r\in (s,T]$ we  denote by $\Sigma^{s,r}:C([0,T]; \rncp{d})\to C([s,r];\rncp{d})$ the map induced by restriction to $[s,r]$, 
namely $\Sigma^{s,r}(\eta):=\eta\vert_{[s,r]}$.

For every $R>0$, $r\in (s,T]$, let us consider 
$$\eeta^{R,r} := \Sigma^{s,r}_\# \Big( \eeta \res \bigl\{ \eta: \eta(t) \in B_R \mbox{ for every } t\in [s,r]\bigr\} \Big).$$
By construction $\eeta^{R,r}$ is a regular generalized flow relative to $\bb$ with compact support, hence our regularity assumption on $\bb$
allows us to apply \cite[Theorem 3.4]{amcofi} to deduce that 
\begin{equation}
\label{eqn:repr-eta-R-r}
\eeta^{R,r}= \int \delta_{\sYY(\cdot, x)}\,d[(e_s)_\# \eeta^{R,r}](x),
\end{equation}
where $\YY(\cdot,x)$ is an integral curve of $\bb$ in $[s,r]$ for $(e_s)_\# \eeta$-a.e. $x\in \R^d$. Let us denote by $\rho_{R,r}$ the density of $(e_s)_\# \eeta^{R,r}$ with respect to $\Leb{d}$, which is bounded by $L_0$ thanks to \eqref{eqn:noconcentration-eta-lemma}. For every $\delta>0$ we have that
\begin{equation}
\label{eqn:incompr-est1}
\begin{split}
\YY(t,\cdot)_\# \big( \Leb{d} \res \{\rho_{R,r}>\delta\}\big)
&= (e_t)_\# \int_{\{\rho_{R,r}>\delta\}}\!\!\!\!\! \delta_{\sYY(\cdot,x)}\,d\Leb{d}(x)
\\
&\leq \frac{1}{\delta}\,(e_t)_\# \int_{\{\rho_{R,r}>\delta\}}\!\!\!\!\! \delta_{\sYY(\cdot,x)}\,d[(e_s)_\# \eeta^{R,r}](x)
\\
&\leq \frac{1}{\delta}\,
(e_t)_\#\eeta^{R,r}
\leq \frac{1}{\delta}\,
(e_t)_\#\eeta \res \R^d \leq  \frac{L_0}{\delta} \,\Leb{d},
\end{split}
\end{equation}
hence $\YY(\cdot, x)$ is a Regular Flow of $\bb$ in $[s,r] \times \{\rho_{R,r}>\delta\}$ according to Definition~\ref{def:regflow}.
By Theorem~\ref{thm:maximalflow}(ii) we deduce that $\YY(\cdot, x) = \XX(\cdot, s,x)$ for a.e. $x\in \{ \rho_{R,s}>\delta\}$ and therefore, letting $\delta \to 0$,
\begin{equation}
\label{eqn:coinc-flow}
\YY(\cdot, x) = \XX(\cdot, s,x) \qquad \text{in $[s,r]$ for $(e_s)_\# \eeta^{R,s}$-a.e. $x\in \R^d$}.
\end{equation}
Letting $R\to \infty$ we have that $\eeta^{R,r} \to \ssigma^r$ increasingly, where
$$\ssigma^r := \Sigma^{s,r}_\# \Big( \eeta \res \{ \eta: \eta(t) \neq \infty \mbox{ for every } t\in [s,r]\} \Big),$$
and by \eqref{eqn:repr-eta-R-r} and \eqref{eqn:coinc-flow} we get that
\begin{equation}
\label{eqn:eeta-r-delta}
\ssigma^r= \int \delta_{\sXX(\cdot,s,x)}\,d[(e_s)_\#\ssigma^r](x)\qquad \forall\,r \in (s,T].
\end{equation}
Now, arguing by contradiction, let us assume that there exists a Borel set $E \subset \R^d$ such that $(e_s)_\# \eeta(E)>0$ and 
$\eeta^s_x$ is not concentrated on the  set \eqref{defn:loc-int-fol-s-unilat} for every $x\in E$, namely
$$ \eeta^s_x\Big(\big\{
\eta \in C([0,T]; \rncp{d}) : \eta \neq \XX(\cdot, s,x) \mbox{ as a curve in } [s , T^+_{s,\sXX}(x) )
\big\} \Big) > 0.
$$
Since this is equivalent to
 \begin{equation*}
 \eeta^s_x \bigg( \bigcup_{r\in \Q \cap (s,T^+_{s,\sXX}(x))} \bigl\{ \eta \in C([0,T]; \rncp{d}) : \eta \neq \XX(\cdot, s,x)\mbox{ in } [s ,r],\, \eta([s,r])\subset\R^d\bigr\} \bigg)>0,
 \end{equation*}
 we deduce that
for every $x\in E$ there exists $r_x\in  \Q \cap (s,T^+_{s,\sXX}(x))$ such that 
$$ 
\eeta^s_x\Big(\bigl\{ \eta \in C([0,T]; \rncp{d}) : \eta \neq \XX(\cdot, s,x) \mbox{ as a curve in } [s ,r_x], \, \eta([s,r_x])\subset\R^d\bigr\} \Big) >0.
$$
In other words, for every $x\in E$ there exists a rational number $r_x$ such that 
$$\Sigma^{s,r_x}_\# \Big( \eeta^s_x \res \{ \eta: \eta(t) \neq \infty \mbox{ for every } t\in [s,r_x]\} \Big) \mbox{ is nonzero and not a multiple of }\delta_{\sXX(\cdot,s,x)}.$$ 
Therefore, there exist a Borel set $E' \subset E$ of positive $(e_s)_\#\eeta$-measure and $r\in (s,T] \cap \Q$ such that for every $x\in E'$ 
$$\Sigma^{s,r}_\# \Big( \eeta^s_x \res \{ \eta: \eta(t) \neq \infty \mbox{ for every } t\in [s,r]\}  \Big)
\mbox{ is nonzero and not a multiple of }
 \delta_{\sXX(\cdot,s,x)}.$$
Notice now that, by \eqref{eqn:eeta-r-delta} and $(e_s)_\# \ssigma^r \leq (e_s)_\# \eeta$, it follows that
$$\int \delta_{\sXX(\cdot,s,x)}\,d[(e_s)_\# \eeta](x)\geq\ssigma^r=\int
\Sigma^{s,r}_\#\Big( \eeta^s_x \res \{ \eta: \eta(t) \neq \infty \mbox{ for every } t\in [s,r]\} \Big) \, d[(e_s)_\# \eeta](x),
$$
hence $ \delta_{\sXX(\cdot,s,x)}\geq \Sigma^{s,r}_\#\Big( \eeta^s_x \res \{ \eta: \eta(t) \neq \infty \mbox{ for every } t\in [s,r]\} \Big)$
for $(e_s)_\#\eeta$-a.e. $x$, and therefore a contradiction with the existence of $E'$. This proves 
that $\eeta^s_x$ is concentrated on the  set defined in \eqref{defn:loc-int-fol-s-unilat}, as desired.

\smallskip

Finally, in order to prove that $\eeta$ is transported by $\XX$ we apply the definition of disintegration and the fact that for $(e_s)_\# \eeta$-a.e. $x\in \R^d$ 
the measure $\eeta^s_x$ is concentrated on the set $\hat \Gamma_s$ in \eqref{defn:loc-int-fol-s} to obtain that
$\eeta (\hat \Gamma) =\int \eeta^s_x(\hat \Gamma)\,d[(e_s)_\# \eeta](x)=1$, where $\hat \Gamma$ is the set in \eqref{defn:loc-int-fol}.
\end{Proof}

\subsection{Regular generalized flows and renormalized solutions}
\label{sect:flow renormalized}

We now recall the well-known concept of renormalized solution to a continuity equation.
This was already introduced in Section \ref{sect:statement} in the context of the Vlasov-Poisson system, but we prefer to reintroduce it here in its
general formulation for the convenience of the reader. To fix the ideas we consider
the interval $(0,T)$ and $0$ as initial time, but the definition can be immediately adapted to general intervals, forward and
backward in time.

\begin{definition}[Renormalized solutions] \label{def:renormalized}  Let $\bb\in L^1_{\rm loc} ((0,T) \times \R^d; \R^d)$ be a 
Borel and divergence-free vector field. A Borel function $\rho:(0,T)\times \R^d \to \R$ is a renormalized solution of the continuity 
equation relative to $\bb$ if 
 \begin{equation}
 \label{defn:renorm-sol-eq}
 \partial_t \beta (\rho) + \nabla \cdot (\bb\beta(\rho)) =0\quad\text{in $(0,T)\times\R^d$}\qquad\forall \,\beta \in C^1\cap L^\infty(\R)%\,\, \beta(0)=0,
 \end{equation} 
 in the sense of distributions. Analogously, we say that $\rho$ is a renormalized solutions starting from a Borel
 function $\rho_0: \R^d \to \R$ if
  \begin{equation}
 \label{defn:renorm-sol-tested-st}
 \begin{split}
&\int_{\R^d} \phi_0(x)\, \beta (\rho_0(x))\, dx +
 \int_0^T \! \int_{\R^d} [\partial_t \phi_t(x) + \nabla\phi_t(x) \cdot\bb_t(x) ]\,\beta(\rho_t(x)) \, dx \, dt  =0
\end{split}
\end{equation}
for all  $\phi \in C^{\infty}_c ( [0,T) \times \R^d)$ and all $\beta\in C^1\cap L^\infty(\R)$.% with $\beta(0)=0$.
 \end{definition}

\begin{remark}[Equivalent formulations]\label{rmk:renorm-operativa}{\rm
As shown for instance in \cite[Section 8.1]{amgisa}), an equivalent formulation of \eqref{defn:renorm-sol-tested-st} is the following:
for every $\varphi \in C^{\infty}_c ( \R^d)$ the function $\int_{\R^d} \varphi(x)\, \beta (\rho_t(x))\, dx$ 
coincides a.e. with an absolutely continuous function $t\mapsto A(t)$ such that
$A(0)=\int_{\R^d} \varphi(x) \,\beta (\rho_0(x))\, dx $ and 
 \begin{equation}
 \label{defn:renorm-sol-tested}
 \begin{split}
 \frac{d}{dt}A(t)
 &= \int_{\R^d} \nabla\varphi(x) \cdot\bb_t(x)\,\beta(\rho_t(x)) \, dx\qquad\text{for a.e. $t\in (0,T)$.}
\end{split}
\end{equation}
Moreover, by an easy approximation argument, the same holds for every Lipschitz, compactly supported $\varphi:\R^d\to \R$. This way, possibly
splitting $\varphi$ in positive and negative parts, only nonnegative test functions need to be considered. Analogously, 
by writing every $\beta \in C^1(\R^d)$ as the sum of a $C^1$ monotone nondecreasing function and of a $C^1$ monotone nonincreasing function, we can
use the linearity of the equation with respect to $\beta(\rho_t)$ to reduce to the case of $\beta \in C^1\cap L^\infty(\R)$ 
monotone nondecreasing.
}\end{remark}

In the next theorem we show first that, flowing an initial datum $\rho_0 \in L^1(\R^d)$ through the maximal flow, we obtain a renormalized solution 
of the continuity equation. In turn, this is a key tool to prove the second part of the lemma, namely that any 
measure $\eeta$ transported by the maximal regular flow flow induces, through its marginals, a renormalized solution.
The proof of these facts heavily relies on the incompressibility of the flow and therefore on the assumption that the vector field is divergence-free.
A generalization of this lemma to the case of vector fields with bounded divergence is possible, but rather technical and long. We notice that the assumptions {\bf (A1)} and {\bf (A2)}, as well as the one on the divergence of the vector field $\bb$, are used only for the existence and uniqueness of a maximal regular flow which preserves the Lebesgue measure on its domain of definition (see Theorem~\ref{thm:maximalflow}).

To fix the ideas, in part (i) of the theorem below we consider only $0$ as initial time. An analogous statement can be given for
any other initial time $s\in [0,T]$, considering intervals $[0,s]$ or $[s,T]$, with no additional assumption on $\bb$.

\begin{theorem}\label{thm:distr-on-integral-curves}
 Let $\bb: (0,T) \times \R^d \to \R^d$ be a divergence-free vector field satisfying {\bf (A1)}-{\bf (A2)} of Section~\ref{sec:maxflow}.
 Let $\XX(t,s,x)$ be the maximal regular flow of $\bb$ according to Definition~\ref{def:maxflow}.
\begin{enumerate}
\item If $\rho_0 \in L^1(\R^d)$, we define $\rho_t\in L^1(\R^d)$ by
$$\rho_t := \XX(t ,0,\cdot)_\# (\rho_0 \res \{T^+_{0,\sXX}>t\}), \qquad t \in [0,T).$$ 
Then $\rho_t$ is a renormalized solution of the continuity equation
starting from $\rho_0$. In addition the map $t\mapsto \rho_t$ is strongly continuous on $[0,T)$ with respect to the
$L^1_{\rm loc}$ convergence, and it is also strongly $L^1$ continuous from the right.
\item
If $\eeta \in \Measuresp{C([0,T];\rncp{d})}$ is transported by $\XX$, 
and $(e_t)_\#\eeta \res \R^d\ll \Leb{d}$ for every $t\in [0,T]$, then the density $\rho_t$ of  $(e_t)_\#\eeta \res \R^d$ with 
respect to $\Leb{d}$ is strongly continuous on $[0,T)$ with respect to the
$L^1_{\rm loc}$ convergence and it is a renormalized solution of the continuity equation.
\end{enumerate}
\end{theorem}

\begin{Proof} We split the proof in four steps.

\noindent {\bf Step 1: proof of (i), renormalization property of $\rho_t$.}
In the proof of (i) we set for simplicity $\XX(t,x)=\XX(t,0,x)$ and
 $T_{0,\sXX}^+=T_{\sXX}$. We first notice that, by the incompressibility of the flow \eqref{eqn:incompr-R} and by the definition of $\rho_t$, 
for every $t\in [0,T)$ and $\varphi \in C_c(\R^d)$ one has
$$\int_{\{T_{\sXX}>t\}}\!\!\! \varphi(\XX(t,x))\, \rho_t(\XX(t,x)) \, dx = 
\int_{\sXX(t,\cdot)(\{T_{\sXX}>t\})}\!\!\!\!\! \varphi\, \rho_t\, dx
=\int_{\{T_{\sXX}>t\}} \!\!\!\varphi(\XX(t,x)) \,\rho_0\,dx,
$$
hence, for any $t\in [0,T)$,
\begin{equation}
\label{eqn:rho-t-formula}
\rho_t(\XX(t,x)) = \rho_0(x) \qquad \text{for $\Leb{d}$-a.e. $x\in \{T_{\sXX}>t\}$}.
\end{equation}
Let $\beta \in C^1\cap L^\infty(\R)$. Using again \eqref{eqn:incompr-R} and by \eqref{eqn:rho-t-formula}
we have that
\begin{equation}\label{eqn:rho-t-formula-bis}
\int_{\R^d} \varphi \,\beta (\rho_t)\, dx = \int_{\sXX(t,\cdot)(\{T_{\sXX}>t\})} \!\!\!\!\!\!\!\varphi\, \beta (\rho_t)\, dx = \int_{\{T_{\sXX}>t\}}\!\!\!\!\!\! \varphi(\XX(t,\cdot))\, \beta (\rho_0)\, dx  
\end{equation}
for any $\varphi\in C_c(\R^d)$. In addition, the blow-up property \eqref{eq:blowup} ensures that $t\mapsto\varphi(\XX(t,x))$ can be continuously
extended to be identically $0$ on the time interval $[T_{\sXX}(x),T)$ (in case of blow-up before time $T$); furthermore, for the same reason, if $\varphi\in C^1_c(\R^d)$ then
the extended map is absolutely continuous in $[0,T]$ and
\begin{equation}\label{eq:nov1414}
\frac{d}{dt}\varphi(\XX(t,x))=\chi_{[0,T_{\sXX}(x))}(t)\,\nabla \varphi(\XX(t,x))\cdot \bb_t(\XX(t,x))\qquad
\text{for $\Leb{1}$-a.e. $t\in (0,T)$.}
\end{equation}
Therefore, using \eqref{eqn:rho-t-formula-bis} and integrating \eqref{eq:nov1414}, for all $\varphi\in C^1_c(\R^d)$ we find that 
$$
\frac{d}{dt} \int_{\R^d} \varphi\,\beta (\rho_t)\, dx 
= \int_{\{T_{\sXX}>t\}} \nabla \varphi(\XX(t,\cdot))\cdot \bb_t(\XX(t,\cdot))\, \beta (\rho_0)\, dx 
= \int_{\R^d} \nabla \varphi\cdot \bb_t\, \beta (\rho_t)\, dx
$$
for $\Leb{1}$-a.e. $t\in (0,T)$, which proves the renormalization property.

\smallskip
\noindent {\bf Step 2: proof of (i), strong continuity of $\rho_t$.}
We notice that, as a consequence
of the possibility of continuously extending the map $t\mapsto\varphi(\XX(\cdot,x))$ after the time $T_{\sXX}(x)$ for $\varphi \in C_c(\R^d)$, the map
 $[0,T)\ni t\mapsto\rho_t$ is weakly continuous in the duality with $C_c(\R^d)$.
 Let us prove now the strong continuity of $t\mapsto \rho_t$.
 
 We start with the proof for $t=0$. Fix $\epsilon>0$, let $\psi\in C_c(\R^d)$ with $\|\psi-\rho_0\|_1<\epsilon$, and notice that
the positivity $\Leb{d}$-a.e. in $\R^d$ of $T_{\sXX}$ gives
$$
\int_{\R^d}|\rho_t(x)-\psi(x)|\,dx\leq \int_{\sXX(t,\cdot)(\{T_{\sXX}>t\})}|\rho_t(x)-\psi(x)|\,dx+\int_{\sXX(t,\cdot)(\{0<T_{\sXX}\leq t\})}|\psi(x)|\,dx
$$
and that the second summand in the right hand side is infinitesimal as $t\downarrow 0$.
Changing variables and 
using \eqref{eqn:rho-t-formula} together with the incompressibility of the flow, it follows that
$$
\int_{\sXX(t,\cdot)(\{T_{\sXX}>t\})}|\rho_t(x)-\psi(x)|\,dx=\int_{\{T_{\sXX}>t\}}|\rho_0(x)-\psi(\XX(t,x))|\,dx,
$$
therefore
$$
\limsup_{t\downarrow 0}\int_{\R^d}|\rho_t-\psi|\,dx\leq
\limsup_{t\downarrow 0}\int_{\{T_{\sXX}>t\}}|\rho_0(x)-\psi(\XX(t,x))|\,dx\leq
\int_{\R^d}|\rho_0-\psi|\,dx.
$$
This proves that $\limsup_t\|\rho_t-\rho_0\|_1\leq 2\epsilon$ and, by the arbitrariness of $\epsilon$, the desired strong continuity at $t=0$ follows.

We now notice that the same argument together with the semigroup property of Theorem~\ref{thm:maximalflow}(iii)
 shows that the map $t\mapsto \rho_t$ is strongly continuous from the right in $L^1$.
In addition, reversing the time variable and using again the semigroup property, we deduce that the
identity $\rho_t(x)=\rho_s(\XX(t,s,x))\,1_{\{ T_{\sXX}>t\}}(\XX(0,s,x))$ holds, therefore
$$
\lim_{s\uparrow t}\int_{\R^d}|\rho_t (x)-\rho_s(x)\,
1_{\{ T_{\sXX}>t\}}(\XX(0,s,x)) |\,dx=0\qquad\forall \,t\in (0,T).
$$
Hence, in order to prove that  the map $t\mapsto \rho_t$ is strongly continuous in $L^1_{\rm loc}$, we are left to show that 
for every $R>0$ and $t\in (0,T)$ one has
\begin{equation}
\label{eq:left limit}
\lim_{s\uparrow t}\int_{B_R}|\rho_s(x)-\rho_s(x) \,1_{\{ T_{\sXX}> t\}}(\XX(0,s,x)) |\,dx=0.
\end{equation}
For this, we observe that by \eqref{eqn:rho-t-formula} and the incompressibility of the flow, we have that
\begin{equation}
\begin{split}
\int_{B_R}|\rho_s(x)-\rho_s(x)\, 1_{\{ T_{\sXX}> t\}}(\XX(0,s,x)) |\,dx 
&= \int_{B_R}|\rho_s|(x)\,1_{\{ T_{\sXX}\leq t\}}(\XX(0,s,x)) \,dx
\\
&=
\int_{\R^d}|\rho_0|(y)\,1_{\{ T_{\sXX}\leq t\}} (y)\,1_{B_R}(\XX(s,0,y)) \,dy.
\end{split}
\end{equation}
Since trajectories go to infinity when the time approaches $T_{\sXX}$ (see \eqref{eq:blowup}), it follows that
$$
1_{\{ T_{\sXX}\leq t\}} (y)\,1_{B_R}(\XX(s,0,y)) \to 0\quad \text{ for $\Leb{d}$-a.e. $y$ as $s \uparrow t$},
$$
so \eqref{eq:left limit} follows by dominated convergence.
This concludes the proof of (i).

\smallskip
\noindent {\bf Step 3: proof of (ii), renormalization property of $\rho_t$.}
We begin by showing that $\rho_t$ is a renormalized solution of the continuity equation.

By Remark~\ref{rmk:renorm-operativa} it is enough to prove that, given a bounded monotone nondecreasing function $\beta \in C^1(\R)$ 
and $\varphi \in C^{\infty}_c(\R^d)$ nonnegative, the function $t\mapsto\int_{\R^d} \varphi \,\beta (\rho_t)\, dx$ is absolutely
continuous in $[0,T]$ and
 \begin{equation}
 \label{eqn:foliaz-rinorm}
 \begin{split}
\frac{d}{dt} \int_{\R^d} \varphi \,\beta (\rho_t)\, dx
 &= \int_{\R^d} \nabla\varphi\cdot\bb_t\,\beta(\rho_t) \, dx\qquad\text{for $\Leb{1}$-a.e. $t\in (0,T)$.}
\end{split}
\end{equation}
To show that the map is absolutely continuous, let us consider $s,\,t \in [0,T]$
and let $\tilde \rho^t_r$ be the evolution of $\rho_t$ 
through the flow $\XX(\cdot, t,x)$, namely
\begin{equation}
\label{eqn:rho-tilde-t}
\tilde \rho^t_r:= \XX(r,t, \cdot)_\# (\rho_t \res \{T^+_{t,\sXX}>r>T^-_{t,\sXX}\}) \qquad \mbox{for every }r \in [0,T].
\end{equation}
Since $\eeta$ is transported by $\XX$ (by assumption), we claim that
\begin{equation}
\label{eqn:tilderho-rho}
\tilde \rho^t_r \leq \rho_r \qquad \mbox{for every }r\in [0,T].
\end{equation}
Indeed, with the notation of the statement of Theorem~\ref{thm:no-split-until-explosion},
since $\delta_{\sXX(r,t,x)}=(e_r)_\# \eeta^t_x$ for $\rho_t$-a.e. $x\in \{T^+_{t,\sXX}>r>T^-_{t,\sXX}\}$,
 for every $r\in [0,T]$ one has
 \begin{eqnarray*} 
\tilde \rho^t_r\, \Leb{d}&=&  \int_{\{T^-_{t,\sXX}<s\}} \delta_{\sXX(s,t,x)}\,\rho_t(x)\,dx
\leq \int_{\R^d} (e_r)_\# \eeta^t_x\,\rho_t(x)
\,dx\\
&=&(e_r)_\#\int_{\R^d} \eeta^t_x\,\rho_t(x)\,dx=(e_r)_\#\eeta
=\rho_r\,\Leb{d}.
\end{eqnarray*}
Combining \eqref{eqn:tilderho-rho}, the equality $\tilde \rho^t_t= \rho_t$, the monotonicity of $\beta$, and statement (i), we deduce that
\begin{equation}
\label{eqn:AC-sopra}
\int_{\R^d}[ \beta(\rho_t)- \beta(\rho_s)] \,\varphi \, dx 
\leq  \int_{\R^d}[ \beta(\tilde \rho^t_t)- \beta(\tilde \rho^t_s)]\, \varphi \, dx 
=\int_s^t \int_{\R^d} \beta(\tilde \rho^t_r) \,\nabla \varphi \cdot \bb_r \, dx \, dr
\end{equation}
and similarly
$$
\int_{\R^d}[ \beta(\rho_t)- \beta(\rho_s)] \,\varphi \, dx 
\geq  \int_{\R^d}[ \beta(\tilde \rho^s_t)- \beta(\tilde \rho^s_s)] \,\varphi \, dx 
=\int_s^t \int_{\R^d} \beta(\tilde \rho^s_r) \,\nabla \varphi \cdot \bb_r \, dx \, dr.
$$
In particular
$$\Big| \int_{\R^d}[ \beta(\rho_t)- \beta(\rho_s)] \,\varphi \, dx \Big| 
\leq
\|\beta\|_{\infty}  \int_{\R^d} \int_s^t |\nabla \varphi| \,|\bb_r| \, dr \,dx,
$$
which shows that the function $t\mapsto\int_{\R^d} \varphi \beta (\rho_t)\, dx$ is absolutely continuous in $[0,T]$. 

Hence, in order to prove \eqref{eqn:foliaz-rinorm} it suffices to notice that \eqref{eqn:AC-sopra} and the strong continuity of
$r\mapsto\tilde\rho^t_r$ at $r=t$ (ensured by statement (i)) give
$$
\int_{\R^d}[ \beta(\rho_t)- \beta(\rho_s)]\, \varphi \, dx 
\leq (t-s)\int_{\R^d} \beta(\rho_t)\, \nabla \varphi \cdot \bb_t \, dx +o(t-s),
$$
hence \eqref{eqn:foliaz-rinorm} holds at any differentiability point of $t\mapsto \int_{\R^d} \varphi\, \beta (\rho_t)\, dx$, thus  for a.e. $t$.

{
\smallskip
\noindent {\bf Step 4: proof of (ii), strong continuity of $\rho_t$.}
We now show that $\rho_t$ is strongly continuous on $[0,T)$ with respect to the
$L^1_{\rm loc}$ convergence; more precisely we show that, for every $t \in [0,T)$ and for every $r>0$,
\begin{equation}
\label{eqn:strong-cont}
\lim_{s\uparrow t} \int_{B_r}|\rho_s -\rho_t| \, dx =0
\end{equation}
(reversing the time variable, the same argument gives the right-continuity).
To this end, let us define $\tilde \rho^t$ as in \eqref{eqn:rho-tilde-t} for every $t\in[0,T]$; we claim that
\begin{equation}
\label{eqn:rappr-tilde-rho}
\tilde \rho^t_s =\rho_s\res \{T^+_{s,\sXX}>t\} \qquad \text{for every $s\in [0,t]$
}.
\end{equation}
Indeed, let us fix $s,t\in [0,T]$ and $s\leq t$. 
Denoting with $\eeta^t_x$ the disintegration of $\eeta$ with respect to the map $e_t$, recalling that $\eeta^t_x$ is concentrated on curves $\eta \in C([0,T]; \rncp{d})$ with $\eta(t) =x$,  by Theorem~\ref{thm:no-split-until-explosion}
%since $\delta_{\sXX(s,t,x)}=(e_s)_\# \eeta^t_x$ for $\rho_t$-a.e. $x\in \{T^-_{t,\sXX}<s\}$, 
we have that, for $\Leb{d}$-a.e. $x\in \R^d$,
\begin{equation*}
\begin{split}
1_{\{T^-_{t,\sXX}<s\}}(x)\, \delta_{\sXX(s,t,x)} %&=1_{\{T^-_{t,\sXX}<s\}}(x) (e_s)_\# \eeta^t_x\\
&=
(e_s)_\# \Big( \eeta^t_x \res \big\{  \eta  \in C([0,T]; \rncp{d}): \eta(t) =x \mbox{ and } T^-_{t,\sXX}(x)<s\big\} \Big)
\\&=
(e_s)_\# \Big( \eeta^t_x \res \big\{  \eta  \in C([0,T]; \rncp{d}): \eta(t) \neq \infty \mbox{ and } T^-_{t,\sXX}(\eta(t))<s\big\} \Big),
\end{split}
\end{equation*}
hence we can rewrite $\tilde \rho^t_s $ in terms of $\eeta$ as
 \begin{equation}
 \label{eqn:rho-tilde-1}
 \begin{split}
 \tilde \rho^t_s \,\Leb{d}&=  \int_{\{T^-_{t,\sXX}<s\}} \delta_{\sXX(s,t,x)}\,\rho_t(x)\,dx
%= \int_{\{T^-_{t,\sXX}<s\}} (e_s)_\# \eeta^t_x\,\rho_t(x)\,dx
\\
&= \int_{\R^d} (e_s)_\# \Big( \eeta^t_x \res \big\{  \eta  \in C([0,T]; \rncp{d}): \eta(t) \neq \infty \mbox{ and } T^-_{t,\sXX}(\eta(t))<s\big\} \Big)\,\rho_t(x)
\,dx\\
&= (e_{s})_\# \Big( \eeta \res \big\{ \eta  \in C([0,T]; \rncp{d}): \eta(t) \neq \infty \mbox{ and } T^-_{t,\sXX}(\eta(t))<s\big\} \Big).
 \end{split}
\end{equation}
By the semigroup property (Theorem~\ref{thm:maximalflow}(iii)) there exists a set $E_{s,t} \subseteq \R^d$ of $\Leb{d}$-measure $0$ such that 
$$
T^{\pm}_{s,\sXX}(\XX(s,t,x))=T^\pm_{t,\sXX}(x)
\qquad\text{$\forall\,x\in\{T^+_{t,\sXX}> s>T^-_{t,\sXX}\} \setminus E_{s,t}$,}
$$
$$
T^{\pm}_{t,\sXX}(\XX(t,s,x))=T^\pm_{s,\sXX}(x)
\qquad\text{$\forall\,x\in\{T^+_{s,\sXX}> t>T^-_{s,\sXX}\}\setminus E_{s,t}$,}
$$
$$
\XX\bigl(\cdot,s,\XX(s,t,x)\bigr)=\XX(\cdot, t, x)\quad \text{in $(T^-_{t,\sXX}(x),T^+_{t,\sXX}(x))$ }
\qquad\text{$\forall\,x\in\{T^+_{t,\sXX}>s>T^-_{t,\sXX}\}\setminus E_{s,t}$,}
$$
$$
\XX\bigl(\cdot,t,\XX(t,s,x)\bigr)=\XX(\cdot, s, x)\quad \text{in $(T^-_{s,\sXX}(x),T^+_{s,\sXX}(x))$ }
\qquad\text{$\forall\,x\in\{T^+_{s,\sXX}>t>T^-_{s,\sXX}\}\setminus E_{s,t}$.}
$$
Since $(e_s)_\#\eeta \res \R^d$ is absolutely continuous with respect to $\Leb{d}$ (hence the set of curves $\eta$ such that $\eta(s)\in E_{s,t}$ is $\eeta$-negligible) and $\eeta$ is transported by the maximal regular flow, we have the following equalities, which hold up to a set of $\eeta$-measure $0$:
\begin{equation}
\begin{split}
& \big\{ \eta  \in C([0,T]; \rncp{d}): \eta(s) \neq \infty \mbox{ and } T^+_{s,\sXX}(\eta(s))>t \big\}
\\&=\big\{ \eta  \in C([0,T]; \rncp{d}): \eta(s) \neq \infty,\; \eta(s) \notin E_{s,t}, \; T^+_{s,\sXX}(\eta(s))>t
\\&\hspace{10em} \mbox{ and }\eta(\cdot) = \XX(\cdot, s,\eta(s)) \mbox{ in } (T^-_{s,\sXX}(\eta(s) ), T^+_{s,\sXX}(\eta(s)) \big\}
\\&=\big\{ \eta  \in C([0,T]; \rncp{d}): \eta(t) \neq \infty,\; \eta(t) \notin E_{s,t}, \; T^-_{t,\sXX}(\eta(t))<s
\\&\hspace{10em} \mbox{ and }\eta(\cdot) = \XX(\cdot, t,\eta(t)) \mbox{ in } (T^-_{t,\sXX}(\eta(t)) , T^+_{t,\sXX}(\eta(t)) \big\}
\\&=\big\{ \eta  \in C([0,T]; \rncp{d}): \eta(t) \neq \infty \mbox{ and } T^-_{t,\sXX}(\eta(t))<s\big\}
.\end{split}
\end{equation}
This implies that
\begin{equation*}
\begin{split}
\rho_s\res \{T^+_{s,\sXX}>t\}
&= (e_{s})_\# \Big( \eeta \res \big\{ \eta  \in C([0,T]; \rncp{d}): \eta(s) \neq \infty \mbox{ and } T^+_{s,\sXX}(\eta(s))>t \big\} \Big)
\\&= (e_{s})_\# \Big( \eeta \res \big\{ \eta  \in C([0,T]; \rncp{d}): \eta(t) \neq \infty \mbox{ and } T^-_{t,\sXX}(\eta(t))<s\big\} \Big),
\end{split}
\end{equation*}
that combined with \eqref{eqn:rho-tilde-1} gives \eqref{eqn:rappr-tilde-rho}.

Now, in order to prove \eqref{eqn:strong-cont}, we apply the triangular inequality to infer that
$$ \int_{B_r}|\rho_s -\rho_t| \, dx  \leq \int_{B_r}|\rho_s -\tilde \rho^t_s| \, dx + \int_{B_r}| \tilde\rho^t_s -\rho_t| \, dx .$$
The second term in the right-hand side converges to $0$ when $s\uparrow t$ by the strong $L^1_{\rm loc}$ continuity of $\rho^t_s$ with respect to $s$ proved in statement (i). To see that also the first term converges to $0$, we use \eqref{eqn:rappr-tilde-rho}, the identity $\rho_t \Leb{d}= (e_t)_\#\eeta \res \R^d$, and the fact that $\eeta$ is transported by the maximal flow, to obtain
\begin{equation*}
\begin{split}
\int_{B_r}|\rho_s& -\tilde \rho^t_s| \, dx
= \int_{B_r}\rho_s \,1_{\{T^+_{s,\ssXX}\leq t\}} \, dx= \int 1_{B_r \cap \{T^+_{s,\ssXX}\leq t\}}(\eta(s)) \, d\eeta(\eta)
\\ &=\eeta \Big( \big\{\eta: \eta(s) \in B_r \cap \{ T^+_{s,\ssXX}\leq t\} \mbox{ and } \eta(\cdot) = \XX(\cdot, s,\eta(s)) \mbox{ in } [s , T^+_{s,\sXX}(\eta(s))  \big\} \Big).
\end{split}
\end{equation*}
Notice that, if $\eta$ is a curve which belongs to the set in the last line, then it belongs to $B_r$ at time $s$ and blows up in $[s,t]$, thus
\begin{equation*}
\int_{B_r}|\rho_s - \rho^t_s| \, dx \leq \eeta \Big( \big\{\eta: \text{$\eta(s') \in B_r$ and $\eta(s'')= \infty$ for some $s', s'' \in [s,t]$} \big\} \Big).
\end{equation*}
Since set in the right-hand side monotonically decreases to the empty set as $s\uparrow t$, its $\eeta$-measure converges to $0$, which proves \eqref{eqn:strong-cont}
and concludes the proof.
}
\end{Proof}

{
We now discuss a general no blow-up criterion for a generalized flow $\eeta$. This result plays an important role in the proof of
Corollary~\ref{cor:vp-bounded1}.
\begin{proposition}[No blow-up criterion]\label{prop:no-blow-up}
Let $\bb \in L^1_{\rm loc}([0,T] \times \R^d;\R^d)$ be a Borel vector field, let $\eeta \in \Measuresp{C([0,T];\rncp{d})}$ be a generalized flow of $\bb$, and for $t \in [0,T]$ let $\mu_t := (e_t)_\# \eeta \res \R^d$. Let $\eta_\infty$ denote the constant curve $\eta \equiv \infty$, and assume that $\eeta(\{\eta_\infty\})=0$ and
\begin{equation}
\label{eqn:no-blow-up-log}
\int_0^T \int_{\R^d} \frac{|\bb_t|(x)}{(1+|x|) \log(2+|x|)}\, d\mu_t(x)\, dt <\infty.
\end{equation}
Then $\eeta$ is concentrated on curves that do not blow up, namely
$$\eeta\big( \{ \eta \in C([0,T]; \rncp{d})): \eta(t) =\infty \mbox{ for some } t\in[0,T]\} \big) = 0.$$
In particular, if we assume that $\mu_t\ll\Leb{d}$ for every $t \in [0,T]$ and that $\eeta$ is concentrated on the maximal regular flow $\XX$ associated to $\bb$, then $\XX$ is globally defined on $[0,T]$ for $\mu_0$-a.e. $x$, namely the trajectories
$\XX(\cdot,x)$ belong to $AC([0,T];\R^{d})$ for $\mu _0$-a.e. $x\in\R^{d}$.
\end{proposition}
\begin{proof}
Since $\eeta(\{\eta_\infty\})=0$ we know that $\eeta$-a.e. curve is finite at some time.
In particular, if we fix a dense set of rational times $\{t_n\}_{n \in \N}\subset [0,T]$,
we see that (by continuity of the curves) $\eeta$ is concentrated on $\cup_{n \in \N}\Gamma_n$ with
$$
\Gamma_n:=\{ \eta \in C([0,T]; \rncp{d})): \eta(t_n) \in \R^d \},
$$
so it is enough to show that $\eeta\res{\Gamma_n}$ is concentrated on curves that do not blow up.

By applying Theorem \ref{thm:no-split-until-explosion} with $s=t_n$
it follows that $\eeta\res{\Gamma_n}$ is concentrated on curves $\eta$ that are finite on the time interval 
$(T^-_{t_n,\sXX}(\eta(t_n)),T^+_{t_n,\sXX}(\eta(t_n)))\subset [0,T]$.
Hence, 
since $(e_t)_\#(\eeta\res{\Gamma_n})\leq \mu_t$,
by Fubini theorem and assumption \eqref{eqn:no-blow-up-log} we get 
\begin{equation*}
\begin{split}
\int \int_{T^-_{t_n,\sXX}(\eta(t_n))}^{T^+_{t_n,\sXX}(\eta(t_n))} & \Big|\frac{d}{dt} \bigl[\log \log(2+|\eta(t)|)\bigr]\Big| \, dt \, d[\eeta\res{\Gamma_n}](\eta)\\
&\leq \int \int_{T^-_{t_n,\sXX}(\eta(t_n))}^{T^+_{t_n,\sXX}(\eta(t_n))} \frac{|\dot \eta(t)|}{(1+|\eta(t)|) \log(2+|\eta(t)|)}\,dt \, d[\eeta\res{\Gamma_n}](\eta)
\\&= \int \int_{T^-_{t_n,\sXX}(\eta(t_n))}^{T^+_{t_n,\sXX}(\eta(t_n))} \frac{|\bb_t|( \eta(t))}{(1+|\eta(t)|) \log(2+|\eta(t)|)} \,dt \, d[\eeta\res{\Gamma_n}](\eta)
\\& \leq \int_0^T \int_{\R^d} \frac{|\bb_t|(x)}{(1+|x|) \log(2+|x|)}\, d\mu_t(x)\, dt <\infty.
\end{split}
\end{equation*}
This implies that, for $\eeta$-a.e. curve $\eta \in \Gamma_n$,
\begin{multline*}
\sup_{T^-_{t_n,\sXX}(\eta(t_n))\leq s<\tau \leq T^+_{t_n,\sXX}(\eta(t_n))} \big|\log \log(2+|\eta(s)|) -\log \log(2+|\eta(\tau)|)\big|\\
\leq \int_{T^-_{t_n,\sXX}(\eta(t_n))}^{T^+_{t_n,\sXX}(\eta(t_n))} \Big|\frac{d}{dt} \bigl[\log \log(2+|\eta(t)|)\bigr]\Big| \, dt < \infty,
\end{multline*}
which in turn says that $T^-_{t_n,\sXX}(\eta(t_n))=0$, $T^+_{t_n,\sXX}(\eta(t_n))=T$, and the curve $\eta$ cannot blow up in $[0,T]$, as desired.

To show the second part of the statement, le us consider the disintegration of $\eeta$ with respect to $e_0$. By the properties of $\eeta$ we have that, for $\mu_0$-a.e. $x$, the probability measure $\eeta_x$ is concentrated on the set
$$\big\{ \eta: \eta(0)= x,\; \eta \neq \infty \mbox{ in }[0,T], \eta = \XX(\cdot, x) \mbox{ in }[0,T_\sXX(x)) \big\}.$$
Since $\eeta_x$ is a probability measure it follows that this set is nonempty, that $T_\sXX(x) =T$, and this set has to coincide with $\{ \XX(\cdot, x)\}$, thus $\eeta_x = \delta_{\sXX(\cdot, x)}$, as desired.
\end{proof}
}

\section{The superposition principle under local integrability bounds}\label{sec:repres-flow}
In order to represent the solution to the continuity equation by means of a generalized flow we would like to apply the so-called superposition principle (see  \cite[Theorem~12]{bologna} or \cite[Theorem~2.1]{amcofi}). However, the lack of global bounds makes this approach very difficult to implement. An analogue of the classical superposition principle is the content of the following theorem.

\begin{theorem}[Extended superposition principle]\label{thm:repr-foliat} Let $\bb \in L^1_{\rm loc}([0,T] \times \R^d;\R^d)$ be a Borel vector field,
and let $\rho_t\in L^\infty((0,T); L^1_+(\R^d))$ be a distributional solution of the continuity equation, weakly continuous
in duality with $C_c(\R^d)$. Assume that:\\
(i) either $|\bb_t|\rho_t \in L^1_{\rm loc}([0,T]\times\R^d)$;\\
(ii) or $\div \bb_t=0$ and $\rho_t$ is a renormalized solution.\\
Then there exists $\eeta \in \Measuresp{C([0,T]; \rncp{d})}$, concentrated on the set $\Gamma$ defined in \eqref{defn:gen-flow},
which satisfies
$$|\eeta|(C([0,T]; \rncp{d})) \leq \sup\limits_{t\in [0,T]} \|\rho_t\|_{L^1(\R^d)}$$
and 
$$(e_t)_\# \eeta \res \R^d= \rho_t\Leb{d} \qquad \mbox{for every }t\in [0,T].$$
In addition, if $\rho_t$ belongs also to $L^\infty((0,T); L^\infty_+(\R^d))$ (or $\rho_t$ is a renormalized solution)
and $\bb$ is divergence-free and satisfies {\bf (A1)}-{\bf (A2)} of Section~\ref{sec:maxflow}, then $\eeta$ is transported by the Maximal Regular Flow of $\XX$ of $\bb$.
\end{theorem}

\begin{remark} {\rm 
Noticing that the assumption $|\eeta|(C([0,T]; \rncp{d})) \leq \sup_{t\in [0,T]}\mu_t(\R^d)$ implies that the curve $\eta \equiv \infty$ has $\eeta$-measure $0$, 
if we assume that
\begin{equation}
\label{hp:int-b}
\int_0^T \int_{\R^d} \frac{|\bb_t|(x)}{(1+|x|) \log(2+|x|)}\,\rho_t(x)\,dx\,dt <\infty
\end{equation}
then 
it follows by Proposition \ref{prop:no-blow-up} that $\rho_t$ is transported by the Maximal Flow, namely 
$T_{0,\sXX}^+(x)=T$, $\XX(\cdot,0,x)\in AC([0,T];\R^d)$ for a.e. $x\in\{\rho_0>0\}$, and $\rho_t\Leb{d} = \XX(t,\cdot)_\#\rho_0\Leb{d}$.
%
%Indeed, by Theorem~\ref{thm:noblowup} and \eqref{hp:int-b} we know that the Maximal Regular Flow is well defined in $[0,T]$ for a.e. $x\in \R^d$. Since $\eeta$ is transported by $\XX$, for $\eeta$-a.e. $\eta$ we know that $\eta = \XX( \cdot,0, \eta(0))$ in $[0,T]$. This implies that for a.e. $x\in \{\rho_0>0\}$ 
%the measure $\eeta_x$, obtained through disintegration of $\eeta$ with respect to $e_0$, coincides with $\delta_{\sXX( \cdot,0, x)}$, therefore 
%$$
%(e_t)_\# \eeta = \int_{\R^d} (e_t)_\# \eeta_x \rho_0(x)\, dx =\int_{\R^d} (e_t)_\# \delta_{\sXX( \cdot, 0,x)}\rho_0(x)\, dx = \XX( \cdot,0, x)_\# \rho_0\Leb{d},
%$$
%as desired.
}\end{remark}

Let us first briefly explain the idea behind the proof of the theorem above.
To overcome the lack of global bounds on $\bb$ we introduce a kind of ``damped'' stereographic projection, with a damping depending on the growth of $|\bb|$ at $\infty$, and
we look at the flow of $\bb$ on the $d$-dimensional sphere
$\sfera{d}$ in such a way that the north pole $N$ of the sphere corresponds to the points at infinity of $\R^d$. Then we apply the superposition principle in these new variables and eventually, going back to the original variables, we  obtain a representation of the solution as a generalized flow. Let us observe that it is crucial for us that the map sending $\R^d$ onto $\sfera{d}$ is chosen a function of $\bb$: indeed, as we shall see, by shrinking enough distances at infinity we can ensure that the vector field read on the sphere becomes \emph{globally} integrable.

We denote by $N$ be the north pole of the $d$-dimensional sphere 
$\sfera{d}$, thought of as a subset of $\R^{d+1}$. For our constructions, we will use a smooth diffeomorphism which 
maps $\R^d$ onto $\sfera{d} \setminus \{N\}$ and whose derivative has a prescribed decay at $\infty$.

\begin{lemma}\label{lemma:decaying-diffeo}
Let $D: [0,\infty) \to (0,1]$ be a monotone nonincreasing function. 
Then there exist $r_0>0$ and a smooth diffeomorphism $\psi:\R^d \to \sfera{d}\setminus\{N\} \subset \R^{d+1}$ such that 
\begin{equation}\label{eqn:psi-cont-1}
\text{$\psi(x)\rightarrow N$ as $|x|\to\infty$,}
\end{equation}
\begin{equation}\label{eqn:psi-lip}
|\nabla \psi(x)| \leq D(0) \qquad \forall \,x\in \R^d,
\end{equation}
\begin{equation}\label{eqn:psi-dec}
|\nabla \psi(x)| \leq D(|x|) \qquad \forall \,x\in \R^d\setminus B_{r_0}.
\end{equation}
\end{lemma}
\begin{Proof} We split the construction in two parts: first we perform a $1$-dimensional construction, and then we use this 
construction to build the desired diffeomorphism.

\smallskip

\noindent {\bf Step 1: $1$-dimensional construction.}
Let $D_0: [0,\infty) \to (0,1]$ be a monotone nonincreasing function. 
We claim that there exists a smooth diffeomorphism $\psi_0 : [0,\infty) \to [0,\pi)$ such that 
\begin{equation}\label{eqn:psi-cont-2}
\lim_{r\to\infty} \psi_0(r) = \pi,\qquad \lim_{r\to\infty} \psi_0'(r) = 0,
\end{equation}
\begin{equation}\label{eqn:psi-vicino-a-0}
\text{$\psi_0(r) = c_0\,D_0(0)\,r$\qquad $\forall\,r\in [0, \pi/D_0(0))$}, \quad \mbox{for some }c_0\in(0,1),
\end{equation}
\begin{equation}\label{eqn:psi-lip-0}
|\psi_0'(r)| \leq {D_0}(0) \qquad \forall \,r\in [0,\infty),
\end{equation}
\begin{equation}\label{eqn:psi-dec-0}
{|\psi_0'(r)| \leq {D_0}(r) \qquad \forall \,r\in [ 2\pi /D_0(0), \infty).}
\end{equation}
Indeed,
define the nonincreasing $L^1$ function $D_1:[0,\infty)\to (0,\infty)$ as 
$$D_1(r) :=
\begin{cases}
D_0(0) \qquad& \mbox{if } r \in [0,  1+\pi/D_0(0)]
\\
\min\{D_0(r), r^{-2}\} \qquad &
\mbox{if } r \in (1+\pi/D_0(0), \infty).
 \end{cases}
$$ 
We then consider an asymmetric convolution kernel, namely a nonnegative function $\sigma \in C^\infty_c((0,1))$ with $\int_\R \sigma = 1$,
and consider the convolution of $D_1(r)$ with $\sigma (-r)$:
$$\psi_1(r) := \int_0^1 \sigma(r') D_1(r+r') \, dr' \qquad \forall \,r\in [0,\infty).$$
Notice that $\psi_1$ is smooth on $(0,\infty)$, positive, nonincreasing, and $\psi_1 \leq D_1$ in $[0,\infty)$. In particular $\psi_1\in L^1((0,\infty))$. 
Moreover we have that $\psi_1\equiv D_0(0)$ in $[0, \pi /D_0(0)]$, hence $\|\psi_1\|_{L^1((0,\infty))} \geq \pi$ and
$c_0:= \pi\|\psi_1\|_{L^1((0,\infty))}^{-1} \in(0,1)$.
Finally, we define $\psi_0$ as
$$
\psi_0(r):=c_0\int_0^r \psi_1(s)\,ds\qquad \forall \,r\in [0,\infty).
$$
Since $|\psi_0'(r)| = c_0 |\psi_1(r)| \leq D_1(r)$, taking into
account that $\pi/D_0(0)>1$ it is easy to check that all the desired properties
are satisfied.
\smallskip

\noindent {\bf Step 2: ``radial'' diffeomorphism in any dimension.}
Let $D_0: [0,\infty) \to (0,1]$ to be chosen later and consider $\psi_0$ and $c_0$ as in Step 1. We define $\psi:\R^d \to \sfera{d}\setminus\{N\} \subset\R^{d+1}$ which maps every half-line starting at the origin to an arc of sphere between the south pole and the north pole:
$$\psi(x) := \sin(\psi_0(|x|)) \Big(\frac{x}{|x|},0\Big)- \cos(\psi_0(|x|)) \bigl(0,\ldots,0,1\bigr).$$
Thanks to \eqref{eqn:psi-vicino-a-0} and to the fact that the functions $x \mapsto |x|^2$, $t\mapsto \sin( \sqrt t )/\sqrt t$,  and $t\mapsto \cos( \sqrt t)$ are all of class $C^\infty$,  we obtain that $\psi\in C^\infty(\R^d; \R^{d+1})$. We also notice that its inverse $\phi:\sfera{d} \setminus\{N\}\to\R^d$ can be explicitly computed: 
\begin{align*}
\phi(x_1,\ldots, x_{d+1}) &= 
\psi_0^{-1} (\arccos(-x_{d+1})) \,\frac{(x_1,\ldots,x_d)}{|(x_1,\ldots,x_d)|}\\
& = 
\psi_0^{-1} (\arcsin(|(x_1,\ldots,x_d)|)) \,
\frac{(x_1,\ldots,x_d)}{|(x_1,\ldots,x_d)|}.
\end{align*}
Writing $r=|x|$ and denoting by $I_d$ the identity matrix on the first $d$ components, we compute the gradient of $\psi$:
\begin{multline*}
\nabla \psi(x) =
\frac{\cos(\psi_0(r))\, \psi_0'(r)\, r- \sin(\psi_0(r))}{r^3} \,(x,0) \otimes (x,0)
+ \frac{\sin(\psi_0(r))}{r} \,I_d 
\\
-\frac{ \sin(\psi_0(r)) \,\psi_0'(r)}{r}\, (x,0) \otimes (0,\ldots,0,1).
\end{multline*}
It is immediate to check that $|\nabla \psi(x)|\neq 0$ for all $x \in \R^d$, so it follows by the Inverse Function Theorem that $\phi$
is smooth as well.
Also, we can estimate
\begin{equation}\label{eqn:nabla-psi-est}
|\nabla \psi(x)| \leq 2\, |\psi'_0(r)|+ 2\, \frac{\sin( \psi_0(r))}{r}.
\end{equation}
Using now \eqref{eqn:psi-lip-0} and \eqref{eqn:psi-dec-0}, the first term in the right hand side above can be estimated with $2 D_0(0)$ for every $x\in \R^d$, and with 
$2D_0(r)$ for every $x\in \R^d$ such that $r=|x|\geq 2\pi/D_0(0)$. As regards the second term, for $r\in [0, \pi/D_0(0)]$ we have that 
\begin{equation}\label{eqn:1-int-est}
\frac{\sin( \psi_0(r))}{r} = \frac{\sin( c_0 \,D_0(0)\, r)}{r} \leq c_0 \,D_0(0),
\end{equation}
while for $r\in [\pi/D_0(0),\infty)$ we estimate the numerator with $1$ to get 
\begin{equation}\label{eqn:2-int-est}
\frac{\sin( \psi_0(r))}{r} \leq \frac{D_0(0)}{\pi}.
\end{equation}
Therefore,  since $c_0<1$, by \eqref{eqn:nabla-psi-est}, \eqref{eqn:1-int-est}, and \eqref{eqn:2-int-est} we get
\begin{equation}
\label{eqn:psi-lip0}
|\nabla \psi(x)| \leq 4 \,D_0(0) \qquad \forall \,x\in \R^d.
\end{equation}
Now,  for $r \in [2\pi/D_0(0),\infty)$, thanks to \eqref{eqn:psi-cont-2} and \eqref{eqn:psi-dec-0}
we can estimate
\begin{equation}\label{eqn:3-int-est}
\frac{\sin( \psi_0(r))}{r} 
= \frac{1}{r} \int_{r}^\infty - \cos(\psi_0(s))\, \psi'_{0}(s) \, ds 
\leq \frac{1}{r} \int_{r}^\infty |\psi'_{0}(s)| \, ds 
\leq \frac{1}{r}\int_{r}^\infty D_0(s) \, ds,
\end{equation}
thus by \eqref{eqn:psi-dec-0}, \eqref{eqn:nabla-psi-est}, and  \eqref{eqn:3-int-est}, we obtain
\begin{equation}
\label{eqn:psi-dec0}|\nabla \psi(x)| \leq 2 \,D_0(r) + \frac{2}{r} \int_r^\infty D_0(s) \, ds 
 \qquad  \forall \,x\in \R^d\setminus B_{2\pi/D_0(0)}.
\end{equation}
So,
provided we choose
$D_0(r) := \min\{4^{-1}, r^{-2}\} \,D(r)$ we obtain that \eqref{eqn:psi-lip0} implies \eqref{eqn:psi-lip}. Also,
by choosing $r_0 := 2\pi/D_0(0)>2$, from \eqref{eqn:psi-dec0} and because $D$ is monotone nonincreasing we deduce that
$$|\nabla \psi(x)| \leq \frac{ D(r)}{2} + \frac{1}{r} \int_r^\infty \frac{ {D(r)}}{s^2} \, ds 
 \leq \frac{D(r)}{2} + \frac{D(r)}{r^2} \leq D(r) 
 \qquad  \forall \,x\in \R^d\setminus B_{r_0},$$
proving \eqref{eqn:psi-dec} and concluding the proof.
\end{Proof}

\begin{proof}[Proof of Theorem~\ref{thm:repr-foliat}]
We first assume that $|\bb_t|\rho_t \in L^1_{\rm loc}([0,T]\times\R^d)$ and we prove the result in this case.
This is done in two steps:\\
- In Step 1, based on Lemma~\ref{lemma:decaying-diffeo}, we construct a diffeomorphism between $\R^d$ 
and $\sfera{d}\setminus\{N\}$ with the property that the vector field $\bb$, read on the sphere, becomes globally integrable.\\
- In Step 2 we associate a solution of the continuity equation on the sphere to the solution of the continuity equation $\rho_t$; 
this is done by adding a time-dependent mass in the north pole. Then the classical superposition principle applies on the sphere,
and this implies the desired superposition result for $\rho_t$.

Once the theorem has been proved for $|\bb_t|\rho_t \in L^1_{\rm loc}([0,T]\times\R^d)$, we show in Step 3 how
to handle the case when $\rho_t$ is a renormalized solution.

Finally,  in Step 4 we exploit the results of Section~\ref{sect:flow} to show that $\rho_t$ is transported by the Maximal Regular Flow.

\smallskip

\noindent {\bf Step 1: construction of a diffeomorphism between $\rncp{d}$ and $\sfera{d}$.}
We build a diffeomorphism $\psi\in C^\infty(\R^d;\sfera{d}\setminus\{N\})$ such that
\begin{equation}\label{eqn:psi-cont}
\lim_{x\to\infty} \psi(x) = N,
\end{equation}
\begin{equation}\label{eqn:diffeom-prop}
\int_0^T\int_{\R^d} |\nabla\psi(x)|\,|\bb_t(x)| \,\rho_t(x)\, dx\, dt <\infty.
\end{equation}
To this end, we apply Lemma~\ref{lemma:decaying-diffeo} with $D(r)=1$ in $[0,1)$ and $D(r) =(2^nC_n)^{-1}$ for $r\in [2^{n-1},2^n)$, where
$$
C_n:= 1+ \int_0^T\int_{B_{2^n}} |\bb_t(x)|\,\rho_t(x)\, dx\, dt \qquad \text{for every $n\in \N$}.
$$
In this way we obtain a smooth diffeomorphism $\psi$ which maps  $\R^d$ onto $\sfera{d} \setminus \{N\}$ 
such that \eqref{eqn:psi-cont} holds, $|\nabla\psi (x) | \leq 1$ on $\R^d$, and
\begin{equation}\label{eqn:psi-decade}
|\nabla \psi (x) | \leq {\frac{1}{2^nC_n} }\qquad \forall \,x\in B_{2^n} \setminus B_{2^{n-1}}, \; n \geq n_0,
\end{equation}
for some $n_0>0$.
Thanks to these facts we deduce that
\begin{equation}
\begin{split}
&\int_0^T\int_{\R^d} |\nabla\psi(x)|\,|\bb_t(x)|\, \rho_t(x)\, dx\, dt
\\
&\leq \int_0^T\int_{B_{2^{n_0}}} |\bb_t(x)|\, \rho_t(x)\, dx\,dt +
\sum_{i=n_0+1}^\infty \int_0^T\int_{B_{2^i}\setminus B_{2^{i-1}}}|\nabla\psi(x)|\, |\bb_t(x)| \,\rho_t(x)\, dx \, dt
\\
&\leq \int_0^T\int_{B_{2^{n_0}}} |\bb_t(x)| \,\rho_t(x)\, dx\,dt +
\sum_{i=n_0+1}^\infty \frac{1}{2^i}
 <\infty,
\end{split}
\end{equation}
which proves \eqref{eqn:diffeom-prop}.

\smallskip

\noindent {\bf Step 2: superposition principle on the sphere.}
We build $\eeta \in \Measuresp{C([0,T]; \rncp{d})}$ such that $|\eeta|(C([0,T]; \rncp{d})) \leq \sup_{t\in [0,T]} \|\rho_t\|_{L^1(\R^d)}$, $\eeta$ is concentrated on 
curves $\eta$ which are locally absolutely continuous integral curves of $\bb$ in $\{\eta\neq\infty\}$,
and whose marginal at time $t$ in $\R^d$ is $\rho_t\Leb{d}$.

Without loss of generality, possibly dividing every $\rho_t$ by $\sup_{t\in [0,T]} \|\rho_t\|_{L^1(\R^d)}$, we can assume that
$\sup_{t\in [0,T]} \|\rho_t\|_{L^1(\R^d)}=1$. Define 
$m_t:=\|\rho_t\|_{L^1(\R^d)}\leq 1$,
\begin{equation}\label{eqn:deftildeb}
\cc_t(y):=
\begin{cases}
\nabla\psi(\phi(y))\,\bb_t(\phi(y)) &\text{if $y\in\sfera{d}\setminus\{N\}$}\\
0 &\text{if $y=N$}
\end{cases}
\end{equation}
and
$$ \mu_t := \psi_\# (\rho_t\Leb{d}) + (1-m_t)\, \delta_N \in \Probabilities{\sfera{d}}, \qquad t\in [0,T].$$
Since $\cc_t(N)=0$ we can neglect the mass at $N=\psi(\infty)$ to get
\begin{eqnarray}
\nonumber
\int_0^T\int_{\sfera{d}} |\cc_t| \, d \mu_t \,dt &=& 
\int_0^T\int_{\sfera{d}\setminus\{N\}} |\nabla\psi|(\phi(y))\,|\bb_t|(\phi(y))\, d \mu_t(y)\,dt \\
\nonumber
&=& \int_0^T\int_{\R^d} |\nabla\psi|(x)\,|\bb_t|(x)\,\rho_t(x)\, dx\, dt <\infty,
\label{eqn:conto}
\end{eqnarray}
where in the last inequality we used \eqref{eqn:diffeom-prop}. 

We now show that the probability measure $\mu_t$ is a solution to the continuity equation on $\sfera{d} \subset \R^{d+1}$ with vector field
$\cc_t$.
To this end we first notice that, by the weak continuity in duality with $C_c(\R^d)$ of $\rho_t$ and by the fact that all the measures $\mu_t$ have unit mass,
we deduce that $\mu_t$ is weakly continuous in time.
Indeed, any limit point of $\mu_s$ as $s\to t$ is uniquely determined on $\sfera{d}\setminus\{N\}$, and then the mass normalization gives
that it is completely determined.
We want to prove that the function $t\mapsto \int_{\sfera{d}} \varphi \, d\mu_t$ is absolutely continuous and satisfies
\begin{equation}
\label{eqn:tilde-cont-eq}
\frac{d}{dt} \int_{\sfera{d}} \varphi \, d \mu_t  = 
 \int_{\sfera{d}} \cc_t \cdot \nabla \varphi  \, d\mu_t\qquad \text{a.e. on $(0,T)$}
\end{equation}
for every $\varphi \in C^\infty(\R^{d+1})$.
We remark that, since $\rho_t$ is a solution to the continuity equation in $\R^d$ with vector field $\bb_t$, changing variables with the diffeomorphism $\psi$ we obtain that
\eqref{eqn:tilde-cont-eq} holds for every $\varphi \in C^\infty_c(\R^{d+1} \setminus \{N\})$, hence we are left to check that \eqref{eqn:tilde-cont-eq} holds also when
 $\varphi$ is not necessarily $0$ in a neighborhood of the north pole.

Fix $\varphi \in C^\infty(\R^{d+1})$. Since $\mu_t(N)= 1- m_t = 1-\mu_t(\sfera{d}\setminus\{N\})$, for every $t \in [0,T]$ we have that
\begin{equation}
\label{eqn:cont-eq-rewritt}
\int_{\sfera{d}} \varphi \, d\mu_t = \int_{\sfera{d}\setminus\{ N\} } \varphi \, d\mu_t + 
\varphi(N)\, \mu_t(N) =\varphi(N)+ \int_{\sfera{d} } (\varphi- \varphi(N)) \, d\mu_t.
\end{equation}
Now, given $\eps>0$ let us consider a function $\chi_\eps \in C^\infty(\R^{d+1})$ which is $0$ in $B_{\eps}(N)$, $1$ outside $B_{2\eps}(N)$, and whose gradient is bounded by $2/\eps$. 
Since $\rho_t$ is a solution to the continuity equation in $\R^d$ and since $\chi_\eps \,(\varphi - \varphi(N))$ 
is a smooth, compactly supported function in $C^\infty_c(\R^{d+1} \setminus \{N\})$ we deduce that
\begin{equation}
\label{eqn:cont-eq-sfera-eps}
\begin{split}
\frac{d}{dt} \int_{\sfera{d}} \chi_\eps\, (\varphi- \varphi(N)) \, d\mu_t
&= \int_{\sfera{d} \setminus \{N\} } \cc_t \cdot \nabla [\chi_\eps\, (\varphi- \varphi(N))]  \, d\mu_t
\\
&= \int_{\sfera{d} \setminus \{N\} } (\varphi- \varphi(N))\,\cc_t \cdot \nabla \chi_\eps   \, d\mu_t
+
\int_{\sfera{d} \setminus \{N\} } \chi_\eps \,\cc_t \cdot \nabla\varphi  \, d\mu_t.
\end{split}
\end{equation}
To estimate the first term in the right-hand side of \eqref{eqn:cont-eq-sfera-eps} we use that $|\varphi- \varphi(N)| \leq \eps \| \nabla \varphi\|_{\infty}$ in $B_{\eps}(N)$ and that $|\nabla \chi_\eps |\leq 2/\eps$ to get that
$$ \bigg| \int_{\sfera{d} \setminus \{N\} } \cc_t \cdot \nabla \chi_\eps\, (\varphi- \varphi(N))  \, d\mu_t
\bigg|\leq
2\,\|\nabla \phi\|_{\infty} \int_{B_{2\eps}(N) \setminus B_{\eps}(N)}| \cc_t| \, d\mu_t,
$$
and notice the latter goes to $0$ in $L^1(0,T)$ as $\eps\to 0$ since $|\cc|$ is integrable with respect to $\mu_t\,dt$ in space-time thanks to \eqref{eqn:conto}.
Since the second term in the right-hand side of \eqref{eqn:cont-eq-sfera-eps} converges in $L^1(0,T)$ to 
$\int_{\sfera{d} \setminus \{N\} }\cc_t \cdot \nabla\varphi  \, d\mu_t$, taking the limit as $\eps \to 0$ in 
\eqref{eqn:cont-eq-sfera-eps} we obtain that $t\mapsto \int_{\sfera{d}} (\varphi- \varphi(N)) \, d\mu_t$
is absolutely continuous in $[0,T]$ and that
for a.e. $t\in (0,T)$ one has
$$\frac{d}{dt} \int_{\sfera{d}} (\varphi- \varphi(N)) \, d\mu_t  = 
 \int_{\sfera{d}} \cc_t \cdot \nabla \varphi  \, d\mu_t.
$$
Using the identity \eqref{eqn:cont-eq-rewritt}, this 
formula can be rewritten in the form \eqref{eqn:tilde-cont-eq}, as desired.
\smallskip

Since $\mu_t$ is a weakly continuous solution of the continuity equation and the integrability condition \eqref{eqn:conto} holds, 
we can apply the superposition principle (see  \cite[Theorem~12]{bologna} or \cite[Theorem~2.1]{amcofi}) to deduce the existence of 
a measure
$\ssigma \in \Probabilities{C([0,T]; \sfera{d})}$ which is concentrated on integral curves of $\cc$ and such that $(e_t)_\# \ssigma =  \mu_t$
for all $t\in [0,T]$.

We then consider $\phi: \sfera{d} \to \rncp{d}$ to be the inverse of $\psi$ extended to $N$ as $\phi(N)= \infty$,
and define $\Phi: C([0,T]; \sfera{d})\to C([0,T]; \rncp{d})$ as $\Phi(\eta) :=  \phi \circ \eta$.
Then the measure 
$$\eeta := \Phi_\# \ssigma \in \Probabilities{C([0,T]; \rncp{d})}$$
is concentrated on locally absolutely continuous integral curves of $\bb$ in the sense stated in \eqref{defn:gen-flow}, and 
$$(e_t)_\# \eeta \res \R^d = \phi_\# (e_t)_\# \ssigma \res \R^d = \phi_\# \mu_t \res \R^d = \rho_t\Leb{d} .$$

\smallskip

\noindent {\bf Step 3: the case of renormalized solutions.}
We now show how to prove the result when $\div \bb_t=0$ and $\rho_t$ is a renormalized solution.
Notice that in this case we have no local integrability information on $|\bb_t|\rho_t$, so the argument above does not apply. 
However, exploiting the fact that $\rho_t$ is renormalized we can easily reduce to that case.

More precisely, we begin by observing that, by a simple approximation argument,
the renormalization property (see Definition \ref{def:renormalized}) is still true when $\beta$ is a bounded Lipschitz function.
Thanks to this observation we consider, for $k \geq 0$, the functions
$$
\beta_k(s):=\left\{
\begin{array}{ll}
0 &\text{if $s \leq k$},\\
s-k &\text{if $k \leq s \leq k+1$},\\
1 &\text{if $s \geq k+1$}.
\end{array}
\right.
$$
Since $\rho_t$ is renormalized, $\beta_k(\rho_t)$ is a bounded distributional solution of the continuity equation,
hence by Steps 1-2 above there exists a measure
$\eeta_k \in \Measuresp{C([0,T]; \rncp{d})}$ with $$|\eeta_k|(C([0,T]; \rncp{d})) \leq \sup\limits_{t\in [0,T]} \|\beta_k(\rho_t)\|_{L^1(\R^d)},$$
which is concentrated on the set defined in \eqref{defn:gen-flow}  and satisfies
$$(e_t)_\# \eeta_k \res \R^d= \beta_k(\rho_t)\,\Leb{d} \qquad \mbox{for every }t\in [0,T].$$
Since $\sum_{k\geq 0}\beta_k(s)=s$, we immediately deduce that the measure $\eeta:=\sum_{k\geq 0}\eeta_k$ satisfies all the desired properties.

\smallskip

\noindent {\bf Step 4: representation via the Maximal Regular Flow.}
Under the additional assumption that $\bb$ is divergence-free and satisfies {\bf (A1)}-{\bf (A2)} of Section~\ref{sec:maxflow},
if $\rho_t \in L^\infty((0,T) \times \R^d)$ (resp. that $\rho_t$ is renormalized) then $\eeta$ (resp. every $\eeta_k$) is a regular generalized flow and by Theorem~\ref{thm:no-split-until-explosion} it is transported by the Maximal Regular Flow.
\end{proof}

\end{document}